\documentclass[final]{amsart}

\usepackage{mathrsfs}
\usepackage{amssymb}
\usepackage{xypic}
\usepackage{amsmath}
\usepackage{amsthm}
\usepackage{amsfonts}
\usepackage{amscd}
\usepackage{xspace}
\usepackage{bbold}

\makeatletter
\CheckCommand*\refstepcounter[1]{\stepcounter{#1}%
    \protected@edef\@currentlabel
       {\csname p@#1\endcsname\csname the#1\endcsname}%
}
\renewcommand*\refstepcounter[1]{\stepcounter{#1}%
  \protected@edef\@currentlabel
    {\csname p@#1\expandafter\endcsname\csname the#1\endcsname}%
}
\def\labelformat#1{\expandafter\def\csname p@#1\endcsname##1}
\DeclareRobustCommand\Ref[1]{\protected@edef\@tempa{\ref{#1}}%
   \expandafter\MakeUppercase\@tempa
}
\makeatother

\makeatletter
\newcommand{\numberlike}[2]{%
    \expandafter\def\csname c@#1\endcsname{%
        \expandafter\csname c@#2\endcsname}%
}
\makeatother

\def\DefaultNumberWithin{subsection}

\theoremstyle{plain}
\newtheorem{Lemma}{Lemma}
   \numberwithin{Lemma}{\DefaultNumberWithin}
   \labelformat{Lemma}{Lemma~#1}

   \numberwithin{Claim}{\DefaultNumberWithin}
   \numberlike{Claim}{Lemma}
   \labelformat{Claim}{Claim~#1}
\newtheorem{Theorem}{Theorem}
   \numberwithin{Theorem}{\DefaultNumberWithin}
   \numberlike{Theorem}{Lemma}
   \labelformat{Theorem}{Theorem~#1}
\newtheorem{Corollary}{Corollary}
   \numberwithin{Corollary}{\DefaultNumberWithin}
   \numberlike{Corollary}{Lemma}
   \labelformat{Corollary}{Corollary~#1}
\newtheorem{Proposition}{Proposition}
   \numberwithin{Proposition}{\DefaultNumberWithin}
   \numberlike{Proposition}{Lemma}
   \labelformat{Proposition}{Proposition~#1}

\theoremstyle{definition}
\newtheorem{Definition}{Definition}
   \numberwithin{Definition}{\DefaultNumberWithin}
   \numberlike{Definition}{Lemma}
   \labelformat{Definition}{Definition~#1}

\theoremstyle{remark}
\newtheorem{Remark}{Remark}
   \numberwithin{Remark}{\DefaultNumberWithin}
   \numberlike{Remark}{Lemma}
   \labelformat{Remark}{Remark~#1}
\newtheorem{Case}{Case}
   \labelformat{Case}{Case~#1}
   \numberwithin{Case}{Lemma}

   \labelformat{Step}{Step~#1}
   \numberwithin{Step}{Lemma}

\labelformat{equation}{(#1)}
\labelformat{figure}{Figure~#1}
\labelformat{section}{\S#1}
\labelformat{subsection}{\S#1}
\labelformat{subsubsection}{\S#1}

\let\RealQEDSymbol\qedsymbol
\newcommand{\MyQEDSymbol}{{\huge\ensuremath{\circ}}}
\theoremstyle{definition}
\newtheorem{MyExample}{Example}
   \numberwithin{MyExample}{\DefaultNumberWithin}
   \numberlike{MyExample}{Lemma}
   \labelformat{MyExample}{Example~#1}
\newenvironment{Example}[1][]%
  {\renewcommand{\qedsymbol}{\MyQEDSymbol}%
    \begin{MyExample}[#1]\pushQED{\qed}}%
  {\popQED\end{MyExample}%
    \renewcommand{\qedsymbol}{\RealQEDSymbol}}

\numberwithin{equation}{section}

\usepackage{pstricks}
\usepackage{ifthen}
\psset{unit=\unitlength}
\newpsobject{showgrid}{psgrid}{%
  gridcolor=gray,subgriddiv=1,griddots=10,gridlabels=6pt}
\newpsobject{border}{psframe}{%
  linewidth=0.02,linestyle=dashed,dash=2pt 2pt,
  framearc=0.3,fillstyle=none}
\psset{linewidth=0.02,linestyle=solid,linecolor=black}
\psset{fillstyle=none}
\psset{hatchcolor=black,hatchwidth=0.01,hatchsep=0.3}

\newcommand{\convexset}[1]{%
\pspolygon[linestyle=solid,linecolor=#1,fillstyle=hlines,linewidth=0.02,
hatchcolor=#1, hatchwidth=0.01, hatchsep=0.3]}

\newif\ifhrule\hrulefalse

\usepackage{pst-vue3d}

\newcommand{\defn}[1]{\textit{\textbf{#1}}}

\renewcommand{\O}{\varnothing} 
\newcommand{\acb}{(\alpha\circ\beta)}
\newcommand{\bca}{(\beta\circ\alpha)}

\newcommand{\Gx}[1]{\Gamma(#1)\cup\xi(#1)}
\newcommand{\sse}{\subseteq}
\newcommand{\sseq}{\sse}
\newcommand{\res}{\operatorname{res}}
\newcommand{\con}{\operatorname{con}}
\newcommand{\bottom}{\varepsilon} 

\newcommand{\mb}{\mathbb}

\newcommand{\m}{\backslash} 
\newcommand{\F}{\mathscr F} 
\newcommand{\OM}{\mathscr O} 
\newcommand{\OIG}{\mathcal G} 
\newcommand{\TG}{\mathsf T} 
\newcommand{\UOM}{\overline \OIG} 
\newcommand{\supp}{\operatorname{supp}} 
\newcommand{\conv}{\operatorname{conv}} 
\newcommand{\restd}{|_{\Gamma(\O)}} 
\newcommand{\SeparationSet}{S} 
\newcommand{\ext}{\operatorname{ext}} 
\newcommand{\uig}{upper interval greedoid\xspace}
\newcommand{\uigs}{upper interval greedoids\xspace}
\newcommand{\RH}{\Theta}
\newcommand{\hOIG}{\hat\OIG} 

\newcommand{\hzero}{\hat0}
\newcommand{\hone}{\hat1}
\newcommand{\pt}{\mathcal T}

\newcommand{\Flat}[1]{\widehat#1}
\renewcommand{\emptyset}{\O}

\newcommand{\Arrangement}{\mathcal A}
\newcommand{\IntersectionLattice}{\mathcal L}

\allowdisplaybreaks

\usepackage{setspace}
\begin{document}

\title{Oriented Interval Greedoids}
\author{Franco Saliola \and Hugh Thomas}
\email{saliola@gmail.com, hugh@math.unb.ca}

\begin{abstract}
We propose a definition of an \defn{oriented interval greedoid} that
simultaneously generalizes the notion of an oriented matroid and the
construction on antimatroids introduced by L.~J. Billera, S.~K. Hsiao,
and J.~S. Provan in \emph{Enumeration in convex geometries and
associated polytopal subdivisions of spheres} [Discrete Comput. Geom.
\textbf{39} (2008), no.~1-3, 123--137]. As for of oriented matroids,
associated to each oriented interval greedoid is a spherical
simplicial complex whose face enumeration depends only on the
underlying interval greedoid.
\end{abstract}

\maketitle

\begin{spacing}{0.6}
\tableofcontents
\end{spacing}

\section{Introduction}

Consider a hyperplane arrangement in $\mathbb R^n$, with all of the 
hyperplanes containing the origin.  Intersecting this arrangement 
with a sphere centred at
the origin, one obtains a regular cell decomposition of
the sphere.  Taking the barycentric subdivision, one obtains a spherical
simplicial complex.  

Oriented matroids are a generalization of real 
hyperplane arrangements; the Sphericity Theorem of Folkman and Lawrence
\cite{FolkmanLawrence1978} for oriented
matroids says that any oriented matroid induces a certain
regular 
cell decomposition of the sphere (and thus also a spherical simplicial
complex) just as hyperplane arrangements do.  (Terms not defined in
the introduction will be defined later in the paper.)

Billera, Hsiao, and Provan showed in \cite{BilleraHsiaoProvan2008} 
that there is also a certain spherical
simplicial complex
associated to a {\it convex geometry} or {\it antimatroid}.  
These simplicial complexes are not a 
special case of the spheres arising from oriented matroids, but they are
similar in some respects (see \ref{subsec:face} in particular).  

The goal of this paper is to provide a general theory which includes 
both of these as special cases.  Following a suggestion in 
\cite{BilleraHsiaoProvan2008} 
(attributed to Anders Bj\"orner), our approach
is via the notion of interval greedoid.  The precise definition appears
in the next section, but for now, it suffices to know that interval 
greedoids are a common generalization of matroids and antimatroids.  

In this paper, we define the notion of an oriented interval greedoid.  
This is an additional structure on top of the interval greedoid structure.
For a given interval greedoid, there may be no such additional
structure possible, or one,
or more than one.  

For an interval greedoid which is a matroid, an oriented interval
greedoid structure amounts to (the collection of covectors defining) 
an oriented matroid.  In contrast, 
if the underlying interval greedoid
is an antimatroid, it always admits exactly one oriented interval greedoid 
structure.  

Our main result is 
an analogue of the Sphericity Theorem for oriented interval greedoids,
providing a CW-sphere and (by barycentric subdivision) a 
spherical simplicial complex associated to any oriented interval greedoid.
Our proof is based on the proof of the Sphericity Theorem given in 
\cite{OrientedMatroids1993}.  
The spherical simplicial complex associated to the unique 
oriented structure for an 
antimatroid, coincides with that constructed by 
\cite{BilleraHsiaoProvan2008}.  

Along the way, we 
give versions for oriented interval greedoids of a number of 
constructions for oriented matroids, such as restriction and contraction.

\section{Interval Greedoids}

Much of the material in this section, except for \ref{sss:MapsMuAndXi}, is
drawn from \cite{BjornerZiegler1992:a} or \cite{KorteLovaszSchrader1991}.
The material in \ref{sss:MapsMuAndXi} and, by extension, the treatment in
\ref{sss:LatticeOfFlats}, seems to be new.

\subsection{Definition}
Let $E$ denote a finite set and $\F$ a set of subsets of $E$. An
\defn{interval greedoid} is a pair $(E,\F)$ satisfying the following
properties for all $X,Y,Z\in\F$:
\begin{itemize}
\item[(IG1)] 
If $X\neq\O$, then there exists an $x \in X$ such that $X - x \in \F$.
\item[(IG2)] 
If $|X| > |Y|$, then there exists an $x \in X \m Y$ such that $Y \cup
x \in \F$.
\item[(IG3)] 
If $X \subseteq Y \subseteq Z$ and $e \in E \m Z$ with $X \cup e \in
\F$ and $Z \cup e \in \F$, then $Y \cup e \in \F$.
\end{itemize}

The set $E$ is called the \defn{ground set} of the interval greedoid
$(E,\F)$. Elements of $\F$ are called the \defn{feasible sets} of
$(E,\F)$. If $\F$ is a nonempty collection of subsets of $E$
satisfying (IG1), then $\F$ is said to be an \defn{accessible set
system}. A \defn{greedoid} is a pair $(E,\F)$ that satisfies (IG1) and
(IG2). In the literature, (IG3) is often called the \defn{interval
property}.

In the next few sections we present several examples of interval
greedoids.

\subsection{Example: Matroids (\emph{Lower} Interval Greedoids)}
A \defn{matroid} is a pair $(E,\mathscr I)$ where $E$ is a finite set
and $\mathscr I$ is a collection of subsets of $E$ satisfying the
following two properties:
\begin{itemize}
\item[(M1)] 
If $X \in \mathscr I$ and $Y \subseteq X$, then $Y \in \mathscr I$.
\item[(M2)] 
For all $X,Y \in \mathscr I$, if $|X| > |Y|$, then there exists an $x
\in X \m Y$ such that $Y \cup x \in \mathscr I$.
\end{itemize}
Since (M2) is (IG2) and (M1) is a strengthing of (IG1) that implies the
interval property (IG3), a matroid
$(E,\mathscr I)$ is an interval greedoid.  Conversely, any greedoid
$(E,\F)$ satisying the following strengthening of (IG3) is a matroid.
\begin{itemize}
\item[(LIP)] Suppose $X,Y \in \F$ with $X \subseteq Y$. 
If $e \in E \m Y$ and $Y \cup e \in \F$, then $X \cup e \in \F$.
\end{itemize}
The above is called the \defn{interval property without lower bounds},
so a matroid is a \defn{lower interval greedoid}.

\begin{Example}[Vector matroids]\label{x:VectorMatroid}
Let $V=\mathbb R^2$, $\vec x=(-3,1)$, $\vec y=(2,1)$ and $\vec
z=(4,1)$. See \ref{f:VectorMatroid}. 
Let $\mathscr I$ be the collection of subsets of $E=\{\vec x,\vec
y,\vec z\}$ that consist of linearly independent vectors. That is,
\begin{gather*}
\mathscr I = 
\Big\{ \O, \{\vec x\}, \{\vec y\}, \{\vec z\},
\{\vec x,\vec y\}, \{\vec x,\vec z\}, \{\vec y,\vec z\}
\Big\}.
\end{gather*}
Then $(E,\mathscr I)$ is a matroid.
\begin{figure}[htb!]
\centering
\begin{pspicture}(3,2)(-0.5,0)
\psgrid[gridcolor=gray,subgriddiv=1,griddots=10,gridlabels=6pt](-3,0)(4,1)
\psline{->}(-3,1)
\psline{->}(2,1)
\psline{->}(4,1)
\uput[60](-3,1){$\vec x$}
\uput[60](2,1){$\vec y$}
\uput[60](4,1){$\vec z$}
\end{pspicture}
\caption{
}
\label{f:VectorMatroid}
\end{figure}
\end{Example}

\subsection{Example: Antimatroids (\emph{Upper} Interval Greedoids)}
\label{ss:UIG}

Another class of interval greedoids arise from convex geometries.

\subsubsection{Convex geometries}
\label{sss:ConvexGeometry}
Just as matroids can be viewed as an abstraction of linear
independence of vectors in $\mathbb R^n$, convex geometries can be
viewed as an abstraction of convexity of vectors in $\mathbb R^n$.  In
the following, think of $E$ as a finite subset of $\mathbb R^n$ and
$\tau$ as the convex hull operator: $\tau(A) = \operatorname{conv}(A)
\cap E$ for $A \subseteq E$.

A \defn{convex geometry} is a pair $(E,\tau)$, where $E$ is a finite
set and $\tau: 2^E \to 2^E$ is an increasing, monotone and idempotent
function, satisfying the following \defn{anti-exchange} axiom.
\begin{itemize}
\item[(AE)] If $x,y \not\in \tau(X), x \neq y$, and $y \in \tau(X \cup
x)$, then $x \not\in \tau(X \cup y)$.
\end{itemize}

The subsets $A \subseteq E$ satisfying $\tau(A) = A$ are called
\defn{closed sets} of the convex geometry. The \defn{extreme points}
$\ext(A)$ of $A \subseteq E$ are the points $x \in A$ satisfying $x
\not\in \tau(A \m x)$. The extreme points form a \emph{minimal
generating set} for the closed sets: 
if $X\subseteq E$ is a closed set, then 
$X = \tau(\ext(X))$,
and $\ext(X)\subseteq Y$ for all $Y\subseteq E$
satisfying $\tau(Y)=X$
(\cite[Proposition 8.7.2]{BjornerZiegler1992:a} or
\cite[Theorem III.1.1]{KorteLovaszSchrader1991}).

\begin{Example}[Convex geometries from convexity]
\label{x:ConvexGeometryFromConvexity}
The canonical example of a convex geometry is a finite subset $E
\subseteq \mathbb R^n$ with $\tau(A) = \operatorname{conv}(A) \cap E$,
where $\conv(A)$ is the convex hull of the points in
$A$. Then the extreme points of $A$ are precisely the extreme points
of the convex hull of $A$.
\end{Example}

\subsubsection{Antimatroids, or upper interval greedoids}
\label{sss:UIG}
If $(E,\tau)$ is a convex geometry, then the complements of the closed
sets of $E$
\begin{gather*}
\F = \{ E \m \tau(A) : A \subseteq E \}
\end{gather*}
are the feasible sets of an interval greedoid on the ground set $E$.
Moreover, $(E,\F)$ satisfies the following \defn{interval property
without upper bounds}.
\begin{itemize}
\item[(UIP)] Suppose $X,Y \in \F$ with $X \subseteq Y$. 
If $e \in E \m Y$ and $X \cup e \in \F$, then $Y \cup e \in \F$.
\end{itemize}

If $(E,\F)$ is a greedoid satisfying (UIP), then it is said to be an
\defn{\uig,} or an \defn{antimatroid}.
All \uigs arise from convex geometries:
if $(E,\F)$ is an \uig, then the complements of
the feasible sets are the closed sets of the convex geometry
$(E,\tau)$, where $\tau$ is defined for $X\subseteq E$ by
\begin{gather*}
\tau(X) = 
\bigcap_{X\subseteq Y\subseteq E \atop E\m Y \in \F} Y.
\end{gather*}
In other words, $\tau(X)$ is the smallest set in $\F^c=\{E\m Y:
Y\in\F\}$ containing $X$. 
For a proof of this result, see
\cite[Theorem III.1.3]{KorteLovaszSchrader1991} 
or
\cite[Proposition 8.7.3]{BjornerZiegler1992:a}.

\begin{Example}[Antimatroid from three colinear points]
\label{x:ThreeColinearPoints}
Let $x,y,z$ be three colinear points in the plane, $y$ between $x$ and $z$,
and consider the
convex geometry with closure operator $\tau(X)=\conv(X)\cap\{x,y,z\}$
(see \ref{x:ConvexGeometryFromConvexity}).
The closed sets are the subsets
\begin{gather*}
\O,\{x\},\{y\},\{z\},\{x,y\},\{y,z\},\{x,y,z\}.
\end{gather*}
Then $(E,\F)$ is an \uig,
where $E=\{x,y,z\}$ and $\F$ is
\begin{gather*}
\F = 
\Big\{ 
\{x,y,z\}, \{y,z\}, \{x,z\}, \{x,y\}, \{z\}, \{x\}, \O
\Big\}. \qedhere
\end{gather*}
\end{Example}

\begin{Remark}
Upper interval greedoids have been studied under several different
names including antimatroid, APS-structures, discs, and shelling
structures. Some care is required in reading the literature, as some
authors have used the term antimatroid for a convex geometry. 
By antimatroid, we will always mean an \uig.
\end{Remark}

\subsection{Example: Interval greedoids from semimodular lattices}

Let $L$ be a finite lattice. $L$ is said to be \defn{(lower)
semimodular} if it has the following property for all $x,y\in L$: if
$x \lessdot z$ and $y \lessdot z$ for some $z \in L$, then $x\wedge y
\lessdot x$ and $x\wedge y \lessdot y$. An element $e\in L$ such 
that $e\ne \hat 1$ is called
\defn{meet-irreducible} if $e = x \wedge y$ implies $x = e$ or $y =
e$. 

\begin{Proposition}
\label{p:IGsFromSemimodularLattices}
Suppose $L$ is a finite lower semimodular lattice.
Let $E$ be the set of \emph{meet-irreducible} elements of $L$,
and let 
\begin{align*}
\F = \big\{
\{e_1,e_2,\ldots,e_k\}\subseteq E:
\hat1 \gtrdot e_1 \gtrdot (e_1\wedge e_2)
\gtrdot\cdots\gtrdot (e_1\wedge e_2 \wedge \cdots \wedge e_k)
\big\}.
\end{align*}
Then $(E,\F)$ is an interval greedoid.
\end{Proposition}

For a proof of this result see 
\cite[Theorem 8.8.7]{BjornerZiegler1992:a}.

\subsection{Feasible Orderings}

Let $(E,\F)$ denote an interval greedoid. Let $X \in \F$. An ordering
$x_1 < x_2 < \cdots < x_r$ of the elements of $X = \{x_1, x_2, \ldots,
x_r\}$ is denoted by $X = \{x_1 < x_2 < \cdots < x_r\}$.  An ordering
$X = \{x_1 < x_2 < \cdots < x_r\}$ is a \defn{feasible ordering} of
$X$ if $\{x_1, \ldots, x_i\} \in \F$ for all $1 \leq i \leq r = |X|$.
Repeated application of (IG1) shows that every $X \in \F$ has a
feasible ordering. 

\begin{Proposition}
\label{p:StrongExchange}
Let $(E,\F)$ be an interval greedoid. Let $X,Y\in\F$ and $|X| > |Y|$. 
Suppose $X = \{x_1 < \cdots < x_r\}$ is a feasible ordering.
Then there is a subset $\{x_{i_1} < \cdots < x_{i_k}\}$ of $X \m Y$ of
size $|X|-|Y|$ such that
$Y \cup \{x_{i_1}, \ldots, x_{i_j}\} \in \F$ for all $1 \leq j \leq k$.
\end{Proposition}

\begin{proof}
Let $x_1 < \cdots < x_r$ be a feasible ordering of $X$ and suppose
$Y \in \F$ with $|Y| < |X|$. 
We proceed by induction on $|Y|$. If $|Y| = 0$, then the feasible ordering
$x_1 < \cdots < x_r$ of $X$ provides the required subset.

Suppose the result holds for all feasible sets of cardinality less than
$|Y|$. By (IG1), since $Y \in \F$, there is a $y \in Y$ such that $Y\m y
\in \F$. Since $|Y\m y| < |Y| < |X|$, the induction hypothesis gives the
existence of a subset $\{x_{i_1} < \cdots < x_{i_k}\}$ of $X$ of size
$k = |X| - (|Y| - 1)$ such that $Y\m y \cup \{x_{i_1}, \cdots,
x_{i_j}\} \in \F$ for all $1 \leq j \leq k$. 

Since $|Y| < |X| = |(Y \m y) \cup \{x_{i_1}, \ldots, x_{i_k}\}|$,
it follows from repeated application of (IG2) that there exist elements 
$z_j$ in $\{x_{i_1}, \ldots, x_{i_k}\}$ such that 
$Y \cup \{z_1\}, Y \cup
\{z_1,z_2\}, \ldots, Y \cup 
\{z_1,\ldots,z_{k-1}\}$ are in $\F$. Suppose
that for each $1 \leq l < k$ the element $z_l$ is chosen to be the first
element (with respect to the feasible ordering on $X$)
satisfying $(Y \cup \{z_1, \ldots, z_{l-1}\}) \cup z_l \in \F$.

Since $|Y \cup \{z_1, \ldots, z_{l-1}\}| < |(Y \m y) \cup \{x_{i_1}, \ldots,
x_{i_{l+1}}\}|$, it follows from (IG2) that there is an element 
$z \in \{x_{i_1}, \ldots, x_{i_{l+1}}\} \m \{z_1, \ldots, z_{l-1}\}$ 
such that $(Y \cup \{z_1, \ldots, z_{l-1}\}) \cup z \in \F$. 
The minimality of $z_{l}$ implies $z_l$ is amongst these elements. That is,
$z_{l} \in \{x_{i_1}, \ldots,
x_{i_{l+1}}\} \m \{z_1, \ldots, z_{l-1}\}$.

Let $a$ be the first index for which $z_a \neq x_{i_a}$. Then from the
last sentence in the previous paragraph, 
\begin{align*}
z_a 
&\in \{x_{i_1}, \ldots, x_{i_{a+1}}\} \m \{z_1, \ldots, z_{a-1}\} \\
&= \{x_{i_1}, \ldots, x_{i_{a+1}}\} \m \{x_{i_1}, \ldots, x_{i_{a-1}}\} \\
&= \{x_{i_a}, x_{i_{a+1}}\}. 
\end{align*}
Thus, $z_a = x_{i_{a+1}}$.

Suppose 
$z_b = x_{i_a}$ for some index $b$. Induction on $l$
gives $z_l = x_{i_{l+1}}$ for all $l$ such that $a < l < b$
since
$z_l \in \{x_{i_1}, \ldots, x_{i_{l+1}}\} \m \{z_1, \ldots, z_{l-1}\}
= \{x_{i_a}, x_{i_{l+1}}\}$
and $z_l \neq x_{i_a}$. 
Consider the following three sets:
$
((Y \m y) \cup \{z_1, \ldots, z_{a-1}\}) \subseteq
(Y \cup \{z_1, \ldots, z_{a-1}\}) 
\subseteq (Y \cup \{z_1, \ldots, z_{b-1}\}).
$
The first is $(Y \m y) \cup \{x_{i_1}, \ldots, x_{i_{a-1}}\}$, which
is in $\F$ by definition of the element $x_{i_l}$. 
The latter two sets are in $\F$
by definition of the elements $z_l$.
Substituting $z_l = x_{i_{l}}$ for $1 \leq l < a$ and $z_l = x_{i_{l+1}}$
for $a \leq l < b$ gives
\begin{align*}
\Big((Y \m y) \cup \{x_{i_1}, \ldots, x_{i_{a-1}}\}\Big) \subseteq &
\Big(Y \cup \{x_{i_1}, \ldots, x_{i_{a-1}}\}\Big) \\
& \subseteq
\Big(Y \cup \{x_{i_1}, \ldots, x_{i_{a-1}}, x_{i_{a+1}}, \ldots, x_{i_b}\}\Big).
\end{align*}
Applying the interval property (IG3) to the above sets and $x_{i_a}$
gives that $Y \cup \{z_1, \ldots, z_{a-1}, x_{i_{a}}\} = Y \cup \{x_{i_1}, \ldots, x_{i_{a}}\} \in \F$.
This contradicts the minimality of $z_a = x_{i_{a+1}}$. Therefore,
no such $b$ exists.

Therefore, $z_b \neq x_{i_a}$ for all $b > a$. Induction on $l$
(as above) gives $z_l = x_{i_{l+1}}$ for all $l$ such that
$a < l < k$.
Then 
$Y \cup \{x_{i_1}, \ldots, \widehat{x_{i_a}}, \ldots, x_{i_l}\} \in
\F$ for all $1 \leq l \leq k$ and the proposition holds.
\end{proof}

\subsection{The Lattice of Flats}

\subsubsection{Contractions}
Let $(E,\F)$ denote an interval greedoid and let $X\in\F$.
Let $\F/X$ denote the collection of subsets that can be added to
$X$ preserving feasibility:
\begin{align*}
 \F/X & = \{Y \subseteq E \m X : X \cup Y \in \F\}.
\end{align*}
The pair $(\bigcup_{Y\in\F/X}Y,\F/X)$ is an interval greedoid,
which we call the \defn{contraction} of $(E,\F)$ by $X$. Properties of
contractions will be further developed in later sections. For now we record
the following result, which is crucial to much of what follows.

\begin{Proposition}
\label{l:MaximalSubsetsAreEquivalent}
Suppose $(E,\F)$ is an interval greedoid and let $A \subseteq E$. Let
$U$ and $V$ be maximal with respect to inclusion among the feasible
sets contained in $A$. Then $\F/U=\F/V$.
\end{Proposition}
\begin{proof}
Let $U$ and $V$ be two maximal feasible subsets of $A$. Then $|U| =
|V|$ (otherwise we can enlarge the smaller one using (IG2)). Suppose
$W \in \F/U$ with $W \neq \O$.  Then $U \cup W \in \F$. Let $U
= \{u_1< \cdots< u_r\}$ be a feasible ordering of $U$. Repeated
application of (IG2) to $U$ and $U \cup W$ gives a feasible ordering
$\{u_1< \cdots< u_r< w_1 \cdots< w_s\}$ of $U \cup W$.
\ref{p:StrongExchange} applied to $U \cup W$ and $V$ gives an ordered
subset $\{z_1< \cdots< z_t\}$ of $U \cup W$ such that $V \cup \{z_1,
\ldots, z_i\} \in \F$ for each $1 \leq i \leq t$, where $t = |U \cup W|
- |V| = |W|$. If $z_1 \in U$, then $V \cup \{z_1\} \in \F$ and $V \cup
\{z_1\} \subseteq A$, contradicting the maximality of $V$. Therefore,
$z_1 \in W$ and the ordering of the $z_i$ implies $z_i \in W$ for all
$1 \leq i \leq t$. Since $t = |W|$, we have $W = \{z_1, \ldots,
z_t\}$. Thus, $V \cup W \in \F$, or equivalently, $W \in \F/V$.
Reversing the roles of $U$ and $V$ gives the reverse containment $\F/V
\subseteq \F/U$. Thus, $\F/U = \F/V$. 
\end{proof}

\begin{Example}[Convex geometry on three colinear points]
\label{x:FlatsOfConvexGeoemtryOn3ColinearPoints}
Consider the convex geometry on three colinear points from 
\ref{x:ThreeColinearPoints}. The feasible sets are
\begin{gather*}
\F = \Big\{ 
\O,
\{x\}, 
\{z\}, 
\{x,y\}, 
\{x,z\}, 
\{y,z\}, 
\{x,y,z\}
\Big\}.
\end{gather*}
The following table shows $\F/X$ for $X \in \F$.
\begin{gather*}
\begin{array}{c|c}
X & \F/X \\ \hline
\O        & \F\\ 
\{x\}     & \big\{ \O, \{y\}, \{z\}, \{y,z\} \big\} \\ 
\{z\}     & \big\{ \O, \{x\}, \{y\}, \{x,y\} \big\} \\ 
\{x,y\}   & \big\{ \O, \{z\} \big\} \\ 
\{x,z\}   & \big\{ \O, \{y\} \big\} \\ 
\{y,z\}   & \big\{ \O, \{x\} \big\} \\ 
\{x,y,z\} & \big\{ \O \big\}
\end{array}
\end{gather*}
From the table we notice that $\F/X = \F/Y$ implies $X=Y$. It turns
out this is true for any antimatroid; see \ref{x:AntimatroidFlats}.
\end{Example}

\subsubsection{Continuations}
Let $(E,\F)$ denote an interval greedoid and let $X\in\F$.
The set of \defn{continuations} $\Gamma(X)$ of $X$ is the set of
elements that can be added to $X$ preserving feasibility:
\begin{gather*}
\Gamma(X) = \{ x \in E \m X : X \cup x \in \F\}.
\end{gather*}

Of course, if $X,Y\in\F$ and $\F/X = \F/Y$, then
$\Gamma(X)=\Gamma(Y)$. The converse does not hold for arbitrary
greedoids, but it does hold for interval greedoids.

\begin{Proposition}
\label{p:SameContinuationsImpliesEquivalence}
Suppose $(E,\F)$ is an interval greedoid. Then for all $X, Y \in \F$, 
we have $\Gamma(X) = \Gamma(Y)$ if and only if $\F/X=\F/Y$.
\end{Proposition}
\begin{proof}
Suppose $\Gamma(X) = \Gamma(Y)$. We argue that $X$ is maximal among
the feasible sets contained in $X \cup Y$. If not, then there exists
$y \in Y$ such that $y \in \Gamma(X)$. Since $\Gamma(X) = \Gamma(Y)$,
we have $y \in \Gamma(Y)$, contradicting that $Y \cap \Gamma(Y) = \O$.
Therefore, $X$ is maximal among the feasible sets contained in $X \cup
Y$. Similarly, $Y$ is maximal among the feasible sets contained in $X
\cup Y$. Therefore, $\F/X=\F/Y$ by \ref{l:MaximalSubsetsAreEquivalent}. 
The reverse implication follows immediately from the definitions.
\end{proof}

\begin{Example}[Vector Matroids] Let $V$ be a vector space, and $E$ a 
collection of vectors in $V$.  $\mathscr I$ consists of the linearly
independent subsets of $E$. (See \ref{x:VectorMatroid}.)  Let $X\in 
\mathscr I$.  
Then $\Gamma(X)$ consists of those vectors from $E$ not in the span of 
$X$.  
\end{Example}

\begin{Example}[Antimatroids]
\label{sss:ContinuationsInAntimatroids}
\label{l:ExtremePointsAndContinuations}
Let $(E,\tau)$ be a convex geometry and $(E,\F)$ the corresponding
antimatroid. If $X\in\F$, then $\Gamma(X)=\ext(E\m X)$.
\end{Example}

\begin{Example}[Convex geometry on three colinear points]
The following table shows that continuations of the feasible sets of
the antimatroid in \ref{x:ThreeColinearPoints}.
\begin{gather*}
\begin{array}{c|ccccccc}
X & \O& \{x\} & \{z\} & \{x,y\} & \{x,z\} & \{y,z\} & \{x,y,z\} \\ \hline
\Gamma(X) & \{x,z\} & \{y,z\} & \{x,y\} & \{z\} & \{y\} & \{x\} & \O
\end{array}\qedhere
\end{gather*}
\end{Example}

\subsubsection{Flats}
Let $(E,\F)$ be an interval greedoid.
Define an equivalence relation on $\F$ by setting $X \sim Y$ if and
only if $\F/X = \F/Y$. In light of
\ref{p:SameContinuationsImpliesEquivalence}, $X\sim Y$ if and only if
$\Gamma(X)=\Gamma(Y)$. We write $[X]$ for the equivalence class of
$X$:
\begin{gather*}
[X] 
= \{Y\in\F: \F/X = \F/Y\}
= \{Y\in\F: \Gamma(X) = \Gamma(Y)\}.
\end{gather*}
These equivalence classes are called the \defn{flats} of $(E,\F)$.

The set $\Phi$ of flats of $(E,\F)$ is a poset with partial order
induced by reverse inclusion: 
\begin{gather*}
[X] \leq [Y] \text{ iff there exists } Z \in \F/Y 
\text{ such that } Y \cup Z \sim X.
\end{gather*}
In particular, if $Y \subseteq X$, then $[X] \leq [Y]$. (Note that
some authors choose to use inclusion rather than reverse-inclusion to
induce the partial order on $\Phi$.) 

The following result shows that $\Phi$ is a lower semimodular poset. In
fact, $\Phi$ is a semimodular \emph{lattice}; see
\ref{p:PhiIsSemimodularLattice}.

\begin{Proposition}
\label{c:Semimodularity}
Suppose $(E,\F)$ is an interval greedoid. Let $X \in \F$ and suppose
$X \cup x \in \F$ and $X \cup y \in \F$. If $[X\cup x] \neq [X \cup
y]$, then $X \cup \{x,y\} \in \F$.
\end{Proposition}
\begin{proof}
Suppose $(E,\F)$ is an interval greedoid and let $X \in \F$ with
$X \cup x \in \F$ and $X \cup y \in \F$.
If $X \cup \{x,y\} \not\in \F$, then $X \cup x$ and $X \cup y$ are
maximal among the feasible sets contained in $X \cup \{x,y\}$. Then
\ref{l:MaximalSubsetsAreEquivalent} implies $X \cup x \sim
X \cup y$. That is, $[X \cup x] = [X \cup y]$.
\end{proof}

\begin{Example}[Vector Matroids] \label{x:vmFlats}
Let $V$ be a vector space, $E$ a collection
of vectors from $V$, and $\mathscr I$ the subsets of $E$ that are 
linearly independent.  For $X,Y\in \mathscr I$, $X\sim Y$ iff $X$ and $Y$
span the same subspace; and $[X]\leq [Y]$ iff the span of $Y$ is contained in
the span of $X$.   
\end{Example}

\begin{Example}[Convex geometry on three colinear points]
\label{x:3PtsFlats}
Consider the convex geometry on three colinear points
(\ref{x:ThreeColinearPoints}). The contractions of the corresponding
antimatroid were described in
\ref{x:FlatsOfConvexGeoemtryOn3ColinearPoints}.
The poset of flats is illustrated in 
\ref{f:PosetFlatsOfThreeColinearPoints}.
\begin{figure}[htb]
\begin{gather*}
\xymatrix@R=1.5em@C=0.75em{%
& \big[\O\big] \ar@{-}[dl] \ar@{-}[dr] & \\
\big[\{x\}\big] \ar@{-}[d] \ar@{-}[dr] & & \big[\{z\}\big] \ar@{-}[d] \ar@{-}[dl] \\
\big[\{x,y\}\big] \ar@{-}[dr] & \big[\{x,z\}\big] \ar@{-}[d] &
\big[\{y,z\}\big] \ar@{-}[dl] \\
& \big[\{x,y,z\}\big] & 
}
\end{gather*}
\caption{The poset of flats of the convex geometry on three colinear
points.}
\label{f:PosetFlatsOfThreeColinearPoints}
\end{figure}
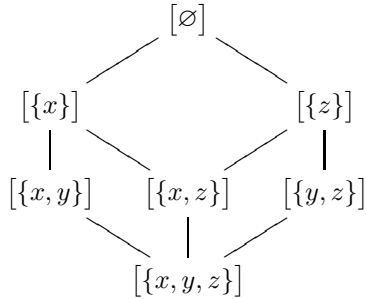
\end{Example}

\begin{Example}[Antimatroids]
\label{x:AntimatroidFlats}
If $(E,\F)$ is an antimatroid, then $[X]=\{X\}$ for all $X\in\F$.
Indeed, if $(E,\tau)$ is the corresponding convex geometry and
$\Gamma(X)=\Gamma(Y)$, then $E\m X=\tau(\ext(E\m X))=\tau(\ext(E\m
Y))=E\m Y$ by \ref{l:ExtremePointsAndContinuations}. Thus, the poset
of flats $\Phi$ of $(E,\F)$ is isomorphic to
$\F$ ordered by reverse inclusion.
\end{Example}

\begin{Example}[Semimodular lattices]\label{x:SemimodularFlats}
Let $L$ be a finite lower semimodular lattice.  Let $E$ be the 
meet-irreducible elements of $L$, and let $\F$ be the set system
as in \ref{p:IGsFromSemimodularLattices}.  The poset of 
flats of $(E,\F)$ is naturally isomorphic to $L$.

Consider the following map $\phi:\F\rightarrow L$:  
$$\phi(X)=\bigwedge_{e\in X} e.$$  It is constant on flats
of $(E,\F)$, and therefore descends to a map from $\Phi$ to $L$, which 
is a poset isomorphism.  See \cite[Theorem 8.8.7]{BjornerZiegler1992:a}.  
\end{Example}

\subsubsection{Maps $\mu$ and $\xi$}
\label{sss:MapsMuAndXi}

Let $(E,\F)$ be an interval greedoid and $\Phi$ its poset of flats.
Define two maps $\mu: 2^E \to \Phi$ and $\xi: \Phi \to 2^E$ as
follows.

\begin{enumerate}
\item[($\mu$)]
Define $\mu: 2^E \to \Phi$ on arbitrary subsets $A \subseteq E$ by
$\mu(A) = [X]$, where $X$ is maximal with respect to inclusion among
the feasible sets contained in $A$. 
\item[($\xi$)]
Define $\xi: \Phi \to 2^E$ for $X\in\F$ by $\xi([X]) = \bigcup_{X'
\sim X} X'$.
\end{enumerate}
It follows from \ref{l:MaximalSubsetsAreEquivalent} that $\mu$ is
well-defined. These maps are very important to what follows, and will
be used to describe the meets and joins in $\Phi$.

\begin{Proposition}
\label{p:MuAndXiProperties}
Suppose $(E,\F)$ is an interval greedoid, $\Phi$ its lattice of flats
and $\mu$ and $\xi$ the maps defined above. 
\begin{enumerate}
\item 
$(\mu \circ \xi)(A) = A$ for all $A\in\Phi$. So $\xi$ is injective.
\item 
$\mu:(2^E,\subseteq)\to(\Phi,\leq)$ is order-reversing.
\item 
$\xi:(\Phi,\leq)\to(2^E,\subseteq)$ is order-reversing.
\item 
$A \leq B$ if and only if $\xi(B) \subseteq \xi(A)$ for
all $A,B\in\Phi$.
\item 
For all $Y\in\F$ and $A\in\Phi$, if $Y \subseteq \xi(A)$, then
$A \leq [Y]$.
\end{enumerate}
\end{Proposition}

\begin{proof}
(1) Suppose that $X$ is not maximal with respect to inclusion among
the feasible sets contained in $\xi([X])$. Then there exists $x \in
\xi([X]) - X$ such that $X \cup x \in \F$. Therefore, $x \in \F/X$ and
$x \in X'$ for some $X' \sim X$ with $X' \neq X$. But $X' \sim X$ if
and only if
$\F/X = \F/X'$, so $x \in \F/X'$. This is a contradiction since $x
\not\in \F/X'$ if $x \in X'$. Thus, $X$ is maximal, and so
$\mu(\xi([X])) = [X]$. 

(2) Suppose $A \subseteq B$. Let $X$ be maximal with respect to
inclusion among the feasible sets contained in $A$. Then there exists
$Y$ such that $X \subseteq Y \subseteq B$ and $Y$ is maximal among the
feasible sets contained in $B$. Therefore, $[Y] \leq [X]$. Hence,
$\mu(B) \leq \mu(A)$.

(3) Suppose $[X] \leq [Y]$. If $e \in \xi([Y])$, then $e \in Y'$ for
some $Y' \sim Y$. So $[X] \leq [Y] = [Y']$. Thus, there is a $Z \in
\F/Y'$ such that $Y' \cup Z \sim X$. Therefore, $e \in Y' \subseteq
(Y' \cup Z) \subseteq \xi([X])$. 

(4) This follows from (1), (2) and (3).

(5) Suppose $Y \in \F$ and $Y \subseteq \xi([X])$. Then there exists
$Z$ containing $Y$ that is maximal among the feasible sets contained
in $\xi([X])$. Then $[Y] \geq [Z]$ since $Y \subseteq Z$ and $[Z] =
[X]$ by \ref{l:MaximalSubsetsAreEquivalent}.
\end{proof}

\begin{Remark}
Let $(E,\F)$ be an interval greedoid.  For $X\subset E$, the \defn{rank} of 
$X$ is the size of a maximal feasible set contained in $X$.  $X$ is closed
if any proper superset of $X$ has larger rank than $X$ does.  The
\defn{closure} of $X$ is the smallest closed set containing $X$.  
(The uniqueness here follows from (IG2).)

If $(E,\F)$ is a matroid without loops, then 
$(\xi\circ\mu)(A)$ is the closure of 
$A$ (see \ref{x:MatroidXi} below).  In general, though, all we can say is
that $(\xi\circ\mu)(A)$ is contained in the closure of $A$.  The
containment follows
from the fact that $(\mu\circ\xi\circ\mu)(A)=\mu(A)$ by
\ref{p:MuAndXiProperties}(1).  The fact that the containment is not necessarily
an equality is shown in the following example.  

Consider the convex
geometry on the three colinear points $x,y,z$ of
\ref{x:ThreeColinearPoints}. The empty set is feasible in the
corresponding antimatroid and we have $\xi(\O)=\O$. But the closure of
$\O$ is $\{y\}$ since the latter is not a feasible set (because
$\{x,z\}$ is not a closed set in the convex geometry).
\end{Remark}

\begin{Example}[Matroids]\label{x:MatroidXi}
Let $(E,\F)$ be a matroid without loops.  In this case $\xi(\Phi)$ 
consists exactly of the closed sets of the matroid.  

Let $A\subset E$.  
As already remarked, \ref{p:MuAndXiProperties}(1)
implies that $(\xi\circ\mu)(A)$ 
is contained in the closure of $A$.  Conversely,
suppose that $e$ is in the closure of $A$.  Since $e$ is not a loop,
$\{e\}$ is feasible, and can therefore be extended to a maximal feasible
set $X$ in $A\cup\{e\}$.  Let $Y$ be a maximal feasible set in $A$.  
Since $e$ is in the closure of $A$, we have that $Y$ is also a 
maximal feasible set inside $A\cup \{e\}$, and thus $X\sim Y$ by
\ref{l:MaximalSubsetsAreEquivalent}.  It follows that $e\in Y
\subset (\xi\circ\mu)(A)$.  Thus $(\xi\circ\mu)(A)$ equals the closure of
$A$.  
\end{Example}

\begin{Example}[Antimatroids]  Let $(E,\F)$ be an antimatroid.  Let
$X$ be feasible.  Since $[X]=\{X\}$, $\xi([X])=X$.  Thus 
$\xi(\Phi)$ consists precisely of the feasible sets.  
\end{Example}

\begin{Example}[Semimodular lattices]\label{x:SemimodularXi}
Let $L$ be a lower semimodular lattice,
and $(E,\F)$ the associated interval greedoid.  
Let $\phi$ be the isomorphism from $\Phi$ 
to $L$, defined in \ref{x:SemimodularFlats}.
Let $X$ be a feasible set.    
Then $\xi([X])$ consists
of the set of meet-irreducibles $f$ such that $f\geq \phi([X])$.  

If $f$ is in a feasible set $Y\sim X$, then $\phi([X])=\phi([Y])\leq f$, which
proves one containment.  For the other direction, let $f\geq \phi([X])$.  
Let $Z$ be a feasible set with $\phi([Z])=f$.  Since $f$ is 
meet-irreducible, $f\in Z$.  Since $f\geq \phi([X])$, we know
$[Z]\geq [X]$, which implies that $f\in Z\subset \xi([X])$, as desired.  
\end{Example}

\subsubsection{Lattice of flats}
\label{sss:LatticeOfFlats}

We have seen that $\Phi$ is a lower semimodular poset. It is also graded:
the corank of any element $A \in \Phi$ is the size of any feasible set in
$A$. The next result establishes that $\Phi$ is also a lattice.

\begin{Proposition}
\label{p:PhiIsSemimodularLattice}
If $(E,\F)$ is an interval greedoid, then $\Phi$ is a lower
semimodular lattice whose lattice operations are given by:
\begin{align*}
A\vee B = \mu\big(\xi(A)\cap\xi(B)\big) 
\quad\text{ and }\quad
A\wedge B = \mu\big(\xi(A) \cup \xi(B)\big)
\end{align*}
for all $A, B \in \Phi$.
That is, $A \vee B = [X]$, where $X$ is maximal among the feasible
sets contained in $\xi(A) \cap \xi(B)$, and $A \wedge B = [X]$, where
$X$ is maximal among the feasible sets contained in $\xi(A)\cup\xi(B)$.
\end{Proposition}
\begin{proof}
By \ref{c:Semimodularity}, $\Phi$ is a lower semimodular poset. It
remains to show that $\Phi$ is a lattice. For $A,B \in \Phi$, define
$j(A,B) = [X]$, where $X \in \F$ is maximal among the feasible sets
contained in $\xi(A) \cap \xi(B)$. \ref{l:MaximalSubsetsAreEquivalent}
implies that $j(A,B)$ is well-defined. (Equivalently, $j(A, B) =
\mu(\xi(A) \cap \xi(B))$.) Since $X \subseteq \xi(A)$, if follows from
\ref{p:MuAndXiProperties} that $A \leq [X] = j(A,B)$. Similarly, $B
\leq j(A,B)$. Therefore, $j(A,B)$ is an upper bound of $A$ and $B$. 

It remains to show that $j(A,B)$ is the least upper bound. Suppose $A,
B \leq [Y]$. Then $\xi([Y]) \subseteq \xi(A)$ and $\xi([Y]) \subseteq
\xi(B)$ since $\xi$ is order-reversing (\ref{p:MuAndXiProperties}).
Therefore, $\xi([Y]) \subseteq \xi(A) \cap \xi(B)$. So there exists
$X' \in \F$ such that $Y \subseteq X' \subseteq \xi(A) \cap \xi(B)$
and $X'$ is maximal among the feasible sets contained in $\xi(A) \cap
\xi(B)$.  Therefore, by the maximality of $X'$ and since $Y \subseteq
X'$, we have $j(A,B) = [X'] \leq [Y]$.

For $A, B \in \Phi$, let $m(A,B) = [X]$, where $X$ is maximal among
the feasible sets contained in $\xi(A) \cup \xi(B)$. Let $A = [Y]$.
Then $Y\subseteq\xi(A)$, so there exists $X'\supseteq Y$ such that
$X'$ is maximal among the feasible sets contained in
$\xi(A)\cup\xi(B)$. Therefore, $m(A,B) = [X'] \leq [Y] = A$.
Similarly, $m(A,B) \leq B$. 

It remains to show that $m(A,B)$ is the greatest lower bound. Suppose
$C \leq A$ and $C \leq B$. Then $\xi(A) \cup \xi(B) \subseteq \xi(C)$.
So there exists a subset $X'\supseteq X$ that is maximal among the
feasible sets contained in $\xi(C)$. Thus, $m(A,B) = [X] \geq [X'] = C$.
\end{proof}

\begin{Example}[Antimatroids]
\label{x:AntimatroidMeets}
Let $(E,\tau)$ be a convex geometry and $(E,\F)$ the corresponding
antimatroid. If $X,Y\in\F$, then $[X]\vee[Y] = [U]$, where $U$ is
maximal among the feasible sets contained in $X\cap Y$. By
\ref{x:AntimatroidFlats}, $U$ is unique and it follows that $U$ is the
complement of the closure of $(E\m X)\cup(E\m Y)$. Hence,
\begin{align*}
[X]\vee[Y]=\left[E\m \tau\Big( (E\m X) \cup (E\m Y)\Big)\right]
\end{align*}
for all $X,Y\in\F$.
\end{Example}

\section{Oriented Interval Greedoids}

Throughout this section $(E,\F)$ will denote an interval greedoid.

\subsection{Signed flats}

A \defn{signed flat} of an interval greedoid $(E,\F)$ is a pair $(A,
\Flat\alpha)$ consisting of a flat $A$ and a map $\Flat\alpha: \Gamma(A) \to
\{+,-\}$. 

Define a partial order on signed flats as follows. 
If $(A,\Flat\alpha)$ and $(B,\Flat\beta)$ are signed flats of $(E,\F)$, let
$(A,\Flat\alpha) \leq (B,\Flat\beta)$ if $A \leq B$ (as flats in $\Phi$) and if
$\Flat\alpha$ and $\Flat\beta$ agree on $\Gamma(A) \cap \Gamma(B)$. (Reflexivity
and anti-symmetry are straightforward to verify; transitivity follows
by a simple application of (IG3).)

Define the product $(A,\Flat\alpha) \circ (B,\Flat\beta)$ of two signed flats
$(A,\Flat\alpha)$ and $(B,\Flat\beta)$ by
\begin{align*}
 (A, \Flat\alpha) \circ (B, \Flat\beta) = (A \vee B, \Flat\alpha \circ \Flat\beta),
\end{align*}
where, for $x \in \Gamma(A \vee B)$,
\begin{align*}
 (\Flat\alpha \circ \Flat\beta)(x) =
 \begin{cases}
  \Flat\alpha(x), & \text{if } x \in \Gamma(A), \\
  \Flat\beta(x), & \text{otherwise}.
 \end{cases}
\end{align*}
This product is well-defined because 
$\Gamma(A \vee B) \subseteq \Gamma(A) \cup \Gamma(B)$
(\ref{p:PropertiesOfContinuations} below).

\begin{Example}[{Antimatroids}]
\label{x:SignedFlatProductForAntimatroids}
Suppose $(E,\F)$ is an antimatroid. As we saw in
\ref{sss:ContinuationsInAntimatroids}, the continuations of a feasible
set $X$ are the extreme points of the complement $E\m X$ in the convex
geometry. Therefore, a signed flat $([X],\Flat\alpha)$ of the antimatroid
is an assignment of $+$ or $-$ to each extreme point of $E\m X$.
\ref{f:TwoCovectorsInAConvexGeometry} depicts a closed set $C$ of a
convex geometry; the extreme points of $C$ are labelled by $+$ or $-$,
the non-extreme points in $C$ are labelled by $1$, and the points in
the exterior of $C$ are labelled by $0$.
\begin{figure}[!ht]
\centering
\begin{pspicture}(10,7.5)(0,-0.5)
\border(-0.75,-0.75)(10.5,7.5)
\convexset{blue}(6,0)(10,4)(4,7)
\uput[0](6,0){$+$}
\uput[0](7,1){$1$}
\uput[90](10,4){$+$}
\uput[0](4,7){$-$}
\uput[180](5,3.5){$1$}
\uput[-40](6.25,5){$1$}
\uput[0](7.5,3){$1$}
\uput[0](0,0){$0$}
\uput[0](2,0){$0$}
\uput[0](3,0.25){$0$}
\uput[45](1,1){$0$}
\uput[0](3.5,2.5){$0$}
\uput[0](0.5,6){$0$}
\psdot(0,0)\psdot(2,0)\psdot(3,0.25)\psdot(1,1)
\psdot(6,0)
\psdot(7,1)\psdot(10,4)\psdot(4,7)\psdot(5,3.5)\psdot(6.25,5)\psdot(7.5,3)
\psdot(3.5,2.5)
\psdot(0.5,6)
\uput[-90](5,-0.75){$\alpha$}
\end{pspicture}
\qquad
\begin{pspicture}(10,7.5)(0,-0.5)
\border(-0.75,-0.75)(10.5,7.5)
\convexset{red}(0,0)(6,0)(1,1)
\psdot(0,0)\psdot(2,0)\psdot(3,0.25)\psdot(1,1)
\psdot(6,0)
\psdot(7,1)\psdot(10,4)\psdot(4,7)\psdot(5,3.5)\psdot(6.25,5)\psdot(7.5,3)
\psdot(3.5,2.5)
\psdot(0.5,6)
\uput[90](0,0){$-$}
\uput[150](2,0){$1$}
\uput[55](3,0.25){$1$}
\uput[45](1,1){$+$}
\uput[0](6,0){$-$}
\uput[0](7,1){$0$}
\uput[90](10,4){$0$}
\uput[0](4,7){$0$}
\uput[180](5,3.5){$0$}
\uput[-40](6.25,5){$0$}
\uput[0](7.5,3){$0$}
\uput[0](3.5,2.5){$0$}
\uput[0](0.5,6){$0$}
\uput[-90](5,-0.75){$\beta$}
\end{pspicture}
\caption{Two covectors $\alpha$ and $\beta$ of an antimatroid.}
\label{f:TwoCovectorsInAConvexGeometry}
\end{figure}
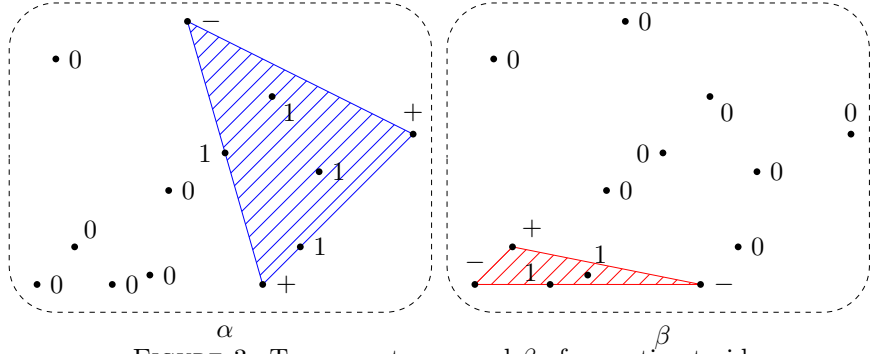

The product of two signed flats $([X],\Flat\alpha)$ and $([Y],\Flat\beta)$ has
a geometric interpretation. If $X'=E\m X$ and $Y'=E\m Y$, then form a
new closed set $Z'$ by taking the closure of $X'\cup Y'$; that is,
$Z'=\tau(X'\cup Y')$. Note that the extreme points of $Z'$ are
contained in $\ext(X')\cup\ext(Y')$. The sign for each $z\in\ext(Z')$
is $\Flat\alpha(z)$ if $z\in\ext(X')$, and $\Flat\beta(z)$ otherwise.
\end{Example}

\begin{Example}[{Matroids}]
Suppose $(E,\F)$ is a matroid with no loops. If $A$ is a flat of the
matroid, then $\xi(A) = E \m \Gamma(A)$ is a closed set of the matroid.
Therefore, a signed flat $(A,\Flat\alpha)$ is an assignment of a sign $+$
or $-$ to each element of the complement of the closed set $A$. If we
extend this by assigning $0$ to each element of $A$, then $\Flat\alpha$
induces a \emph{covector} in the sense of oriented matroids. (See
\ref{sss:OrientedMatroids}.)
\end{Example}

Among other things, the following establishes that the product of
signed flats is well-defined.

\begin{Proposition}
\label{p:PropertiesOfContinuations}
Let $(E,\F)$ be an interval greedoid, $A, B \in \Phi$ and $X \in \F$.
\begin{enumerate}
\item 
If $B \leq [X]$ and $x \in \Gamma(X)$, then either 
$x \in \Gamma(B)$ or $B \leq [X \cup x]$.
\item 
If $A \leq B$, then $\Gamma(B) \subseteq \Gamma(A) \cup \xi(A)$.
\item
\label{i:GammaOfAvB}
$\Gamma(A \vee B) \subseteq \Gamma(A) \cup \Gamma(B)$.
\item 
$\Gamma(A \vee B) \cup \xi(A \vee B) \subseteq 
       (\Gamma(A) \cup \xi(A)) \cap (\Gamma(B) \cup \xi(B))$.
\end{enumerate}
\end{Proposition}

\begin{proof}
(1) Pick $Y \in \F$ such that $B = [Y]$. Suppose $[Y] \leq [X]$ and
let $x \in \Gamma(X)$. Then there exists $Z \in \F/X$ such that $X
\cup Z \sim Y$.  Applying axiom (IG2) repeatedly to $X$ and $X \cup Z$
yields a sequence $X \subset (X \cup z_1) \subset (X \cup \{z_1,z_2\})
\subset \cdots \subset (X \cup Z)$ of feasible sets. Put $X_0 = X$ and
let $X_i = X_{i-1} \cup z_i$ for $1 \leq i \leq r = |Z|$. 
\begin{gather*}
\xymatrix@R=0pt{
& X \cup \{z_1\} \ar@{^{(}->}[r] 
& X \cup \{z_1,z_2\} \ar@{^{(}->}[r] 
& \cdots \ar@{^{(}->}[r]
& X \cup Z \sim Y \\
X \ar@{^{(}->}[ur] \ar@{^{(}->}[dr] \\
& X \cup \{x\} \ar@{^{(}->}[r] 
& X \cup \{z_1, x\} \ar@{^{(}->}[r] 
& \cdots \ar@{^{(}->}[r]
& X \cup Z \cup \{x\}
}
\end{gather*}
Since $x \in \Gamma(X)$, $X \cup x \in \F$. If $[X \cup x] = [X \cup
z_1]$, then $[Y] = [X \cup Z] \leq [X \cup z_1] = [X \cup x]$ since $X
\subseteq X \cup Z$. If $[X \cup x] \neq [X \cup z_1]$, then $X \cup
\{x,z_1\} \in \F$ by \ref{c:Semimodularity}
since $x, z_1 \in \Gamma(X)$.  Thus, $x, z_2 \in \Gamma(X \cup z_1)$.
If $[X \cup \{z_1,x\}] = [X \cup \{z_1,z_2\}]$, then $[Y] \leq [X \cup
x]$ using a similar argument as in the previous case. If $[X \cup
\{z_1,x\}] \neq [X \cup \{z_1,z_2\}]$, then $X \cup \{z_1, z_2, x\}
\in \F$ by \ref{c:Semimodularity}. Continuing
in this manner we get either that $[Y] \leq [X \cup x]$ or $(X \cup Z)
\cup x \in \F$. That is, either $B = [Y] \leq [X \cup x]$, or $x \in
\Gamma(Y) = \Gamma(B)$. This proves the statement.

(2) Pick $X,Y \in \F$ such that $A = [Y]$ and $B = [X]$. If $[Y] \leq
[X]$ and $x \in \Gamma(Y)$, then $x \in \Gamma(Y)$ or $[Y] \leq [X
\cup x]$ by (1). In the latter case, $x \in \xi(X \cup x) \subseteq
\xi(Y)$ since $\xi$ is order-reversing. Thus, $x \in \Gamma(Y)$ or $x
\in \xi(Y)$.

(3) Pick $X \in \F$ such that $A \vee B = [X]$. Let $x \in \Gamma(X)$.
Then $X \cup x \in \F$ and $[X \cup x] < [X]$. Since $[X] = A \vee B$,
it follows that $[X \cup x]$ is not above both $A$ and $B$. If $A
\not\leq [X \cup x]$, then the above applied to $A \vee B$ and $A$
gives that $x \in \Gamma(A)$ since $A \not\leq [X \cup x]$. Similarly,
if $B \not\leq [X \cup x]$, then $x \in \Gamma(B)$. Hence, $x \in
\Gamma(A) \cup \Gamma(B)$.

(4) Since $\xi$ is order-reversing, it follows that $\xi(A \vee B)
\subseteq \xi(A) \cap \xi(B)$. (4) now follows from (2).
\end{proof}

The next result collects some properties of the product and
partial order of signed flats.

\begin{Proposition}
\label{p:ProductAndPartialOrder}
Let $(A, \Flat\alpha)$ and $(B, \Flat\beta)$ denote two signed flats over an
interval greedoid $(E, \F)$. Then
\begin{enumerate}
\item
$(A,\Flat\alpha) \leq (B,\Flat\beta)$ if and only if $(A,\Flat\alpha) \circ
(B,\Flat\beta) =
(B,\Flat\beta)$.
\item 
$(A,\Flat\alpha) \leq (A,\Flat\alpha) \circ (B,\Flat\beta)$.
\item
If $A \leq B$, then $(B,\Flat\beta) \circ (A,\Flat\alpha) = (B,\Flat\beta)$.
\item
The product $\circ$ is associative.
\item
$(A,\Flat\alpha) \circ (B, \Flat\beta) \circ (A,\Flat\alpha) = (A,\Flat\alpha) \circ
(B,\Flat\beta).$
\item
$(A,\Flat\alpha) \circ (A,\Flat\alpha) = (A,\Flat\alpha).$
\end{enumerate}
\end{Proposition}

\begin{proof}
(1) Suppose $(A,\Flat\alpha)\circ(B,\Flat\beta) = (B,\Flat\beta)$. Since $(A,\Flat\alpha)
\circ (B,\Flat\beta) = (A \vee B, \Flat\alpha \circ \Flat\beta)$, it follows that $B
= A \vee B$ and that $\Flat\beta = \Flat\alpha \circ \Flat\beta$. Therefore, $A \leq
B$ and $\Flat\beta(x) = (\Flat\alpha \circ \Flat\beta)(x) = \Flat\alpha(x)$ for all $x \in
\Gamma(A) \cap \Gamma(B)$. Thus, $(A,\Flat\alpha) \leq (B,\Flat\beta)$.

Conversely, suppose $(A,\Flat\alpha) \leq (B,\Flat\beta)$. Then $A \leq B$ and
$\Flat\alpha(x) = \Flat\beta(x)$ for all $x \in \Gamma(A) \cap \Gamma(B)$.
Therefore, $A \vee B = B$. It remains to show that $(\Flat\alpha \circ
\Flat\beta)(x) = \Flat\beta(x)$ for all $x \in \Gamma(A \vee B) = \Gamma(B)$.
Let $x \in \Gamma(B)$. If $x \in \Gamma(A)$, then $(\Flat\alpha \circ
\Flat\beta)(x) = \Flat\alpha(x)$ and $\Flat\alpha(x) = \Flat\beta(x)$ since $\Flat\alpha$ and
$\Flat\beta$ agree on $\Gamma(A) \cap \Gamma(B)$. If $x \not\in \Gamma(A)$,
then $(\Flat\alpha\circ \Flat\beta)(x) = \Flat\beta(x)$. Therefore, $\Flat\beta$ and
$\Flat\alpha \circ \Flat\beta$ agree on $\Gamma(A \vee B)$. 

(2) First note that $A \leq A \vee B$ by the definition of $\vee$.  We
need only show that $\Flat\alpha$ and $\Flat\alpha \circ \Flat\beta$ agree on
$\Gamma(A) \cap \Gamma(A \vee B)$, which follows from the definition of
$\Flat\alpha\circ\Flat\beta$.  
Therefore, $(A, \Flat\alpha) \leq (A \vee B, \Flat\alpha
\circ \Flat\beta) = (A,\Flat\alpha)\circ(B,\Flat\beta)$.  

(3) If $A \leq B$, then $A \vee B = B$, so $\Gamma(A \vee B) =
\Gamma(B)$.  So the domains of $\Flat\beta \circ \Flat\alpha$ and $\Flat\beta$ are
the same.  And from the definition of $\circ$, if $x \in \Gamma(B)$,
then $(\Flat\beta \circ \Flat\alpha)(x) = \Flat\beta(x)$.

(4), (5) and (6) are straightforward to verify using similar arguments.
\end{proof}

\subsection{Covectors}

Let $(A,\Flat\alpha)$ denote a signed flat. Then $\Flat\alpha: \Gamma(A) \to
\{+,-\}$ can be extended to a map $\alpha: E \to \{0,+,-,1\}$
as follows,
\begin{align*}
 \alpha(e) =
 \begin{cases}
 \Flat\alpha(e), & \text{if } e \in \Gamma(A), \\
 \hfill 0, & \text{if } e \in \xi(A), \\
 \hfill 1, & \text{otherwise}.
 \end{cases}
\end{align*}
This map $\alpha$ is called the \defn{covector} of the signed flat $(A,\Flat\alpha)$.

\begin{Example}[Antimatroids]
\label{x:AntimatriodCovectors}
Let $E$ be a finite subset of $\mathbb R^n$ and $\tau(X)=\conv(X)\cap
E$. Let $(E,\F)$ denote the corresponding \uig.
Suppose $(A,\Flat\alpha)$ is a signed flat of $(E,\F)$ and let $X\in\F$
with $A=[X]$. Then the covector $\alpha$ of the signed flat is
obtained by assigning $0$ to the points in the exterior of $E\m X$,
$\Flat\alpha(x)$ to the points $e\in\ext(E\m X)$, and $1$ to the
non-extreme points contained in $E\m X$.
See \ref{f:TwoCovectorsInAConvexGeometry} for an example.
\end{Example}

Note that a signed flat $(A,\Flat\alpha)$ can be recovered from its 
covector $\alpha$.
Indeed, the set of
indices $e \in E$ such that $\alpha(e) = 0$ is precisely the set
$\xi(A)$, from which $A$ can be recovered (\ref{p:MuAndXiProperties}).
Therefore, there exists a map from the set
of covectors of $(E,\F)$ to the lattice of flats $\Phi$,
\begin{gather*}
 \supp(\alpha) = \mu\left( \{x \in E : \alpha(x) = 0\} \right).
\end{gather*}

The product on signed flats can be formulated for covectors as
follows. Let $\alpha,\beta: E \to \{0,+,-,1\}$ be the covectors of
the signed flats $(A,\Flat\alpha)$, $(B,\Flat\beta)$, respectively.  Define a
partial order on the symbols $0,+,-,1$ according to the following
Hasse diagram.
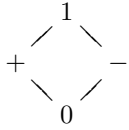
\begin{align*}
\xymatrix@=0.7em{
& 1 \ar@{-}[dr] \ar@{-}[dl] & \\
+ \ar@{-}[dr] & & - \ar@{-}[dl]  \\
& 0 & 
}
\end{align*}
Define
\begin{align*}
(\alpha \star \beta)(e) =
\begin{cases}
\beta(e), & \text{if } \beta(e) > \alpha(e), \\
\alpha(e), & \text{otherwise}
\end{cases}
\end{align*}
and 
\begin{align*}
(\alpha \circ \beta)(e) =
\begin{cases}
(\alpha \star \beta)(e), & \text{if } e \in 
\Gamma(A \vee B) \cup \xi(A \vee B), \\
\hfill 1, & \text{otherwise}.
\end{cases}
\end{align*}

\begin{Example}[Covector multiplication in an antimatroid]
Let $\alpha$ and $\beta$ be covectors of $(E,\F)$ from
\ref{x:AntimatriodCovectors} and $X'$ and $Y'$ their underlying closed
sets.
%
%
The covector $\alpha\circ\beta$ is obtained as follows. Let $Z' =
\tau(X'\cup Y')$. Then $(\alpha\circ\beta)(z)$ is $0$ if $z$ is in
the exterior of $Z'$, $1$ if $z$ is a non-extreme point of $Z'$,
$\alpha(z)$ if $z\in\ext(X')$, and $\beta(z)$ otherwise. 
\ref{f:ProductOfCovectorsInAConvexGeometry} depicts the product of the
covectors from \ref{f:TwoCovectorsInAConvexGeometry}.
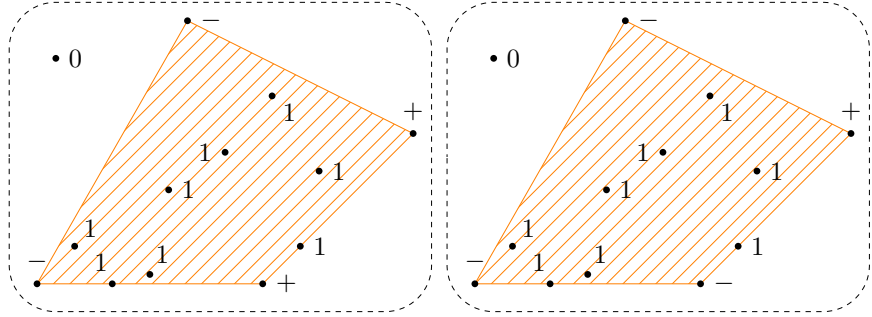
\begin{figure}[!ht]
\centering
\begin{pspicture}(10,7)
\border(-0.75,-0.75)(10.5,7.5)
\convexset{orange}(6,0)(10,4)(4,7)(0,0)
\uput[0](6,0){$+$}
\uput[0](7,1){$1$}
\uput[90](10,4){$+$}
\uput[0](4,7){$-$}
\uput[180](5,3.5){$1$}
\uput[-40](6.25,5){$1$}
\uput[0](7.5,3){$1$}
\uput[90](0,0){$-$}
\uput[120](2,0){$1$}
\uput[55](3,0.25){$1$}
\uput[45](1,1){$1$}
\uput[0](3.5,2.5){$1$}
\uput[0](0.5,6){$0$}
\psdot(0,0)\psdot(2,0)\psdot(3,0.25)\psdot(1,1)
\psdot(6,0)
\psdot(7,1)\psdot(10,4)\psdot(4,7)\psdot(5,3.5)\psdot(6.25,5)\psdot(7.5,3)
\psdot(3.5,2.5)
\psdot(0.5,6)
\end{pspicture}
\qquad
\begin{pspicture}(10,7)
\border(-0.75,-0.75)(10.5,7.5)
\convexset{orange}(6,0)(10,4)(4,7)(0,0)
\psdot(0,0)\psdot(2,0)\psdot(3,0.25)\psdot(1,1)
\psdot(6,0)
\psdot(7,1)\psdot(10,4)\psdot(4,7)\psdot(5,3.5)\psdot(6.25,5)\psdot(7.5,3)
\psdot(3.5,2.5)
\psdot(0.5,6)
\uput[0](6,0){$-$}
\uput[0](7,1){$1$}
\uput[90](10,4){$+$}
\uput[0](4,7){$-$}
\uput[180](5,3.5){$1$}
\uput[-40](6.25,5){$1$}
\uput[0](7.5,3){$1$}
\uput[90](0,0){$-$}
\uput[120](2,0){$1$}
\uput[55](3,0.25){$1$}
\uput[45](1,1){$1$}
\uput[0](3.5,2.5){$1$}
\uput[0](0.5,6){$0$}
\end{pspicture}
\caption{The products $\alpha\circ\beta$ (left) and $\beta\circ\alpha$
(right) of the covectors $\alpha$ and $\beta$ in
\ref{f:TwoCovectorsInAConvexGeometry}.}
\label{f:ProductOfCovectorsInAConvexGeometry}
\end{figure}
\end{Example}

\begin{Proposition}
\label{p:CovectorOfProduct}
Suppose $\alpha$ and $\beta$ are the covectors of the signed flats 
$(A,\Flat\alpha)$ and $(B,\Flat\beta)$, respectively. 
Then the covector $\gamma$ of $(A,\Flat\alpha) \circ (B,\Flat\beta)$ 
is $\alpha\circ\beta$.
\end{Proposition}
\begin{proof}
By definition, the covector $\gamma$ of the signed flat $(A\vee B,
\Flat\alpha\circ\Flat\beta)$ is given by: $\gamma(e) =
(\Flat\alpha\circ\Flat\beta)(e)$ if $e \in \Gamma(A \vee B)$; $\gamma(e) =
0$ if $e \in \xi(A \vee B)$; and $\gamma(e) = 1$ otherwise.

Suppose $e \not\in \Gamma(A \vee B) \cup \xi(A \vee B)$. By the definition
of the product of covectors, $(\alpha\circ\beta)(e) = 1$. Hence,
$(\alpha\circ\beta)(e)=\gamma(e)$.  

Suppose $e\in\xi(A\vee B)$. Then $e\in\xi(A)$ and $e\in\xi(B)$ since $\xi$
is order-reversing. This implies that $\alpha(e)=\beta(e)=0$, hence
$(\alpha\circ\beta)(e)=0$. Therefore, $\acb(e)=\gamma(e)$. 

Suppose $e \in \Gamma(A \vee B)$. 
By \ref{p:PropertiesOfContinuations}(3), $e \in \Gamma(A) \cup \Gamma(B)$. 
If $e\in\Gamma(A)$, then $\beta(e)\not>\alpha(e)$, so
$\acb(e)=\alpha(e)=\Flat\alpha(e)$. 
If $e\not\in\Gamma(A)$, then $e\in\Gamma(B)\cap\xi(A)$ and so
$\beta(e)>0=\alpha(e)$. Hence, $\acb(e)=\beta(e)=\Flat\beta(e)$.
Therefore, $\acb(e)=(\Flat\alpha\circ\Flat\beta)(e)$ for all
$e\in\Gamma(A\vee B)$.
\end{proof} 

\begin{Example}
\label{x:Rank1ComplexArrangement}
Let $E = \{x,y\}$ and $\F = \{\O, \{y\}, \{x,y\}\}$. Then $(E,\F)$ is
an \uig. 
There are five covectors of $(E,\F)$,
described in the following table. 
\begin{gather*}
\begin{array}{c|c|c}
[X] & \Gamma(X) & \text{covectors over } [X] \\ \hline
\null[\O] & \{y\} & (1,+),\ (1,-) \\
\null[\{y\}] & \{x\} & (+,0),\ (-,0) \\
\null[\{x,y\}] & \O & (0,0)
\end{array}
\end{gather*}
The partial order on these covectors is illustrated below.
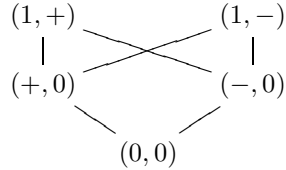
\begin{gather*}
\xymatrix@R=1em@C=1em{
(1,+)\ar@{-}[drr]\ar@{-}[d] && (1,-)\ar@{-}[d]\ar@{-}[dll] \\
(+,0)\ar@{-}[dr] && (-,0)\ar@{-}[dl] \\
& (0,0)}
\end{gather*}
Observe that the product of two covectors $\alpha$ and $\beta$ can be
computed using $\star$, or using the following identity:
\begin{align*}
\alpha\circ\beta = 
\begin{cases}
\beta, & \text{if } \beta > \alpha, \\
\alpha, & \text{otherwise}.
\end{cases}
\end{align*}
For example, $(+,0)\circ(-,0)=(+,0)$
and $(+,0)\circ(1,-)=(1,-)$.
\end{Example}

Let $\alpha$ and $\beta$ be covectors of $(E,\F)$.
The \defn{separation set} of $\alpha$ and $\beta$ is
\begin{gather*}
\SeparationSet(\alpha,\beta) 
  = \{e \in E: \alpha(e) = -\beta(e) \in \{+,-\}\}.
\end{gather*}
Note that $\SeparationSet(\alpha,\beta) \subseteq
\Gamma(\supp(\alpha))\cap\Gamma(\supp(\beta))$.

The next result establishes some properties about covectors.
See also \ref{p:ProductAndPartialOrder}.

\begin{Lemma}
\label{l:CovectorsAndSeparationSets}
Let $\alpha$ and $\beta$ be covectors of an interval greedoid $(E,\F)$. 
\begin{enumerate}
\item
\label{i:CovectorProductAndPartialOrder}
$\alpha\leq\beta$ if and only if $\alpha\circ\beta = \beta$.
\item
\label{i:CriterionForLEQ2}
$\alpha\leq\beta$ if and only if $\alpha(e)\leq\beta(e)$ for all $e\in E$.
\item 
\label{i:CriterionForLEQ1}
$\alpha \leq \beta$ if and only if 
$\SeparationSet(\alpha, \beta) 
= \O$ and $\supp(\alpha) \leq \supp(\beta)$.
\item
\label{i:CovectorProductInvolving1}
If $\alpha(e) = 1$ or $\beta(e) = 1$, 
 then $(\alpha \circ \beta)(e) = 1 = (\beta \circ \alpha)(e)$.
\end{enumerate}
\end{Lemma}
\begin{proof}
Let $A = \supp(\alpha)$ and $B=\supp(\beta)$.
By definition, $\alpha\leq\beta$ if and only if $A\leq B$ and $\alpha$
and $\beta$ agree on $\Gamma(A)\cap\Gamma(B)$. 

\eqref{i:CovectorProductAndPartialOrder}
This follows from \ref{p:ProductAndPartialOrder} and
\ref{p:CovectorOfProduct}.

\eqref{i:CriterionForLEQ2}
Suppose $\alpha\leq\beta$ and let $e\in E$. If
$e\notin\xi(B)\cup\Gamma(B)$, then $\beta(e)=1$, so
$\alpha(e)\leq\beta(e)$. If $e\in\xi(A)$, then $\alpha(e)=0$, so
$\alpha(e)\leq\beta(e)$. So suppose $e\in\xi(B)\cup\Gamma(B)$ and
$e\notin\xi(A)$. Then $e\in\Gamma(A)\cap\Gamma(B)$, by
\ref{p:PropertiesOfContinuations}. Then $\alpha(e)\leq\beta(e)$
because $\alpha$ and $\beta$ agree on $\Gamma(A)\cap\Gamma(B)$.

Conversely, suppose $\alpha(e)\leq\beta(e)$ for all $e\in E$. Since
$\{e : \beta(e)=0\} \subseteq \{e: \alpha(e)=0\}$, we have
$\supp(\alpha)\leq\supp(\beta)$. If $e\in\Gamma(A)\cap\Gamma(B)$, then
$\alpha(e),\beta(e)\in\{+,-\}$, which implies $\alpha(e)=\beta(e)$
because $\alpha(e)\leq\beta(e)$. Thus, $\alpha\leq\beta$.

\eqref{i:CriterionForLEQ1}
If $\alpha \leq \beta$, then $A \leq B$, and
$\SeparationSet(\alpha,\beta) = \O$ because $\SeparationSet(\alpha,
\beta) \subseteq \Gamma(A) \cap \Gamma(B)$. Conversely, if $A\leq B$
and $\SeparationSet(\alpha,\beta)=\O$, then $\alpha(e)=\beta(e)$ for
all $e\in\Gamma(A)\cap\Gamma(B)$, so $\alpha\leq\beta$.

\eqref{i:CovectorProductInvolving1}
If $(\alpha \circ \beta)(e) \neq 1$, then $e \in \Gamma(A \vee B)
\cup \xi(A \vee B) \subseteq (\Gamma(A) \cup \xi(A)) \cap (\Gamma(B)
\cup \xi(B))$ by \ref{p:PropertiesOfContinuations}(4).
Hence, $\alpha(e) \neq 1$ and $\beta(e) \neq 1$.
\end{proof}

\begin{Remark}
The converse of \eqref{i:CovectorProductInvolving1} is false.
Counter-examples are depicted in
\ref{f:ConverseOfLemmaOnCovectorsAndSeparationSets}. They also
illustrate that the following containments can be proper.
\begin{gather*}
\Gamma(A \vee B) \cup \xi(A \vee B) \subseteq
 (\Gamma(A) \cup \xi(A) ) \cap (
 (\Gamma(B) \cup \xi(B) ),\\
 \xi(A' \vee B') \subseteq \xi(A') \cap \xi(B').
\end{gather*}
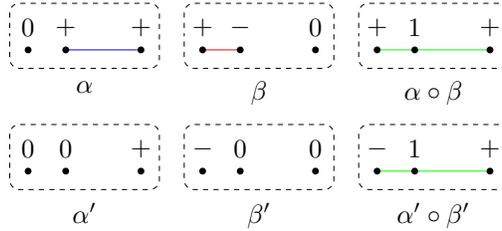
\begin{figure}[!ht]
\centering
\def\xO{-0.5}\def\yO{-0.5}
\def\xl{3.5}\def\yl{1.25}
\begin{pspicture}(3,1)(0,-0.75)
\border(-0.5,-0.5)(3.5,1.25)
\psline[linecolor=blue,linewidth=0.02](1,0)(3,0)
\psdot(0,0)\psdot(1,0)\psdot(3,0)
\uput[90](0,0){$0$}
\uput[90](1,0){$+$}
\uput[90](3,0){$+$}
\uput[-90](1.5,-0.5){$\alpha$}
\end{pspicture}
\qquad
\begin{pspicture}(3,1)(0,-0.75)
\border(-0.5,-0.5)(3.5,1.25)
\psline[linecolor=red,linewidth=0.02](0,0)(1,0)
\psdot(0,0)\psdot(1,0)\psdot(3,0)
\uput[90](0,0){$+$}
\uput[90](1,0){$-$}
\uput[90](3,0){$0$}
\uput[-90](1.5,-0.5){$\beta$}
\end{pspicture}
\qquad
\begin{pspicture}(3,1)(0,-0.75)
\border(-0.5,-0.5)(3.5,1.25)
\psline[linecolor=green,linewidth=0.02](0,0)(3,0)
\psdot(0,0)\psdot(1,0)\psdot(3,0)
\uput[90](0,0){$+$}
\uput[90](1,0){$1$}
\uput[90](3,0){$+$}
\uput[-90](1.5,-0.5){$\alpha\circ\beta$}
\end{pspicture} 
\\\vspace{2em}
\begin{pspicture}(3,1)(0,-0.75)
\border(-0.5,-0.5)(3.5,1.25)
\psdot(0,0)\psdot(1,0)\psdot(3,0)
\uput[90](0,0){$0$}
\uput[90](1,0){$0$}
\uput[90](3,0){$+$}
\uput[-90](1.5,-0.5){$\alpha'$}
\end{pspicture}
\qquad
\begin{pspicture}(3,1)(0,-0.75)
\border(-0.5,-0.5)(3.5,1.25)
\psdot(0,0)\psdot(1,0)\psdot(3,0)
\uput[90](0,0){$-$}
\uput[90](1,0){$0$}
\uput[90](3,0){$0$}
\uput[-90](1.5,-0.5){$\beta'$}
\end{pspicture}
\qquad
\begin{pspicture}(3,1)(0,-0.75)
\border(-0.5,-0.5)(3.5,1.25)
\psline[linecolor=green,linewidth=0.02](0,0)(3,0)
\psdot(0,0)\psdot(1,0)\psdot(3,0)
\uput[90](0,0){$-$}
\uput[90](1,0){$1$}
\uput[90](3,0){$+$}
\uput[-90](1.5,-0.5){$\alpha'\circ\beta'$}
\end{pspicture}
\caption{Counter-examples to the converse of \ref{l:CovectorsAndSeparationSets}
\eqref{i:CovectorProductInvolving1}.}
\label{f:ConverseOfLemmaOnCovectorsAndSeparationSets}
\end{figure}
\end{Remark}

\subsection{Oriented interval greedoids}

For any covector $\alpha$, let $-\alpha$ be the covector obtained from
$\alpha$ by replacing $+$ with $-$ and $-$ with $+$.

\begin{Definition}
\label{d:OrientedIntervalGreedoid}
An \defn{oriented interval greedoid} is a triple $(E,\F, \OIG)$, where
$(E,\F)$ is an interval greedoid and $\OIG$ is a set of covectors of
$(E,\F)$ satisfying the following axioms.
\begin{itemize}
\item[(OG1)] 
The map $\supp:\OIG \to \Phi$ is surjective.
\item[(OG2)] 
If $\alpha \in \OIG$, then $-\alpha \in \OIG$.
\item[(OG3)] 
If $\alpha, \beta \in \OIG$, then $\alpha \circ \beta \in \OIG$.
\item[(OG4)] 
If $\alpha,\beta\in\OIG$, $x\in\SeparationSet(\alpha,\beta)$ and
$(\alpha\circ\beta)(x)\neq1$, then there exists $\gamma\in\OIG$ such
that $\gamma(x)=0$ and for all $y\notin\SeparationSet(\alpha,\beta)$,
if $(\alpha\circ\beta)(y)\neq1$, then 
$\gamma(y)=(\alpha\circ\beta)(y)=(\beta\circ\alpha)(y)$.

\end{itemize}
\end{Definition}

As we will see in Section \ref{sss:OrientedMatroids}, these conditions
are modelled on the covector axioms for oriented matroids. In the next
section we will present various examples of oriented interval
greedoids. We record here the following observation.

\begin{Lemma}
Suppose $\alpha$ and $\beta$ are covectors of an oriented interval
greedoid. If $(\alpha\circ\beta)(y)\neq(\beta\circ\alpha)(y)$, then
$\alpha(y)=-\beta(y)\in\{+,-\}$ (that is,
$y\in\SeparationSet(\alpha,\beta)$.)
\end{Lemma}

\begin{proof}
Let $C=\supp(\alpha\circ\beta)=\supp(\beta\circ\alpha)$.  
Then $(\alpha\circ\beta)(y)=1$ iff $y \not\in \Gamma(C)\cup\xi(C)$
iff $(\beta\circ\alpha)(y)=1$.
Similarly,  $(\alpha\circ\beta)(y)=0$ iff $y \in\xi(C)$
iff $(\beta\circ\alpha)(y)=0$.  Thus $\alpha(y),\beta(y)
\subset \{+,-\}$.  The result follows.    
\end{proof}

\begin{Corollary}
Suppose $\alpha$ and $\beta$ are covectors of an oriented interval
greedoid. Then $(\alpha\circ\beta)(y)=(\beta\circ\alpha)(y)$ for all
$y\notin\SeparationSet(\alpha,\beta)$. 
\end{Corollary}

\subsection{Examples} 
This section presents some examples of oriented interval greedoids.

\subsubsection{Oriented Matroids}
\label{sss:OrientedMatroids}
Let $E$ be a finite set. An \defn{oriented matroid} is a collection
$\OM$ of maps from $E$ to $\{0,+,-\}$ that satisfies the following
axioms.
\begin{itemize}
\item[(OM1)] $\OM$ contains the map $z(e)=0$ for all $e\in E$.
\item[(OM2)] If $\alpha \in \OM$, then $-\alpha \in \OM$.
\item[(OM3)] If $\alpha,\beta \in \OM$, then $\alpha \circ \beta \in \OM$, where
\begin{gather*}
 (\alpha \circ \beta)(e) =
 \begin{cases}
 \alpha(e), & \text{if } \alpha(e) \neq 0, \\
 \beta(e), & \text{if } \alpha(e) = 0.
 \end{cases}
\end{gather*}
\item[(OM4)] Suppose $\alpha,\beta \in \OM$ and let
$S(\alpha,\beta) = \{ e\in E: \alpha(e) = - \beta(e) \neq 0\}$.
For every $e \in S(\alpha,\beta)$ there exists $\gamma \in \OM$ 
with $\gamma(e) = 0$ and $\gamma(f) = (\alpha \circ \beta)(f)
= (\beta \circ \alpha)(f)$ for all $f \not\in S(\alpha,\beta)$.
\end{itemize}

If $\OM$ is an oriented matroid, then the set of zeros of the elements
of $\OM$ form the closed sets of a matroid $(E,\F)$. The matroid
$(E,\F)$ is the \defn{underlying matroid} of the oriented matroid and
$\OM$ is said to be an oriented matroid on $(E,\F)$.

\begin{Theorem}
\label{p:OMAndOIGs}
Suppose $(E,\F)$ is a matroid without loops. 
Then $\OM$ is an oriented matroid with
underlying matroid $(E,\F)$ if and only if $(E,\F,\OM)$ is an oriented
interval greedoid.
\end{Theorem}

\begin{proof}
Let $(E,\F)$ be a matroid and let $(E,\F,\OIG)$ be an oriented
interval greedoid. Since $\xi(A) = E \m \Gamma(A)$ for any flat $A$ of
a matroid without loops, a covector of $(E,\F)$ takes values in $\{0,+,-\}$.
Therefore, $\OIG$ is a collection of maps from $E$ to $\{0,+,-\}$, and
$\OIG$ satisfies (OM1)--(OM4) since it satisfies (OG1)--(OG4). So
$\OIG$ is an oriented matroid.

Conversely, suppose that $\OM$ is an oriented matroid with underlying
matroid $(E,\F)$. If $\alpha \in \OM$, then the set $\zeta(\alpha)$ of
zeros of $\alpha$ is a closed set of the matroid, and there is a
unique flat $A$ satisfying $\xi(A) = \zeta(\alpha)$. Therefore,
$\alpha$ gives a signed flat $(A, \alpha|_{\Gamma(A)})$, and the
covector of this signed flat is $\alpha$. So $\OM$ is a set of
covectors of the interval greedoid $(E,\F)$. It is straightforward to
check that the axioms for an oriented interval greedoid are satisfied
by $\OM$.
\end{proof}

\subsubsection{Antimatroids}
Next we show that the set of all covectors of an antimatroid forms an
oriented interval greedoid.  This collection of covectors, viewed as a 
poset, is the central object of study in the work of Billera, Hsiao,
and Provan \cite{BilleraHsiaoProvan2008}.  We also show that 
this is the only oriented interval greedoid structure on an antimatroid.

We begin with an example to illustrate how
to obtain a covector $\gamma$ satisfying (OG4).

\begin{Example}
Let $\alpha$ and $\beta$ be the covectors in 
\ref{f:TwoCovectorsInAConvexGeometry}. Then
$\SeparationSet(\alpha,\beta)=\{x\}$, where $x$ is the vertex that is
circled in \ref{f:OG4InConvexGeometry}. Let $\gamma$ be the covector
in \ref{f:OG4InConvexGeometry}. Then $\gamma(x)=0$; and for all
$y\notin\SeparationSet(\alpha,\beta)$:
\begin{enumerate}
\item 
if $(\alpha\circ\beta)(y)=0$, then $\gamma(y)=(\beta\circ\alpha)(y)=0$.
\item
if $(\alpha\circ\beta)(y)=+$, then $\gamma(y)=(\beta\circ\alpha)(y)=+$.
\item
if $(\alpha\circ\beta)(y)=-$, then $\gamma(y)=(\beta\circ\alpha)(y)=-$.
\item
if $(\alpha\circ\beta)(y)=1$, then $\gamma(y)\neq0$. \qedhere
\end{enumerate}
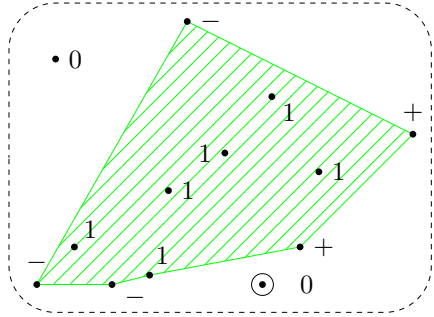
\begin{figure}[!ht]
\centering
\begin{pspicture}(10,7)
\border(-0.75,-0.75)(10.5,7.5)
\convexset{green}(0,0)(2,0)(3,0.25)(7,1)(10,4)(4,7)(0,0)
\psdot(0,0)\psdot(2,0)\psdot(3,0.25)\psdot(1,1)
\psdot(6,0)\pscircle[linewidth=0.02](6,0){0.3}
\psdot(7,1)\psdot(10,4)\psdot(4,7)\psdot(5,3.5)\psdot(6.25,5)\psdot(7.5,3)
\psdot(3.5,2.5)
\psdot(0.5,6)
\uput{1.0}[0](6,0){$0$}
\uput[0](7,1){$+$}
\uput[90](10,4){$+$}
\uput[0](4,7){$-$}
\uput[180](5,3.5){$1$}
\uput[-40](6.25,5){$1$}
\uput[0](7.5,3){$1$}
\uput[90](0,0){$-$}
\uput[-30](2,0){$-$}
\uput[55](3,0.25){$1$}
\uput[45](1,1){$1$}
\uput[0](3.5,2.5){$1$}
\uput[0](0.5,6){$0$}
\end{pspicture}
\caption{If $\alpha$ and $\beta$ are the two covectors in
\ref{f:TwoCovectorsInAConvexGeometry} and $x$ is the circled vertex,
then the covector $\gamma$ illustrated here satisfies the conditions
of (OG4).}
\label{f:OG4InConvexGeometry}
\end{figure}
\end{Example}

\begin{Theorem}\label{thm:antiOIG}
Suppose $(E,\F)$ is an \uig. Let $\OIG$ denote
the set of all covectors of $(E,\F)$. Then $(E,\F,\OIG)$ is
an oriented interval greedoid.
\end{Theorem}

\begin{proof}
We will show that (OG1)--(OG4) hold.

(OG1) Suppose $A\in\Phi$ is a flat. Let $(A,\Flat\alpha)$ be a signed flat
of $(E,\F)$ and let $\alpha$ denote the covector of this signed flat.
Then $\alpha\in\OIG$ and $\supp(\alpha)=A$.

(OG2) Suppose $\alpha\in\OIG$. Let $A=\supp(\alpha)$. Then
$(A,-\alpha|_{\Gamma(A)})$ is a signed flat of $(E,\F)$. The covector
of this signed flat is precisely $-\alpha$. So $-\alpha\in\OIG$.

(OG3) If $\alpha,\beta\in\OIG$, then $\alpha\circ\beta$ is a covector
of $(E,\F)$, so $\alpha\circ\beta\in\OIG$.

(OG4) Suppose $\alpha,\beta\in\OIG$ and
$x\in\SeparationSet(\alpha,\beta)$ satisfies
$(\alpha\circ\beta)(x)\neq1$. Let $(E,\tau)$ denote the convex
geometry that is complementary to $(E,\F)$ (see \ref{sss:UIG}). Let
$A=\supp(\alpha)$ and $B=\supp(\beta)$. Then $A=[X]$ and $B=[Y]$ for
some $X,Y\in\F$. Let $X'=E\m X$ and $Y'=E\m Y$. Then $X'$ and $Y'$ are
closed sets in $(E,\tau)$.

\emph{Step 1: We show that $x$ is an extreme point of $\tau(X'\cup
Y')$.} 
Since $(\alpha\circ\beta)(x)\not\in\{0,1\}$, we have $x\in\Gamma(A \vee
B)$. 
Since $A\vee B = [E\m\tau(X'\cup Y')]$ (\ref{x:AntimatroidMeets}), we
have $\Gamma(A\vee B) = \ext(\tau(X'\cup Y'))$ by
\ref{l:ExtremePointsAndContinuations}. Thus, $x\in\ext(\tau(X'\cup
Y'))$.

\emph{Step 2: We define $\gamma$.}
Let $Z'=\tau(X'\cup Y')-x$. Since $x$ is an extreme point of
$\tau(X'\cup Y')$, it follows that $x$ is not in $Z'$. Hence $Z'$ is a
closed set not containing $x$. Let $Z=E\m Z'$. Then $Z\in\F$ and $x\in
Z$. Define a map $\gamma':\Gamma(Z)\to\{+,-\}$ for $y\in\Gamma(Z)$ as
follows: if $y\in\Gamma(A\vee B)$, then set
$\gamma'(y)=(\alpha\circ\beta)(y)$; otherwise arbitrarily set
$\gamma'(y)$ to be $+$ or $-$. Let $\gamma$ be the covector of the
signed flat $([Z],\gamma')$ 

\emph{Step 3: $\gamma$ has the desired properties.} First note that
$\gamma\in\OIG$ since $\gamma$ is a covector of $(E,\F)$. Next observe
that $\gamma(x)=0$ since $x\in Z\subseteq\xi([Z])$. Let
$y\notin\SeparationSet(\alpha,\beta)$.


Suppose that $(\alpha\circ\beta)(y)\neq1$. Then $y\in\Gamma(A\vee
B)\cup\xi(A\vee B)$. Since $\xi(A\vee B)=E\m\tau(X'\cup Y')\sseq
Z\sseq\xi([Z])$, if $y\in\xi(A\vee B)$, then
$\gamma(y)=0=\acb(y)=\bca(y)$. On the other hand, if $y\in\Gamma(A\vee
B)=\ext(\tau(X'\cup Y'))$, then $y$ is an extreme point of
$Z'=\tau(X'\cup Y')-x$ (since $y\neq x$). Equivalently,
$y\in\Gamma(Z)$. So, by definition of $\gamma$ and because
$y\notin\SeparationSet(\alpha,\beta)$, $\gamma(y)=\acb(y)=\bca(y)$.
\end{proof}

\begin{Proposition} Let $(E,\F)$ be an antimatroid.  Then the only oriented
structure on $(E,\F)$ is that constructed in \ref{thm:antiOIG}.  
\end{Proposition}

\begin{proof} Let $(E,\F,\OIG)$ be an oriented interval greedoid.  
Let $\alpha$ be an arbitrary covector.  We wish to show that 
$\alpha\in\OIG$.  

Let $\Gamma(\supp(\alpha))=X=\{x_1,\dots,x_r\}$.  Let $Y=E\setminus (X\cup \supp(\alpha))$.  
For any $x\in X$, $Y_x=\supp(\alpha)\cup (X\setminus \{x\})$ is feasible.  
Also, $\Gamma(Y_x)\cap\Gamma(\supp(\alpha))=\{x\}$.  By (OIG1), we can
find a covector $\beta_x\in \OIG$ with $\supp(\beta_x)=Y_x$.  
By (OIG2), we can choose $\beta_x$ so that $\beta_x$ agrees with
$\alpha$ on $x$.  Now $\beta_{x_1}\circ\dots\circ\beta_{x_r}=\alpha$ is
in $\OIG$.  
\end{proof}

\subsubsection{Complexified Hyperplane Arrangements}

An (essential) \defn{real hyperplane arrangement} is a finite set of 
hyperplanes $\{\RH_1,\RH_2,\dots,\RH_n\}$ in $\mathbb R^d$ satisfying
$\bigcap \RH_i=\{\vec0\}$.  
Let
$E=\{1,2,\ldots,n\}$ and for each $e\in E$ fix a linear form
$\ell_e:\mathbb R^d\to\mathbb R$ such that $\RH_e=\ker(\ell_e)$. 
Extending scalars, we can also think of $\ell_e$ as defining a linear map from
$\mathbb C^d$ to $\mathbb C$.  Define $H_e$ to be the kernel of this map.
It is a hyperplane in $\mathbb C^d$.  The collection 
$\Arrangement=\{H_1,\dots,H_n\}$ forms a \defn{complexified hyperplane
arrangement}.  
Also define
$H_e^\Re
= \{\vec z\in\mb C^d: \Im(\ell_e(\vec z))=0\}$. 

(Note that not all complex hyperplane arrangements are complexified
arrangments;
that is to say, not all complex hyperplane arrangements arise from a 
real hyperplane arrangement in the way we have just described.)


For any $z=x+iy\in\mathbb C$, let 
\begin{align*}
\sigma_\Re(x+iy) =
\begin{cases}
\hfill1, & \text{if } y \neq 0, \\
\hfill+, & \text{if } y = 0, x > 0, \\
\hfill-, & \text{if } y = 0, x < 0, \\
\hfill0, & \text{if } y = 0, x = 0,
\end{cases}
\qquad
\sigma_\Im(x+iy) =
\begin{cases}
\hfill+, & \text{if } y > 0, \\
\hfill-, & \text{if } y < 0, \\
\hfill0, & \text{if } y = 0,
\end{cases}
\end{align*}
and for every $\vec z\in\mathbb C^d$, let 
\begin{align*}
\alpha_{\vec z}(h) =
\begin{cases}
\sigma_\Im(\ell_i(\vec z)), & \text{if } h = H_i^\Re, \\
\sigma_\Re(\ell_i(\vec z)), & \text{if } h = H_i.
\end{cases}
\end{align*}
Note that
$(\alpha(H_i),\alpha(H_i^\Re))\in\{(0,0),(+,0),(-,0),(1,+),(1,-)\}$
for all $1\leq e\leq n$. 

\begin{Example}
\label{x:Rank1ComplexArrangementCovectors}
There is a unique complexified hyperplane arrangement in $\mb C$, namely
$\Arrangement = \{H_0=\{\vec 0\}\}$. In this case $\{\alpha_{\vec z}:
\vec z\in\mb C\} = \{(0,0),(+,0),(-,0),(1,+),(1,-)\}$ is the set of
covectors of the interval greedoid $(E,\F)$ with $E=\{H_0,H_0^\Re\}$
and $\F=\{\O,\{H_0^\Re\},\{H_0,H_0^\Re\}\}$ (cf.
\ref{x:Rank1ComplexArrangement}).
\ref{f:PartialOrderOnComplexCovectors} illustrates the partial order
on these covectors.
\end{Example}
\begin{figure}
\begin{gather*}
\xymatrix@R=1em@C=1em{
(1,+)\ar@{-}[drr]\ar@{-}[d] && (1,-)\ar@{-}[d]\ar@{-}[dll] \\
(+,0)\ar@{-}[dr] && (-,0)\ar@{-}[dl] \\
& (0,0)}
\end{gather*}
\caption{The poset of covectors of the complex hyperplane arrangement
in $\mb C$.}
\label{f:PartialOrderOnComplexCovectors}
\end{figure}
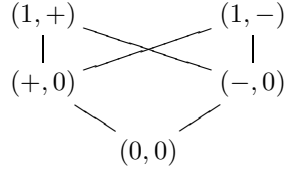

Let $\Arrangement=\{H_1,\ldots,H_n\}$ be a complexified hyperplane
arrangement in $\mb C^d$. Let $\IntersectionLattice$ be the lattice of
all intersections of subspaces from the set
\begin{align*}
E_\Arrangement=\{H_1,\ldots,H_n,H_1^\Re,\ldots,H_n^\Re\},
\end{align*}
ordered by inclusion. Then $\IntersectionLattice$ is a lower
semimodular lattice and $E_\Arrangement$ is the set of
meet-irreducible elements of $\IntersectionLattice$
\cite{BjornerZiegler1992:b}.
By \ref{p:IGsFromSemimodularLattices},
$(E_\Arrangement,\F_\Arrangement)$ is an interval greedoid, where 
\begin{align*}
\F_\Arrangement = \left\{\{h_1,h_2,\ldots,h_k\}\subseteq E_\Arrangement : 
\mb C^d \gtrdot h_1 \gtrdot 
\cdots \gtrdot (h_1\cap h_2\cap \cdots\cap h_k) 
\right\}.
\end{align*}

\begin{Lemma} Let $\Arrangement$ be a complexified hyperplane arrangement,
and let $(E_\Arrangement,\F_\Arrangement)$ be the interval greedoid as
defined above.  Then, for $X\in\F_\Arrangement$, we have
\begin{gather}\label{eqone}
\xi(X) = \left\{h\in E: \bigcap_{h'\in X} h' \subseteq h\right\}, \\
\label{eqtwo}
\Gamma(X) = 
\left\{ H_e^\Re \in E : \bigcap_{h\in X} h \not\subseteq H_e^\Re\right\}
\cup
\left\{ H_e \in E : 
\bigcap_{h\in X} h \subseteq H_e^\Re \text{ and } 
\bigcap_{h\in X} h \not\subseteq H_e 
\right\}.
\end{gather}
\end{Lemma}

\begin{proof}
\ref{eqone} follows directly from \ref{x:SemimodularXi}.  We now
show  \ref{eqtwo}.  Let $M=\cap_{h\in X}h\subset \mathbb C^d$.
Thinking of $\mathbb C^d$ as a $2d$-dimensional real vector space, we can
decompose it into real and complex parts as $\mathbb C^d = 
\Re(\mathbb C^d)\oplus \Im(\mathbb C^d)$, where each of the summands is a $d$-dimensional
real vector space, and multiplication by $i$ provides an isomorphism
from $\Re(\mathbb C^d)$ to $\Im(\mathbb C^d)$.  
Note that $H_i$ and $H_i^\Re$ can also be expressed as a direct sum of 
a real and a complex part.  
(This relies on the fact that our arrangement is a complexified
real arrangment,
rather than being an arbitrary complex arrangement.)  
Note further that in either case, the 
imaginary part corresponds to a subspace of the real part.   
It follows that
$M$, also, can be written as $M= \Re(M)\oplus \Im(M)$, with $\Im(M)$
a subspace of $\Re(M)$.  
  
Observe first that if $H_e^\Re \not\geq M$, then, since 
$H_e^\Re$ is real codimension one in $\mathbb C^d$, we have 
$M\gtrdot M\cap H_e^\Re$, so 
$H_e^\Re\in \Gamma(X)$.  
Also, in this case, we have $M \cap H_e^\Re \gtrdot M\cap H_e$, because 
$\RH_e\not\geq \Im(M)$, and thus $\RH_e\not\geq \Re(M)$ either.  
It follows that in this case $H_e^\Re\in \Gamma(X)$ and $H_e\not\in\Gamma(X)$.

Finally, if $H_e^\Re \geq M$, we observe that $H_e$ is codimension one
in $H_e^\Re$, and thus that either $H_e \geq M$ or $M\gtrdot M\cap H_e$.  
This completes the proof of the lemma. \end{proof}

\begin{Lemma}
Let $\Arrangement=\{H_1,\ldots,H_n\}$ be a complexified hyperplane
arrangement in $\mb C^d$. Then $\alpha_{\vec z}$ is a covector over
the interval greedoid $(E_\Arrangement, \F_\Arrangement)$, for every
$\vec z\in\mb C^d$. 
\end{Lemma}

\begin{proof}
Recall that a map $\alpha:E\to\{0,+,-,1\}$ is a covector of an
interval greedoid $(E,\F)$ if and only if there exists $X\in\F$ such
that $\alpha(e)=0$ if and only if $e\in\xi(X)$, $\alpha(e)=\pm$ if and
only if $e\in\Gamma(X)$, and $\alpha(e)=1$ otherwise.

Let $\alpha=\alpha_{\vec z}$ be defined as above. Observe that
$\alpha(h)=0$ if and only if $\vec z\in h$. Let
\begin{align*}
A &= 
\left\{ h\in E: \alpha(h) = 0 \right\} =
\left\{ H_i : \vec z \in H_i \right\} \cup
\left\{ H_i^\Re : \vec z \in H_i^\Re \right\} \subseteq E
\end{align*}
and let $X$ be maximal among the elements of $\F_\Arrangement$
contained in $A$. 

We show that $\alpha(h) = 0$ if and only if $h\in\xi(X)$ by showing
that $\xi(X)=A$. Suppose $h\in\xi(X)$. Then $h \supseteq \cap_{h'\in
X}h'$. Since $X\subseteq A$, it follows that $\alpha(h')=0$ for all
$h'\in X$. Thus, $\vec z\in h'$ for all $h'\in X$. It follows that
$\vec z\in h$. Thus, $h\in A$.

Conversely, suppose $h\in A$. Then $\alpha(h)=0$. If $h = H_i^\Re$,
then $\{h\}\in\F_\Arrangement$, so we can augment $\{h\}$ from $X$
until we get a set $Y$ of cardinality $|X|$. Since $X$ is maximal
among the feasible sets contained in $A$ and $|X|=|Y|$, $Y$ is maximal
as well. Thus, $X \sim Y$, so $Y \subseteq \xi(X)$. In particular,
$h\in\xi(X)$. On the other hand, if $h = H_i$, then $\ell_i(\vec
z)=0$, so $H_i^\Re\in A$ also. Since
$\{H_i,H_i^\Re\}\in\F_\Arrangement$, the same argument shows that
$h\in\xi(X)$. Thus, $A\subseteq\xi(X)$.

Next we show that $\alpha(h)\in\{+,-\}$ if and only if
$\alpha(h)\in\Gamma(X)$. Let $h\in\Gamma(X)$. Since $\xi(X)$ and
$\Gamma(X)$ are disjoint, it follows from the above that
$\alpha(h)\neq0$. So it suffices to show that $\alpha(h)\in\{0,+,-\}$.
By construction, this is true for $h=H_i^\Re$ since
$\alpha(H_i^\Re)=\sigma_\Im(\ell_i(\vec z))\in\{0,+,-\}$. If $h=H_i$,
then, by the above description of $\Gamma(X)$, we have $H_i^\Re\in
\xi(X)$. So, $\sigma_\Im(\ell_i(\vec z))=0$, which implies
$\alpha(h)=\sigma_\Re(\ell_i(\vec z))\in\{0,+,-\}$. 

Conversely, suppose $\alpha(h)\in\{+,-\}$. Since $\alpha(h)\neq0$, we
have $h\notin\xi(X)$, or equivalently, $\bigcap_{x\in X}x\not\subseteq
h$. So if $h=H_i^\Re$, then $h\in\Gamma(X)$. If $h=H_i$, then we need
to show that $H_i^\Re\supseteq\bigcap_{x\in X}x$, or equivalently,
$H_i^\Re\in\xi(X)$. Well, $\sigma_\Re(\ell_i(\vec
z))=\alpha(H_i)\in\{+,-\}$, so $\sigma_\Im(\ell_i(\vec z))=0$. This
implies $H_i^\Re\in\xi(X)$. Hence, $h=H_i\in\Gamma(X)$. 

Finally, it follows from the above that $\alpha(h)=1$ if and only if
$h\notin\Gamma(X)\cup\xi(X)$. Therefore, $\alpha$ is a covector of
$(E_\Arrangement,\F_\Arrangement)$.
\end{proof}

\begin{Remark}
As in \ref{x:Rank1ComplexArrangement}, the product of two covectors
$\alpha$ and $\beta$ can be computed component-wise, or pair-wise
using the identity:
\begin{align*}
 \Big(\acb(H_i),&\ \acb(H_i^\Re)\Big) \\
&=
 \begin{cases}
 \big(\beta(H_i),\beta(H_i^\Re)\big), & \text{if }  
 \big(\beta(H_i),\beta(H_i^\Re)\big) > \big(\alpha(H_i),\alpha(H_i^\Re)\big), \\
 \big(\alpha(H_i),\alpha(H_i^\Re)\big), & \text{otherwise},
 \end{cases}
\end{align*}
where the comparison $(\beta(H_i),\beta(H_i^\Re)) >
(\alpha(H_i),\alpha(H_i^\Re))$ is performed in the poset 
illustrated in \ref{f:PartialOrderOnComplexCovectors}.
\end{Remark}

\begin{Theorem}
If $\Arrangement=\{H_1,\ldots,H_n\}$ is a complexified hyperplane
arrangement in $\mb C^d$, then $\OIG=\{\alpha_{\vec z}: \vec z\in\mb
C^d\}$ is an oriented interval greedoid over
$(E_\Arrangement,\F_\Arrangement)$.
\end{Theorem}

\begin{proof}
We show that $\OIG$ satisfies (OG1)--(OG4). 

(OG1) Suppose $[X]$ is a flat of $(E_\Arrangement,\F_\Arrangement)$
for some $X\in\F_\Arrangement$. Let $\vec z$ be a generic point of
$\bigcap_{x\in X} x$. The support of $\alpha_{\vec z}$ is the flat
$[Y]$ such that $Y$ is maximal among feasible sets contained in $\{h:
\alpha_{\vec z}(h)=0\}$. Since $\vec z$ is generic, this set is equal
to $\xi(X)$. It follows that $Y$ is equivalent to $X$, so
$[X]=[Y]=\supp(\alpha_{\vec z})$.

(OG2) Suppose $\alpha_{\vec z}\in\OIG$. Then $-\alpha_{\vec z} =
\alpha_{-\vec z}$ because $\sigma_i(\ell_j(-\vec z)) =
\sigma_i(-\ell_j(\vec z))$. So $-\alpha_{\vec z}\in\OIG$.

(OG3) 
Let $\alpha_{\vec x},\alpha_{\vec y}\in\OIG$.
For sufficiently small $t>0$, we have 
$\sigma_\Re(\vec u+ t\vec v) =
\sigma_\Re(\vec u) \star \sigma_\Re(\vec v)$ and
$\sigma_\Im(\vec u+ t\vec v) =
\sigma_\Im(\vec u) \star \sigma_\Im(\vec v)$.
It follows that $\alpha_{\vec x + t\vec y}=
\alpha_{\vec x} \circ \alpha_{\vec y}$
for a sufficiently small $t>0$.

(OG4) 
Let $\alpha,\beta\in\OIG$ and $h\in\SeparationSet(\alpha,\beta)$ such
that $(\alpha\circ\beta)(h)\neq1$.
Pick $\vec x,\vec y\in\mb C^d$ such that $\alpha=\alpha_{\vec x}$ and
$\beta=\alpha_{\vec y}$. 
We can assume for all $1\leq i\leq n$
that the line $t\vec x + (1-t)\vec y$, for $0<t<1$, does
not intersect $H_i$ if $\vec x$ and $\vec y$ are not both contained in
$H_i^\Re$ (otherwise perturb $\vec x$ and $\vec y$ slightly).

Since $h\in\SeparationSet(\alpha_{\vec x},\alpha_{\vec y})$, we have
$\alpha_{\vec x}(h) = -\alpha_{\vec y}(h)\in\{+,-\}$. 
Hence,
$\Re(\ell_h(\vec x))$ and $\Re(\ell_h(\vec y))$ 
or
$\Im(\ell_h(\vec x))$ and $\Im(\ell_h(\vec y))$ 
have opposite
signs, where $\ell_h$ is the form associated to $h$ (that is, $h =
\ker(\ell_h)$ or $h=\ker(\ell_h)^\Re$). 
So there exists $0<t<1$ such that the real part (or imaginary part) of
$\ell_h(t\vec x + (1-t)\vec y)$ is zero. Let $\gamma = \alpha_{t\vec x
+ (1-t)\vec y}$. Then $\gamma(h) = 0$.


Let $e\notin\SeparationSet(\alpha_{\vec x},\alpha_{\vec y})$
and $(\alpha_{\vec x}\circ\alpha_{\vec y})(e)\neq1$.
Suppose first that $e=H_i$ for some $i$.
Then $\Im(\ell_i(\vec x))=0=\Im(\ell_i(\vec y))$, for otherwise
$(\alpha_{\vec x}\circ\alpha_{\vec y})(e)=1$.
This implies that 
$\Im(t\ell_i(\vec x)+(1-t)\ell_i(\vec y))=0$,
so 
$\gamma(e)=\sigma_\Re(\ell_i(t\vec x + (1-t)\vec y))$ is
the sign of
\begin{align*}
\Re\left(\ell_i(t\vec x + (1-t)\vec y)\right)
= t \Re\left(\ell_i(\vec x)\right) + (1-t)\Re\left(\ell_i(\vec y)\right).
\end{align*}
Since both of the coefficients $t$ and $(1-t)$ are positive and since
$\Re\left(\ell_i(\vec x)\right)$ and $\Re\left(\ell_i(\vec y)\right)$
are not of opposite signs, it follows that 
$\gamma(e)$ is the sign of
$\Re(\ell_i(\vec x))$ if it is nonzero and the sign of
$\Re(\ell_i(\vec y))$ otherwise. This is precisely 
$(\alpha_{\vec x}\circ\alpha_{\vec y})(e)$.
Similarly, 
if $e=H_i^\Re\notin\SeparationSet(\alpha_{\vec x},\alpha_{\vec y})$,
then $\gamma(e)=(\alpha_{\vec x}\circ\alpha_{\vec y})(e)$.
\end{proof}

\section{Restriction and contraction of oriented interval greedoids}

\subsection{Contraction}
\label{ss:Contraction}

This section introduces an operation on oriented interval greedoids
that produces an oriented interval greedoid on the contraction of the
underlying interval greedoid. We begin by studying the relationship
between an interval greedoid and its contractions.

\subsubsection{Contraction of interval greedoids}
Let $(E,\F)$ denote an interval greedoid and $\Phi$ its lattice of
flats. Recall that for $X \in \F$, the \defn{contraction} of $(E,\F)$
by $X$ is the interval greedoid with feasible sets
\begin{gather*}
\F/X = \{ Y \subseteq E\m X : Y \cup X \in \F \}
\end{gather*}
and ground set $\bigcup_{Y\in\F/X}Y$. 
We let $\Phi/X$, $\Gamma/X$ and $\xi/X$ denote the corresponding notions in
the contraction. For $Y\in\F/X$, we let $(\Phi/X)(Y)$ denote the flat
in the contraction that contains $Y$.

\begin{Proposition}
\label{p:PropertiesOfContraction}
Suppose $(E,\F)$ is an interval greedoid and $X \in \F$.
Then
\begin{enumerate}
\item $\Phi/X \cong [\hat 0, [X]] \subseteq \Phi$. 
\item If $Y \in \F/X$, then $(\Gamma/X)(Y) = \Gamma(X \cup Y)$.
\item If $Y \in \F/X$, then $(\xi/X)(Y) \sseq \xi(Y \cup X) \cap \bigcup_{Z\in\F/X} Z$.
\end{enumerate}
\end{Proposition}

\begin{proof}
(1) Define a map $\Phi/X \to [\hat 0, [X]]$ by mapping the flat containing
$Y$ (in the contraction $\F/X$) to the flat $[Y \cup X]$ of
$(E,\F)$. The fact that this map is well-defined follows from the
identity: $(\F/X)/Y = \F/(X \cup Y)$ for $X \in \F$ and $Y
\in \F/X$. This identity also implies that the map is injective. It
remains to show that the map is surjective. Let $[Z] \leq [X]$.
Then $\xi(X) \subseteq \xi(Z)$. Hence, there exists $Z'$ containing
$X$ with $Z'$ maximal among the feasible sets contained in $\xi(Z)$.
Therefore, $Z' \m X \in \F/X$, and $Z' \m X$ maps to the flat
containing $(Z' \m X) \cup X = Z'$, which is $[Z]$ by
\ref{l:MaximalSubsetsAreEquivalent}.

(2) Suppose $x \in (\Gamma/X)(Y)$. Then $Y \cup x \in \F/X$. So $(X \cup Y)
\cup x \in \F$.  That is, $x \in \Gamma(X \cup Y)$. Conversely, suppose $x
\in \Gamma(X \cup Y)$.  Then $X \cup (Y \cup x) \in \F$, and so $(Y \cup x)
\in \F/X$. That is, $x \in (\Gamma/X)(Y)$.

(3) Let $x \in (\xi/X)(Y)$. Then $x \in W$ for some $W\in\F/X$ that is
equivalent (in $\F/X)$ to $Y$. So $x\in\bigcup_{Z\in\F/X}Z$. And since the
map defined in (1) is well-defined, we have $[X\cup Y]=[X\cup W]$. Hence,
$x \in W\sseq\xi(X\cup W)=\xi(X\cup Y)$.
\end{proof}

We remark that the containment in the previous result can be proper.

\subsubsection{Contractions of oriented interval greedoids}

Let $(E,\F,\OIG)$ be an oriented interval greedoid and
$\Phi=\supp(\OIG)$ the lattice of flats of $(E,\F)$. 
For $A\in\Phi$, let
\begin{align*}
 \OIG_{\leq A} = \{ \alpha \in \OIG : \supp(\alpha) \leq A \}.
\end{align*}
Then $\OIG_{\leq A}$ is a subsemigroup of $\OIG$. We'll show that it is
isomorphic to an oriented interval greedoid over the contraction of
$(E,\F)$ by $X\in\F$, where $A=[X]$.

Let $\alpha$ be a covector of $(E,\F)$ with $\supp(\alpha)\leq[X]$. By
definition of the partial order, there exists $Y\in\F/X$ such that
$\supp(\alpha)=[X\cup Y]$. Therefore, $Y$ is a feasible set in the
contracted interval greedoid and so it makes sense to talk about its flat
$(\Phi/X)(Y)$. By restricting $\alpha$ to the subset $(\Gamma/X)(Y)$, we
get a signed flat $((\Phi/X)(Y),\alpha|_{(\Gamma/X)(Y)})$ of the contracted
interval greedoid. We denote the covector of this signed flat by
$\con_{X}(\alpha)$. Then,
\begin{align}
\label{e:CovectorOfContraction}
\con_X(\alpha)(e) =
\begin{cases}
0, & \text{if } e \in (\xi/X)(Y), \\
\alpha(e), & \text{if } e \in (\Gamma/X)(Y), \\
1, & \text{otherwise.} 
\end{cases}
\end{align}
It follows from \ref{p:PropertiesOfContraction} that if
$\con_X(\alpha)(e)\neq 1$, then $\con_X(\alpha)(e) = \alpha(e)$.

\begin{Lemma}
\label{l:ContractedCovectorsAreCovectors}
Suppose $(E,\F)$ is an interval greedoid and let $X \in \F$. Let $\alpha$
and $\beta$ be covectors of $(E,\F)$ with
$\supp(\alpha),\supp(\beta)\leq[X]$. 
\begin{enumerate}
\item
$(\supp/X)(\con_X(\alpha)) = (\Phi/X)(Y)$ and 
$(\supp/X)(\con_X(\beta)) = (\Phi/X)(Z)$,
where $Y,Z\in\F/X$ satisfy 
$[X\cup Y]=\supp(\alpha)$ and $[X\cup Z]=\supp(\beta)$.
\item
$\con_X(\alpha)\circ\con_X(\beta) = \con_X(\alpha\circ\beta)$.
\end{enumerate}
\end{Lemma}

\begin{proof}
(1) Since $\con_X(\alpha)$ is the covector of the signed flat $(
(\Phi/X)(Y), \alpha|_{(\Gamma/X)(Y)} )$, it follows from the definition
of $\supp/X$ that $(\supp/X)(\con_X(\alpha)) = (\Phi/X)(Y)$.

(2) We first argue that the supports of the two elements are the same. It
follows from the definition of $\circ$ that the support of
$\con_X(\alpha)\circ\con_X(\beta)$ is the join of their supports, so it is
$(\Phi/X)(Y)\vee(\Phi/X)(Z)$ by (1). Under the isomorphism $\Phi/X \cong
[\hat0,[X]]$, this corresponds to $[X\cup Y]\vee[X\cup Z]$, which we can
express as $[X\cup W]$ for some $W\in\F/X$. Hence,
$(\Phi/X)(Y)\vee(\Phi/X)(Z) = (\Phi/X)(W)$. Note that $[X\cup W]$ is also
the support of $\alpha\circ\beta$, so (1) implies that
$(\supp/X)(\con_X(\alpha\circ\beta))=(\Phi/X)(W)$.

Since both $\con_X(\alpha)\circ\con_X(\beta)$ and
$\con_X(\alpha\circ\beta)$ are covectors of support $(\Phi/X)(W)$, to show
that they are equal it suffices to show that they agree on $(\Gamma/X)(W)$.
Let $e\in(\Gamma/X)(W)$. Then,
\begin{align*}
\left(\con_X(\alpha)\circ\con_X(\beta)\right)(e) =
\begin{cases}
\con_X(\beta)(e), & \text{if } \con_X(\beta)(e) > \con_X(\alpha)(e), \\
\con_X(\alpha)(e), & \text{otherwise}.
\end{cases}
\end{align*}
Since $(\con_X(\alpha)\circ\con_X(\beta))(e)\neq1$, it follows that neither
$\con_X(\alpha)(e)$ nor $\con_X(\beta)(e)$ is 1. Hence, $\con_X(\alpha)(e)
= \alpha(e)$ and $\con_X(\beta)(e) = \beta(e)$ (see the sentence
following \ref{e:CovectorOfContraction}). Therefore,
\begin{align*}
\left(\con_X(\alpha)\circ\con_X(\beta)\right)(e) =
\begin{cases}
\beta(e), & \text{if } \beta(e) > \alpha(e), \\
\alpha(e), & \text{otherwise}.
\end{cases}
\end{align*}
This is precisely $(\alpha\circ\beta)(e)$, which is
$\con_X(\alpha\circ\beta)(e)$ by \ref{e:CovectorOfContraction}.
\end{proof}

\begin{Proposition}
\label{p:OIGContraction}
Let $(E,\F, \OIG)$ denote an oriented interval greedoid
and let $X \in \F$. Then 
\begin{gather*}
\OIG/X = \{ \con_X(\alpha) : \alpha \in \OIG \text{ and } \supp(\alpha) \leq [X] \}
\end{gather*}
defines an oriented interval greedoid over the contraction of $(E,\F)$ by $X$.
\end{Proposition}

\begin{proof}
(OG1). Let $A \in \Phi/X$. Then $A = (\Phi/X)(Y)$ for some $Y\in\F/X$, and
so $[Y \cup X] \in \Phi$. Since $\OIG$ satisfies (OG1), there exists
$\alpha \in \OIG$ with $\supp(\alpha) = [Y \cup X] \leq [X]$. Then
$\con_X(\alpha) \in \OIG/X$ and $(\supp/X)(\con_X(\alpha))=(\Phi/X)(Y)=A$
by \ref{l:ContractedCovectorsAreCovectors}.

(OG2) Suppose $\nu \in \OIG/X$. Then there exists some $\beta \in \OIG$
such that $\supp(\beta) \leq [X]$ and $\con_X(\beta) = \nu$. Then $-\beta
\in \OIG$ by (OG2), and so $-\nu = \con_X(-\beta) \in \OIG/X$.

(OG3) Suppose $\con_X(\alpha)$ and $\con_X(\beta)$ are in $\OIG/X$.
Then $\alpha \circ \beta \in \OIG$, by (OG3), and $\supp(\alpha \circ
\beta) = \supp(\alpha) \vee \supp(\beta) \leq [X]$.
Therefore, $\con_X(\alpha \circ \beta) \in \OIG/X$. By 
\ref{l:ContractedCovectorsAreCovectors}, 
$\con_X(\alpha\circ\beta) = \con_X(\alpha)\circ\con_X(\beta)$, so
$\con_X(\alpha)\circ\con_X(\beta) \in \OIG/X$.

(OG4) Suppose $\con_X(\alpha), \con_X(\beta) \in \OIG/X$, and let $x \in
(\SeparationSet/X)(\con_X(\alpha),\con_X(\beta))$ such that
$(\con_X(\alpha)\circ\con_X(\beta))(x)\neq1$. 

Since $\con_X(\alpha)(x)=-\con_X(\beta)(x)\in\{+,-\}$, it follows from
\ref{e:CovectorOfContraction} that $\alpha(x)=-\beta(x)\in\{+,-\}$. Hence,
$x\in\SeparationSet(\alpha,\beta)$. Since
$\con_X(\alpha\circ\beta)=\con_X(\alpha)\circ\con_X(\beta)$, it follows
that $\con_X(\alpha\circ\beta)(x)\neq1$, which implies that
$(\alpha\circ\beta)(x)\neq1$. Therefore, (OG4) applies to $\alpha,\beta$
and $x$ to guarantee the existence of $\gamma\in\OIG$ satisfying
$\gamma(x)=0$ and for all $y\notin\SeparationSet(\alpha,\beta)$,
if $(\alpha\circ\beta)(y)\neq1$, 
    then $\gamma(y) = (\alpha\circ\beta)(y) = (\beta\circ\alpha)(y)$.
We claim that $\con_X(\gamma)$ satisfies the conditions of (OG4) for
$\OIG/X$.

We first show that $\supp(\gamma)\leq\supp(\alpha\circ\beta)$. Indeed, if
$(\alpha\circ\beta)(y)=0$, then $\alpha(y)=0$ and $\beta(y)=0$, so
$y\notin\SeparationSet(\alpha,\beta)$. We conclude from (OG4) that $\gamma(y)
= (\alpha\circ\beta)(y) = 0$.

Next we argue that $\con_X(\gamma)(x) = 0$. Since
$\supp(\gamma)\leq\supp(\alpha\circ\beta)$, it follows that
$(\supp/X)(\con_X(\gamma)) \leq (\supp/X)(\con_X(\alpha\circ\beta))$. Then
\ref{p:PropertiesOfContinuations} and the assumption that
$\con_X(\alpha\circ\beta)(x)\neq1$ implies that $\con_X(\gamma)(x)\neq1$.
If $\con_X(\gamma)(x)\in\{+,-\}$, then $\gamma(x)\in\{+,-\}$ contradicting
the fact that $\gamma(x)=0$. Therefore, $\con_X(\gamma)(x)=0$.

Now let $y\in\bigcup_{Z\in\F/X}Z$ with
$y\notin(\SeparationSet/X)(\con_X(\alpha),\con_X(\beta))$. We claim that
$y\notin\SeparationSet(\alpha,\beta)$. If
$y\in\SeparationSet(\alpha,\beta)$, then $\alpha(y)=-\beta(y)\in\{+,-\}$,
and so $\con_X(\alpha)(y)=\alpha(y)=-\beta(y)=-\con_X(\beta)(y)\in\{+,-\}$
by \ref{e:CovectorOfContraction}, a contradiction.

Now suppose that $\con_X(\alpha\circ\beta)(y)\neq1$. As above,
\ref{p:PropertiesOfContinuations} implies that $\con_X(\gamma)(y)\neq1$.
Then the sentence following \ref{e:CovectorOfContraction} implies that
$\con_X(\alpha\circ\beta)(y)=(\alpha\circ\beta)(y)$ and that
$\con_X(\gamma)(y)=\gamma(y)$. Hence, $(\alpha\circ\beta)(y)\neq1$, so
$\gamma(y)=(\alpha\circ\beta)(y)=(\beta\circ\alpha)(y)$ by (OG4).
Therefore, $\con_X(\gamma)(y) = 
(\con_X(\alpha)\circ\con_X(\beta))(y) =
(\con_X(\beta)\circ\con_X(\alpha))(y)$.
\end{proof}

The following result identifies $\OIG/X$ with a subsemigroup of
$\OIG$.

\begin{Proposition}
\label{p:ContractionSemigroup}
Let $(E,\F,\OIG)$ denote an oriented interval greedoid and let $X\in\F$. 
Then there is a semigroup isomorphism
\begin{gather*}
\OIG_{\leq[X]}\cong\OIG/X
\end{gather*}
given by mapping
$\alpha\in\OIG$ with $\supp(\alpha)\leq[X]$ to $\con_X(\alpha)$. 
\end{Proposition}

\begin{proof}
\ref{l:ContractedCovectorsAreCovectors} shows this is a
semigroup morphism. The morphism is surjective by definition of
$\OIG/X$. It remains to show that the morphism is injective.  Suppose
$\con_X(\alpha)= \con_X(\beta)$.  Then $\supp(\alpha)=\supp(\beta)$,
which can be written as $[X\cup Y]$.  Now $(\Gamma/X)(Y)=\Gamma(X\cup Y)$, 
so $\alpha$ and $\beta$ agree on $\Gamma(X\cup Y)$, and they each are zero
on exactly $\xi([X\cup Y])$, so $\alpha$ and $\beta$ agree, as desired.  
%
%
\end{proof}

\subsection{Restriction}
\label{ss:Restriction}

We introduce a restriction operation for an oriented interval greedoid
$(E,\F,\OIG)$ that produces an oriented interval greedoid on a
restriction of the interval greedoid $(E,\F)$. We begin by recalling
restriction for interval greedoids.

\subsubsection{Restriction of an interval greedoid}
\label{sss:RestrictionofIGs}

Let $(E,\F)$ denote an interval greedoid and $\Phi$ its lattice of
flats. If $W \subseteq E$ is an arbitrary subset, then the
\defn{restriction} of $(E,\F)$ to $W$ is the interval greedoid 
$(W,\F|_W)$, where
\begin{gather*}
\F|_W = \{ X \in \F : X \subseteq W \}.
\end{gather*}

To distinguish between objects defined for $(E,\F)$ and $(W,\F|_W)$,
we take the following convention.  If $\Xi$ is an object defined for
$(E, \F)$ (for example, its lattice of flats $\Phi$, the set of
continuations $\Gamma$), then $\Xi|_W$ will denote the corresponding
object defined for $(W,\F|_W)$ (for example, $\Phi|_W$, $\Gamma|_W$).

There is a map $\Phi \to \Phi|_W$ that maps a flat $C\in\Phi$ onto the flat
$\mu|_W(W \cap \xi(C))$. We denote the image of $C$ by $C|_W$. Note that if
$Y\in\F|_W$, then $Y\in\F$ and the image of $[Y]\in\Phi$ under $\Phi \to
\Phi|_W$ is the flat in $\Phi|_W$ that contains $Y$, which by our above
convention is denoted by $[Y]|_W$.

\begin{Lemma}
\label{l:IntervalGreedoidsAndRestriction}
Suppose $(E,\F)$ is an interval greedoid and 
let $W \subseteq E$. 
\begin{enumerate}
\item 
If $Y \in \F|_W$, then $\Gamma|_W(Y) = W \cap \Gamma(Y)$.
\item 
If $Y \in \F|_W$ and $\xi(Y) \subseteq W$, then $\xi|_W(Y) = \xi(Y)$.
\item 
If $A \in \Phi$, 
  then $\Gamma|_W(A|_W)\subseteq W\cap\Gamma(A)$.
\item 
If $A \in \Phi$, 
  then $\xi|_W(A|_W) \subseteq W \cap \xi(A)$.
\end{enumerate}
\end{Lemma}

\begin{proof}
(1) If $Y \in \F|_W$, then
$\Gamma_W(Y) = \{y \in W \m Y : Y \cup y \in \F\}
= W \cap \{ y \in E \m Y : Y \cup y \in \F\} = W \cap \Gamma(Y)$.

(2) Suppose $Y \in \F|_W$ and $\xi(Y) \subseteq W$. The latter
assumption implies that all feasible sets that are equivalent to $Y$
in $(E,\F)$ are contained in $W$. So they are contained in 
$\F|_W$. Moreover, they are also equivalent in $\F|_W$ since they are
all maximal among the feasible sets contained in $\xi(Y)$ (see 
\ref{l:MaximalSubsetsAreEquivalent}). 

(3) Let $Z \in A|_W$ and let $x \in \Gamma|_W(A|_W) =
\Gamma|_W(Z)$. By (1), $x\in W\cap\Gamma(Z)$.
Since $A|_W=\mu|_W(W\cap\xi(A))$, there exists $Y\in\F$ containing $Z$
that is maximal among the feasible sets contained in $\xi(A)$.
Thus, $[Z]\geq[Y]=A$, and by \ref{p:PropertiesOfContinuations}, 
\begin{align*}
\Gamma|_W(Z) =
(\Gamma(Z)\cap W)\subseteq(\Gamma(A)\cap W)\cup(\xi(A)\cap W).
\end{align*}
If $x \in W \cap \xi(A)$, then $Z \cup x \in W \cap \xi(A)$,
contradicting that $Z$ is maximal among the feasible sets contained in
$W \cap \xi(A)$. Therefore, $x \in W \cap \Gamma(A)$.

(4) By definition $A|_W = \mu|_W(W \cap \xi(A))$, so the sets contained in
$A|_W$ are the sets that are maximal among the feasible sets contained in
$W \cap \xi(A)$. Let $D$ be a maximal feasible set in $W\cap\xi(A)$, and
let $C$ be a set in $W$ that is equivalent to $D$ in the restriction. We
want to show that $C$ is contained in $\xi(A)$.

Since $C$ and $D$ are equivalent in the restriction, they have the same
continuations inside $W$. Let $x$ be a continuation of $C$ with $x \notin
W$. Then $D$ can be augmented from $C\cup x$, and clearly $D$ can't be
augmented from $C$, so it can be augmented by $x$. Thus $\Gamma(C)$
contains $\Gamma(D)$, and the converse is also true.  So $C$ and $D$ have
the same continuations in the original interval greedoid, and therefore are
equivalent. In particular, $C$ is in $\xi(A)$ as well.
\end{proof}

\begin{Remark}
\label{r:IGsAndRestriction}
The inclusions in (3) and (4) can be proper, as can be seen in the
following example.
Let $E=\{a,b,c\}$ and $\F=\{\O,\{a\},\{a,b\},\{a,c\}\}$. 
Then $(E,\F)$ is an interval greedoid.
If $W=\{b,c\}$, then $\F|_W = \{\O\}$, so
\begin{gather*}
\Gamma|_W([\{a\}]|_W)=\Gamma|_W(\O)
  =\O\subsetneq\Gamma([\{a\}])\cap W=\{b,c\}, \\
\xi|_W([\{a,b\}]|_W)=\xi|_W(\O)
 =\O\subsetneq\xi([\{a,b\}])\cap W=\{b,c\}.
\end{gather*}
\end{Remark}

\begin{Proposition}
\label{p:CanonicalSurjectionForFlatsInRestriction}
The map $\Phi \to \Phi|_W$ defined by $C \mapsto C|_W = \mu|_W(W \cap
\xi(C))$ for all $C \in \Phi$ is order-preserving, surjective and preserves
joins: $(A \vee B)|_W = A|_W \vee B|_W$ for all $A, B \in \Phi$.
\end{Proposition}

\begin{proof}
The mapping is order-preserving since $\xi$ and $\mu|_W$ are
order-reversing. The map is surjective since if $[Y]|_W \in \Phi|_W$,
then it follows that $Y \in \F$ and that $Y$ is maximal among the
feasible subsets contained in $W \cap \xi([Y])$. So, $[Y] \mapsto
[Y]|_W$ under this mapping.

Since $A, B \leq A \vee B$, and since the map is order-preserving,
$A|_W \vee B|_W \leq (A \vee B)|_W$. Since $A|_W \vee
B|_W \in \Phi|_W$, there exists $Z \in \F|_W$ such that $A|_W \vee B|_W =
[Z]|_W$. Then $Z$ is maximal among the feasible sets contained in
$\xi|_W(A|_W) \cap \xi|_W(B|_W)$ by definition of $\vee$
(\ref{p:PhiIsSemimodularLattice}). Since $A|_W$ is the
collection of sets that are maximal among the feasible sets contained
in $W \cap \xi(A)$, it follows that $\xi|_W(A|_W) \subseteq W \cap
\xi(A)$.  Therefore, $\xi|_W(A|_W) \cap \xi|_W(B|_W) \subseteq W \cap
\xi(A) \cap \xi(B)$. So there exists $Y$ containing $Z$ that is
maximal among the feasible sets contained in $W \cap \xi(A) \cap
\xi(B)$. We now argue that $Y$ is maximal among the feasible sets
contained in $W \cap \xi(A \vee B)$. There exists $U$ containing $Y$
that is maximal among the feasible sets contained in $\xi(A) \cap
\xi(B)$. Thus, $[Y] \geq [U] = A \vee B$.
This implies $Y \subseteq \xi(Y) \subseteq \xi(A \vee B)$ since $\xi$
is order-reversing. So $Y \subseteq W \cap \xi(A \vee B)$.  Since
$\xi(A \vee B) \subseteq \xi(A) \cap \xi(B)$, it follows that $Y$ is
maximal among the feasible sets contained in $W \cap \xi(A \vee B)$.
Therefore, $(A \vee B)|_W = [Y]|_W$. Since $Y \supseteq Z$, we have
$[Y]|_W \leq [Z]|_W$. Thus, $(A \vee B)|_W \leq A|_W \vee B|_W$. Therefore,
$(A \vee B)|_W = A|_W \vee B|_W$.  
\end{proof}

Since $\F|_W \subseteq \F$, there is also a map in the reverse
direction $\Phi|_W\to\Phi$, defined by $A \mapsto [Y]$ for any $Y\in
A$. \ref{l:MaximalSubsetsAreEquivalent} implies the map is well-defined,
and the identity
$(\F|_W)/Y = (\F/Y)|_W = \{X \subseteq W \m Y : X \cup Y \in \F\}$ 
implies the map is injective.
It is order-preserving and its image is contained in the interval
$[[X],\hat1]$, where $X$ is maximal among the feasible sets contained
in $W$.

Unlike for matroids, for an arbitrary interval greedoid, the lattice of
flats $\Phi|_W$ is not, in general, an interval of $\Phi$. However, if
$W\supseteq\xi(X)$, where $X$ is maximal among the feasible subsets
contained in $W$, then $\Phi|_W\cong[[X],\hat1]\subseteq \Phi$. (This
is obtained by considering the compositions of the maps defined
above.) 

\subsubsection{Restricting Covectors} 

Let $W \subseteq E$ and let $\alpha$ be a covector of $(E,\F)$.  Let
$A = \supp(\alpha)$. It follows from
\ref{l:IntervalGreedoidsAndRestriction} that $\Gamma|_W(A|_W)
\subseteq \Gamma(A)$, so $(A|_W, \alpha|_{\Gamma|_W(A|_W))})$ is a
signed flat of $(W, \F|_W)$. Let $\res_W(\alpha)$ denote the covector
of this signed flat:
\begin{gather}
\label{e:RestrictedCovectors}
\res_W(\alpha)(w) =
\begin{cases}
0, & \text{if } w \in \xi|_W(A|_W),\\
\alpha(w), & \text{if } w \in \Gamma|_W(A|_W), \\
1, & \text{otherwise},
\end{cases}
\end{gather}
for all $w\in W$. 
Observe that by construction
$\supp|_W(\res_W(\alpha))=\supp(\alpha)|_W$.
Also note that by \ref{l:IntervalGreedoidsAndRestriction}, if
$\res_W(\alpha)(w)\neq\alpha(w)$, then $\res_W(\alpha)(w)=1$. 

\begin{Example}[Antimatroid from three colinear points]
Let $(E,\F,\OIG)$ be the oriented interval greedoid arising from
the convex geometry of three colinear points, $x,y,z$ with $y$ between
$x,z$.  Let $W=\{x,y\}$.  The covectors of $\OIG|_W$ are: $(\pm,1)$;
$(0,\pm)$; $(0,0)$.  For $\alpha\in \OIG$, if $\alpha(z)\ne0$, then 
$\res_W(\alpha)$ equals the restriction of $\alpha$ to $W$.  However, 
if, for example, $\alpha=(+,+,0)$, then $\res_W(\alpha)=(+,1)$.  
\end{Example}

As we have just seen in an example, $\res_W(\alpha)$ cannot 
necessarily be obtained by restricting $\alpha$ to $W$.  The following
proposition sheds more light on this.

\begin{Proposition}
\label{p:RestrictionOfCovectors}
Suppose $(E,\F)$ is an interval greedoid and let $W \subseteq E$.
\begin{enumerate}
\item 
If $\alpha$ is a covector of $(E,\F)$ and $A = \supp(\alpha)$, 
 then $\res_W(\alpha)=\alpha|_W$ if and only if 
  $\xi|_W(A|_W) = W \cap \xi(A)$ and 
  $\Gamma|_W(A|_W) = W \cap \Gamma(A)$.
\item 
If $\alpha$ and $\beta$ are covectors of $(E,\F)$, then 
  $\res_W(\alpha\circ\beta)=\res_W(\alpha)\circ\res_W(\beta)$.
\end{enumerate}
\end{Proposition}

\begin{proof}
(1) This is obvious from the definitions.
%

(2) 
Let $A=\supp(\alpha)$ and $B=\supp(\beta)$. 
By \ref{sss:RestrictionofIGs}, $A|_W \vee B|_W = (A\vee B)|_W$. 
Thus, $\res_W(\alpha)\circ\res_W(\beta)$ and $\res_W(\alpha\circ\beta)$
have the same support.  It is therefore clear that they coincide.
\end{proof}

\subsubsection{Restriction of an oriented interval greedoid}

The following results shows that the covectors obtained by restricting
the covectors of an oriented interval greedoid $(E,\F,\OIG)$ satisfy
the first three axioms for an oriented interval greedoid.

\begin{Proposition}
\label{p:RestrictedOIGs}
Suppose $(E,\F,\OIG)$ is an oriented interval greedoid.
If $W \subseteq E$, then $(W,\F|_W,\OIG|_W)$ satisfies
(OG1), (OG2) and (OG3), where
\begin{gather*}
\OIG|_W = \{\res_W(\alpha):\alpha\in\OIG\}.
\end{gather*}
\end{Proposition}
\begin{proof}
By construction, we have that $\OIG|_W$ is a collection of covectors
of $(W,\F|_W)$ and that 
$\supp|_W(\res_W(\alpha))=\supp(\alpha)|_W\in\Phi|_W$.

(OG1) Let $A \in \Phi|_W$ and let $Y \in A$. Then $[Y] \in \Phi$ and so
there exists $\alpha\in\OIG$ with $\supp(\alpha)=[Y]$. So
$\res_W(\alpha)\in\OIG|_W$ and
$\supp|_W(\res_W(\alpha))=\supp(\alpha)|_W = A|_W$. 

(OG2) If $\res_W(\alpha)\in\OIG|_W$, then $-\alpha\in\OIG$.
Hence, $-\res_W(\alpha)=\res_W(-\alpha)\in\OIG|_W$.

(OG3) Suppose $\res_W(\alpha), \res_W(\beta) \in \OIG|_W$ and let $A =
\supp(\alpha)$ and $B = \supp(\beta)$. Then $\alpha \circ \beta \in
\OIG$, and so $\res_W(\alpha \circ \beta) \in
\OIG|_W$. By \ref{p:RestrictionOfCovectors},
$\res_W(\alpha\circ\beta) = \res_W(\alpha) \circ \res_W(\beta)$, so 
$\res_W(\alpha)\circ\res_W(\beta) \in \OIG|_W$. 
\end{proof}

Although (OG4) may not hold for an arbitrary restriction, it does hold for
certain restrictions, so we get an oriented interval greedoid. 

\begin{Theorem}
\label{c:ConditionsForRestrictedOIG}
Suppose $(E,\F,\OIG)$ is an oriented interval greedoid and 
let $W\subseteq E$.
If $\res_W(\alpha) = \alpha|_W$ for all $\alpha\in\OIG$,
then $(W,\F|_W,\OIG|_W)$ is an oriented interval greedoid.
\end{Theorem}

\begin{proof}
(OG1)--(OG3) hold by \ref{p:RestrictedOIGs}. 
The assumption that $\res_W(\alpha) = \alpha|_W$ for all $\alpha\in\OIG$
means that (OG4) for $\OIG$ implies (OG4) for $(W,\F|_W,\OIG|_W)$.
\end{proof}

\ref{ss:RestrictionToGamma} 
and 
\ref{ss:RestrictionToXi} 
describe
restriction to two particular types of subsets of $E$.

\subsection{Restriction to $\Gamma(\O)$} 
\label{ss:RestrictionToGamma}
In this section we treat restriction to $\Gamma(\O)$.

\begin{Proposition}
\label{p:RestrictionToUnderlyingOrientedMatroid}
Suppose $(E,\F,\OIG)$ is an oriented interval greedoid.
Then the restriction $(\Gamma(\O),\F\restd, \OIG\restd)$ is an oriented 
matroid, and
\begin{gather*}
	\OIG\restd = \{ \alpha\restd : \alpha\in\OIG \}.
\end{gather*}
\end{Proposition}
\begin{proof}
Since $(\Gamma(\O),\F|_{\Gamma(\O)})$ is a matroid, it will follow
from \ref{p:OMAndOIGs} that $(\Gamma(\O),\F\restd,\OIG\restd)$ is an
oriented matroid once we show that it is an oriented interval
greedoid. By \ref{p:RestrictionOfCovectors} and
\ref{c:ConditionsForRestrictedOIG}, we need only show that
$\Gamma|_{\Gamma(\O)}(A|_{\Gamma(\O)})=\Gamma(\O)\cap\Gamma(A)$ and
$\xi|_{\Gamma(\O)}(A|_{\Gamma(\O)})=\Gamma(\O)\cap\xi(A)$. 

Suppose $x \in \Gamma(\O) \cap \Gamma(A)$. Let $Y\in A|_{\Gamma(\O)}$. 
Then $Y\sseq\Gamma(\O)\cap\xi(A)$, so there exists $Z\supseteq Y$ such that
$Z$ is maximal among the sets in $\F$ contained in $\xi(A)$. By
\ref{l:MaximalSubsetsAreEquivalent}, $Z\in A$, so
$\Gamma(A)=\Gamma(Z)$. Hence, $Z\cup x\in\F$ since
$x\in\Gamma(A)$. Also, $\{x\}\in\F$ since $x\in\Gamma(\O)$. Therefore,
(IG3) applied to $\O\sseq Y\sseq Z$ implies
$Y\cup x\in\F$. Since $Y\cup x\sseq\Gamma(\O)$, we have $Y\cup
x\in\F|_{\Gamma(\O)}$. So $x\in\Gamma|_{\Gamma(\O)}(A|_{\Gamma(\O)})$.
This establishes one inclusion. The reverse inclusion follows from
\ref{l:IntervalGreedoidsAndRestriction}.

It remains to show that $\xi|_{\Gamma(\O)}(A|_{\Gamma(\O)}) =
\Gamma(\O) \cap \xi(A)$. Suppose $x \in \Gamma(\O) \cap \xi(A)$. Then
$\{x\} \in \F$ since $x \in \Gamma(\O)$. Therefore, there exists $Y$
containing $x$ such that $Y$ is maximal among the feasible sets
contained in $\Gamma(\O) \cap \xi(A)$. Then $Y\in A|_{\Gamma(\O)}$.
Therefore, $Y \subseteq \xi|_{\Gamma(\O)}(A|_{\Gamma(\O)})$, and so $x
\in \xi|_{\Gamma(\O)}(A|_{\Gamma(\O)})$. This, combined with
\ref{l:IntervalGreedoidsAndRestriction}, establishes the
equality $\xi|_{\Gamma(\O)}(A|_{\Gamma(\O)})=\Gamma(\O)\cap\xi(A)$. 
\end{proof}

\subsection{Restriction to $\xi(X)$}

To simplify notation, we write $\xi(X)$ for $\xi([X])$
for any feasible set $X\in\F$.

\label{ss:RestrictionToXi}
We show that restriction to $\xi(X)$ for $X\in\F$ produces an
oriented interval greedoid $(\xi(X),\F|_{\xi(X)},\OIG|_{\xi(X)})$
and that there is a semigroup isomorphism
\begin{gather*}
\OIG|_{\xi(X)}\buildrel\cong\over\longrightarrow\OIG_{\geq\alpha} 
 =\{\beta\in\OIG:\beta\geq\alpha\}
 =\{\alpha\circ\beta:\beta\in\OIG\},
\end{gather*}
where $\alpha$ is any covector with $\supp(\alpha) = [X]$.

\begin{Lemma}
\label{l:XiAndGammaInRestrictionToXi}
Suppose $(E,\F)$ is an interval greedoid and let $X\in\F$
and $A\in\Phi$.
\begin{enumerate}
\item $A|_{\xi(X)} = A \vee [X]$.
\item 
$\xi|_{\xi(X)}(A|_{\xi(X)}) 
  = \xi(A \vee [X]) = \xi(A \vee [X]) \cap \xi(X)$.  
\item $\Gamma|_{\xi(X)}(A|_{\xi(X)}) = \Gamma(A \vee [X])\cap\xi(X)$.  
\end{enumerate}
\end{Lemma}
\begin{proof}
(1) $A|_{\xi(X)}$ is $\mu|_{\xi(X)}(\xi(X) \cap \xi(A))$, the flat
that consists of the sets that are maximal among the feasible sets
contained in $\xi(X) \cap \xi(A)$. By \ref{p:PhiIsSemimodularLattice},
this is $A \vee [X]$. 

(2) This follows from (1) since all feasible sets in $A \vee
[X]$ are contained in $\xi(X)$. 

(3) Write $A \vee [X] =
[Y]$ for some $Y \in \F$. Then $[X] \leq [Y]$, so $Y \in \F|_{\xi(X)}$
since $\xi(Y) \subseteq \xi(X)$. Thus, $\Gamma|_{\xi(X)}(A|_{\xi(X)})
=  \Gamma|_{\xi(X)}(Y) = \{y \in \xi(X)
\m Y : Y \cup y \in \F \} = \xi(X) \cap \Gamma(Y) = \xi(X) \cap
\Gamma(A \vee [X])$. 
\end{proof}

\begin{Lemma}
\label{l:RestrictionOfAnOIG}
Suppose $(E,\F,\OIG)$ is an oriented interval greedoid
and let $X\in\F$.
\begin{align*}
\OIG|_{\xi(X)} 
 &= \{\res_{\xi(X)}(\alpha) :
     \alpha\in\OIG \text{ and } \supp(\alpha)\geq[X]\}\\
 &= \{\alpha|_{\xi(X)} :
     \alpha\in\OIG \text{ and } \supp(\alpha)\geq[X]\}.
\end{align*}
\end{Lemma}
\begin{proof}
We begin by proving the first equality.
We show that if $\beta\in\OIG$, then there exists $\alpha\in\OIG$ 
with $\supp(\alpha)\geq[X]$ and 
$\res_{\xi(X)}(\alpha)=\res_{\xi(X)}(\beta)$. 
Let $\beta\in\OIG$ and let $B=\supp(\beta)$. Since $\OIG$ 
satisfies (OG1), there exists $\gamma\in\OIG$ such that 
$\supp(\gamma)=[X]$. 
Since $\supp(\gamma\circ\beta)=[X]\vee B$ and
$B|_{\xi(X)}=B\vee[X]=(B\vee[X])|_{\xi(X)}$ by
\ref{l:XiAndGammaInRestrictionToXi},
\begin{gather*}
\res_{\xi(X)}(\gamma \circ \beta)=
\begin{cases}
0, & 
  \text{if } x \in \xi|_{\xi(X)}( B|_{\xi(X)}),\\
(\gamma \circ \beta)(x), & 
  \text{if } x \in \Gamma|_{\xi(X)}(B|_{\xi(X)}), \\
1, & \text{otherwise}.
\end{cases}
\end{gather*}
Therefore, $\res_{\xi(X)}(\gamma\circ\beta)=\res_{\xi(X)}(\beta)$ if
and only if $(\gamma\circ\beta)(x)=\beta(x)$ for
$x\in\Gamma|_{\xi(X)}(B|_{\xi(X)})$. 
So suppose $x\in\Gamma|_{\xi(X)}(B|_{\xi(X)})$. 
By \ref{l:XiAndGammaInRestrictionToXi},
$x\in\xi(X)\cap\Gamma(B\vee[X])$, and
by \ref{p:PropertiesOfContinuations}, 
$x\in\xi(X)\cap(\Gamma(B)\cup\Gamma(X))$. 
This implies $x\in\Gamma(B)$
because $\xi(X)\cap\Gamma(X)=\O$. Therefore, 
$(\gamma\circ\beta)(x)=\beta(x)$, and so 
$\res_{\xi(X)}(\gamma\circ\beta)(x)
=\beta(x)=\res_{\xi(X)}(\beta)(x)$.

We now prove the second equality.
Let $\res_{\xi(X)}(\alpha)$ such that $\alpha\in\OIG$ and
$A=\supp(\alpha)\geq[X]$. By \ref{l:XiAndGammaInRestrictionToXi}
we have $A|_{\xi(X)}=A\vee[X]=A$, 
$\xi|_{\xi(X)}(A|_{\xi(X)})=\xi(X)\cap\xi(A)$ and
$\Gamma|_{\xi(X)}(A|_{\xi(X)})=\xi(X)\cap\Gamma(A)$. Therefore, by
\ref{p:RestrictionOfCovectors},
$\res_{\xi(X)}(\alpha)=\alpha|_{\xi(X)}$ for all $\alpha\in\OIG$ such
that $\supp(\alpha)\geq[X]$.
\end{proof}

\begin{Theorem}
Let $(E,\F,\OIG)$ denote an oriented interval greedoid and let $X \in
\F$. Then $(\xi(X), \F|_{\xi(X)}, \OIG|_{\xi(X)})$ is an oriented
interval greedoid.
\end{Theorem}
\begin{proof}
By \ref{p:RestrictedOIGs} we need only show that
$\OIG|_{\xi(X)}$ satisfies (OG4). 
Let $\res_{\xi(X)}(\alpha)$ and $\res_{\xi(X)}(\beta)$ be
covectors in $\OIG|_{\xi(X)}$.
By \ref{l:RestrictionOfAnOIG}, we can assume 
$\supp(\alpha)\geq[X]$, $\supp(\beta)\geq [X]$,
$\res_{\xi(X)}(\alpha)=\alpha|_{\xi(X)}$
and $\res_{\xi(X)}(\beta)=\beta|_{\xi(X)}$.

Let
$x\in\SeparationSet(\res_{\xi(X)}(\alpha),\res_{\xi(X)}(\beta))$ with
$\res_{\xi(X)}\acb(x)\neq1$. Then $x\in\SeparationSet(\alpha,\beta)$
and $\acb(x)\neq1$.
By (OG4) applied to $\OIG$, there exists $\gamma\in\OIG$ such that
$\gamma(x)=0$ and for all $y\notin\SeparationSet(\alpha,\beta)$, if
$\acb(y)\neq1$, then $\gamma(y)=\acb(y)=\bca(y)$.

We show that $\res_{\xi(X)}(\gamma)$ satisfies the conditions of
(OG4). 
Let $A=\supp(\alpha)$, $B=\supp(\beta)$ and $C=\supp(\gamma)$. 
Observe that $C\vee[X]\leq A\vee B$: indeed, if $\acb(e)=0$, then
$\alpha(e)=0$, so $e\notin\SeparationSet(\alpha,\beta)$, which implies
that $\gamma(e)=\acb(e)=0$.

We first argue that $\res_{\xi(X)}(\gamma)(x)$ is $0$.
By construction, it is either $\gamma(x)$ or $1$.
Suppose it is $1$. Then 
$x\notin\xi|_{\xi(X)}(C|_{\xi(X)})\cup\Gamma|_{\xi(X)}(C|_{\xi(X)})
=\xi(C\vee[X])\cup\Gamma(C\vee[X])$.
By \ref{p:PropertiesOfContinuations}, since $C\vee[X]\leq A\vee B$, 
$x\notin\Gx{A\vee B}$. This implies $\acb(x)=1$, which contradicts
$\acb(x)\neq1$. Thus, $\res_{\xi(X)}(\gamma)(x)=\gamma(x)=0$.

Let
$y\notin\SeparationSet(\res_{\xi(X)}(\alpha),\res_{\xi(X)}(\beta))$.
Then $y\notin\SeparationSet(\alpha,\beta)$. Suppose
$\res_{\xi(X)}\acb(y)\neq1$. Then $\acb(y)\neq1$. 
This implies, as above, that $\res_{\xi(X)}(\gamma)(y)\neq1$.
Thus, $\res_{\xi(X)}(\gamma)(y)=\gamma(y)=\acb(y)$. 
By \ref{l:RestrictionOfAnOIG}, $\res_{\xi(X)}\acb=\acb|_{\xi(X)}$, so
$\res_{\xi(X)}(\gamma)(y)
 =\res_{\xi(X)}\acb(y)
 =(\res_{\xi(X)}(\alpha)\circ\res_{\xi(X)}(\beta))(y)$.
%
\end{proof}

The following result identifies the semigroup $\OIG|_{\xi(X)}$ with a
subsemigroup of $\OIG$.

\begin{Proposition}
\label{p:RestrictionSemigroup}
Let $(E,\F,\OIG)$ denote an oriented interval greedoid and let $X \in
\F$. Then $\res_{\xi(X)}(\beta)\mapsto\alpha\circ\beta$
defines a semigroup isomorphism
\begin{gather*}
\OIG|_{\xi(X)}\buildrel\cong\over\longrightarrow\OIG_{\geq\alpha} 
 =\{\beta\in\OIG:\beta\geq\alpha\},
\end{gather*}
where $\alpha$ is any covector with $\supp(\alpha) = [X]$.
\end{Proposition}
\begin{proof}
Define a map $g:\OIG_{\geq \alpha}\to\OIG|_{\xi(X)}$ by
$\beta\mapsto\res_{\xi(X)}(\beta)$. Then $g$ is a semigroup morphism
by \ref{p:RestrictionOfCovectors}. 
Define a map $f:\OIG|_{\xi(X)}\to\OIG_{\geq\alpha}$ by
$f(\res_{\xi(X)}(\beta)) = \alpha\circ\beta$. 

We argue that $f$ is well-defined. Suppose $\beta, \gamma
\in \OIG$ with $B = \supp(\beta)$ and $C = \supp(\gamma)$, and suppose
$\res_{\xi(X)}(\beta) = \res_{\xi(X)}(\gamma)$. 
Then the support of
$\res_{\xi(X)}(\beta) = \res_{\xi(X)}(\gamma)$ is 
$B\vee[X]=C\vee[X]$. 
This implies that $\supp(\alpha\circ\beta)=\supp(\alpha\circ\gamma)$
because $\supp(\alpha) = [X]$. Therefore, to show
$\alpha\circ\beta=\alpha\circ\gamma$ it suffices to show that they 
agree on $\Gamma([X]\vee B)$. 
Let
$x \in
\Gamma([X] \vee B)$. Then $x \in \Gamma(X) \cup \xi(X)$ by
\ref{p:PropertiesOfContinuations}. If $x \in \Gamma(X)$, then
$(\alpha \circ \beta)(x) = \alpha(x) = (\alpha \circ \gamma)(x)$. So
suppose $x \in \xi(X)$.  Then $(\alpha \circ \beta)(x) = \beta(x)$ and
$(\alpha \circ \gamma)(x) = \gamma(x)$. Moreover, $x \in
\Gamma|_{\xi(X)}(B|_{\xi(X)})$ and $x \in \Gamma|_{\xi(X)}(C|_{\xi(X)})$
by \ref{l:XiAndGammaInRestrictionToXi}. 
So $\res_{\xi(X)}(\beta)(x)=\beta(x)$ and 
$\res_{\xi(X)}(\gamma)(x)=\gamma(x)$.
Since $\res_{\xi(X)}(\beta)=\res_{\xi(X)}(\gamma)$, we have
$(\alpha\circ\beta)=(\alpha\circ\gamma)$.

Now $f$ is a semigroup morphism since 
$\alpha\circ\beta\circ\alpha=\alpha\circ\beta$ 
for all covectors $\alpha$ and $\beta$ 
(see \ref{p:ProductAndPartialOrder}). 
To complete the proof observe that $f\circ g$ and $g\circ f$ are the
identity morphisms of $\OIG|_{\xi(X)}$ and $\OIG_{\geq \alpha}$,
respectively.
\end{proof}

\section{Structure of oriented interval greedoids}

\subsection{$\OIG$ is a graded poset}

The next result generalizes \cite[Lemma~4.1.12]{OrientedMatroids1993} from
oriented matroids to oriented interval greedoids.

\begin{Lemma}
\label{l:Lemma4.1.12}
Let $(E,\F,\OIG)$ denote an oriented interval greedoid. Suppose
$\alpha, \beta \in \OIG$ with $\supp(\alpha) \leq \supp(\beta)$ and
$\alpha \not\leq \beta$. Then there exists $\delta \in \OIG$ such that
$\delta \lessdot \beta$ and 
for all $x \not\in \SeparationSet(\alpha,\beta)$, if $\beta(x)\neq1$, then $\delta(x)=\beta(x)$.  
\end{Lemma}

\begin{proof}
Let $A = \supp(\alpha)$ and $B = \supp(\beta)$. Suppose the result is
not true. Of all $\alpha,\beta\in\OIG$ that violate the result choose
a pair with $|\SeparationSet(\alpha,\beta)|$ minimal. If
$\SeparationSet(\alpha,\beta) = \O$, then $\alpha \leq \beta$ by
\ref{l:CovectorsAndSeparationSets}, contradicting the assumption that
$\alpha\not\leq\beta$. Therefore, $\SeparationSet(\alpha,\beta) \neq
\O$. Let $y \in \SeparationSet(\alpha,\beta)$. If $\acb(y)=1$, then
$\bca(y)=1$ and so $\beta(y)=1$ because $A\leq B$ implies $\bca =
\beta$. This contradicts the fact that
$y\in\SeparationSet(\alpha,\beta)$. Therefore, $\acb(y)\neq1$. (OG4)
implies there exists $\gamma \in \OIG$ with $\gamma(y)=0$ and
for all $x \not\in \SeparationSet(\alpha,\beta)$,
if $\beta(x)=\bca(x)\neq1$, then $\gamma(x)=\acb(x)=\bca(x)=\beta(x)$.

We argue that
$\SeparationSet(\gamma,\beta)\subsetneq\SeparationSet(\alpha,\beta)$.
Suppose $e\notin\SeparationSet(\alpha,\beta)$. 
Then either $\beta(e)=1$ or $\gamma(e)=\beta(e)$. In both cases
$e\notin\SeparationSet(\gamma,\beta)$.
Since $y\in\SeparationSet(\alpha,\beta)$ and
$y\notin\SeparationSet(\gamma,\beta)$ (because $\gamma(y)=0$), 
the inclusion is proper.

Let $C = \supp(\gamma)$. We argue that $C < B$ by showing
that $\beta(e)=0$ implies $\gamma(e)=0$ and that $B\neq C$. 
If $\beta(e)=0$, then $e\notin\SeparationSet(\alpha,\beta)$,
so $\beta(e)=1$ (not possible) or $\gamma(e)=\beta(e)=0$. 
Since $\gamma(y)=0$ and $\beta(y)\in\{+,-\}$, we have $B\neq C$.

We argue that $\SeparationSet(\gamma,\beta)\neq\O$. Suppose
$\SeparationSet(\gamma,\beta)=\O$. Then $\gamma < \beta$ by
\ref{l:CovectorsAndSeparationSets}. Let $\delta \in \OIG$ denote a
coatom in the interval $[\gamma, \beta]$ of the poset $\OIG$ and let
$D = \supp(\delta)$.  We will argue that $\delta$ satisfies the
result, contradicting our assumption that no such $\delta$ exists.
First note that $\delta \lessdot \beta$ by the choice of $\delta$. 
It remains to show that 
for all $x \not\in \SeparationSet(\alpha,\beta)$,
if $\beta(x)\neq1$, then $\delta(x)=\beta(x)$.
Let $x\not\in\SeparationSet(\alpha,\beta)$.
If $\beta(x)\neq1$, then $\gamma(x) = \beta(x)$. 
Since $\gamma(x)\leq\delta(x)\leq\beta(x)$ and 
$\gamma(x)=\beta(x)$, it follows that $\delta(x)=\beta(x)$.
And if $\beta(x)=1$, then $\gamma(x)\neq0$,
so $\delta(x)\geq\gamma(x)>0$.

We argued above that $C < B$ and $\SeparationSet(\gamma,\beta)\neq\O$.
Hence, $\gamma\not\leq\beta$ by \ref{l:CovectorsAndSeparationSets}.
Since $\O\neq \SeparationSet(\gamma,\beta) \subsetneq
\SeparationSet(\alpha,\beta)$, the minimality of
$|\SeparationSet(\alpha,\beta)|$ implies there exists 
$\delta\in\OIG$ such that $\delta \lessdot \beta$ and 
for all $x\not\in\SeparationSet(\gamma,\beta)$,
if $\beta(x)\neq1$, then $\delta(x)=\beta(x)$.
But since 
$\SeparationSet(\gamma,\beta)\subsetneq\SeparationSet(\alpha,\beta)$, 
if $x\not\in\SeparationSet(\alpha,\beta)$, 
then $x\not\in\SeparationSet(\gamma,\beta)$.
Thus, for all $x\not\in\SeparationSet(\alpha,\beta)$,
if $\beta(x)\neq1$, then $\delta(x)=\beta(x)$.
This contradicts the assumption that no
such $\delta$ exists for the pair $\alpha$ and $\beta$. 
We have arrived at a contradiction; so the result is true.
\end{proof}

\begin{Example}
\ref{f:Lemma4.1.12} illustrates \ref{l:Lemma4.1.12} for the
antimatroid corresponding to the convex geometry on three colinear
points (\ref{x:ThreeColinearPoints}).
\begin{figure}[!ht]
\centering
\def\xO{-0.5}\def\yO{-0.5}
\def\xl{3.5}\def\yl{1.25}
\begin{pspicture}(3,1)(0,-0.75)
\border(-0.5,-0.5)(3.5,1.25)
\psline[linecolor=blue,linewidth=0.02](0,0)(1,0)
\psdot(0,0)\psdot(1,0)\psdot(3,0)
\uput[90](0,0){$-$}
\uput[90](1,0){$+$}
\uput[90](3,0){$0$}
\uput[-90](1.5,-0.5){$\alpha$}
\end{pspicture}
\qquad
\begin{pspicture}(3,1)(0,-0.75)
\border(-0.5,-0.5)(3.5,1.25)
\psline[linecolor=red,linewidth=0.02](0,0)(3,0)
\psdot(0,0)\psdot(1,0)\psdot(3,0)
\uput[90](0,0){$+$}
\uput[90](1,0){$1$}
\uput[90](3,0){$+$}
\uput[-90](1.5,-0.5){$\beta$}
\end{pspicture}
\qquad
\begin{pspicture}(3,1)(0,-0.75)
\border(-0.5,-0.5)(3.5,1.25)
\psline[linecolor=green,linewidth=0.02](1,0)(3,0)
\psdot(0,0)\psdot(1,0)\psdot(3,0)
\uput[90](0,0){$0$}
\uput[90](1,0){$+$}
\uput[90](3,0){$+$}
\uput[-90](1.5,-0.5){$\delta$}
\end{pspicture}
\qquad
\begin{pspicture}(3,1)(0,-0.75)
\border(-0.5,-0.5)(3.5,1.25)
\psline[linecolor=green,linewidth=0.02](1,0)(3,0)
\psdot(0,0)\psdot(1,0)\psdot(3,0)
\uput[90](0,0){$0$}
\uput[90](1,0){$-$}
\uput[90](3,0){$+$}
\uput[-90](1.5,-0.5){$\delta'$}
\end{pspicture}
\caption{The covectors $\delta$ and $\delta'$ both satisfy the
statement of \ref{l:Lemma4.1.12} for the covectors $\alpha$ and
$\beta$.}
\label{f:Lemma4.1.12}
\end{figure}
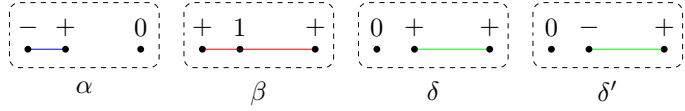
\end{Example}

The partial order on the set of all covectors restricts to a partial
order on $\OIG$. This next result shows that $\OIG$ is a graded poset
and describes the rank function of $\OIG$.

\begin{Proposition}
\label{p:SupportLatticeOIG}
Let $\OIG$ be an oriented interval greedoid over $(E,\F)$. Then
$\supp: \OIG \to \Phi$ is a cover-preserving poset surjection
of $\OIG$ onto $\Phi$ satisfying 
\begin{gather*}
\supp\big(\alpha \circ \beta\big) = 
\supp(\alpha) \vee \supp(\beta).
\end{gather*}
In particular, $\OIG$ is graded of rank equal to the rank of $\Phi$.
The rank of $\alpha \in \OIG$ is the rank of $\supp(\alpha) \in \Phi$.
\end{Proposition}
\begin{proof}
The identity follows immediately because if $A = \supp(\alpha)$ and $B
= \supp(\beta)$, then $\supp(\alpha \circ \beta) = A \vee B$, by
definition of the product. The fact that $\supp$ is a surjection of
posets follows from its definition and axiom (OG1). It remains to show
that $\supp$ is cover-preserving. 

Suppose $\alpha\lessdot\beta$. Let $A=\supp(\alpha)$ and
$B=\supp(\beta)$. 
Suppose there exists $C\in\Phi$ such that $A<C<B$. 
Let $\OIG'=\OIG|_{\xi(A)}$.  Let $\alpha'$ and $\beta'$ be the
elements of $\OIG'$ corresponding to $\alpha,\beta$.  

Since
$\supp:\OIG'\to[A,\hat 1]$ is surjective, there exists $\epsilon'\in\OIG'$ with
$\supp(\epsilon')=C$. Let $\gamma'=\alpha'\circ\epsilon'$. Then
$\alpha'<\gamma'$ and $\supp(\gamma')=A\vee C=C$. If $\gamma'\leq\beta'$,
then $\gamma'=\beta'$, contradicting that $C<B$. Hence,
$\gamma'\not\leq\beta'$. By \ref{l:Lemma4.1.12} there exists
$\delta'\in\OIG'$ such that $\delta'\lessdot\beta'$ and for all
$x\notin\SeparationSet(\gamma',\beta')$, if $\beta'(x)\neq1$, then
$\delta'(x)=\beta'(x)$.  Let $D=\supp(\delta')$.  Since $\delta'\in\OIG'$,
$A<D$.  Thus 
$\alpha<\alpha\circ\delta<\alpha\circ\beta=\beta$
(see \ref{p:ProductAndPartialOrder}), contradicting
that $\alpha\lessdot\beta$.
%
%
\end{proof}

\subsection{Oriented interval greedoids of rank $1$}

Let $(E,\F,\OIG)$ be an oriented interval greedoid and let $\Phi$ be
its lattice of flats. The previous result shows that $\OIG$
is a graded lattice and that its rank is equal to that of $\Phi$.
We define the \defn{rank} of $(E,\F,\OIG)$ to be the rank
of $\OIG$ (equivalently, the rank of $\Phi$). 

We first make a useful observation about arbitrary oriented 
interval greedoids.

\begin{Lemma}
\label{l:bottom}
Suppose $(E,\F,\OIG)$ is an oriented interval greedoid.  Let $\hat 0$
be the minimal element of $\Phi$.  Then there is a unique element of 
$\OIG$ with support $\hat 0$.  
\end{Lemma}

\begin{proof}
By (OG1), there exists $\bottom\in\OIG$ with
$\supp(\bottom)=\hat0$. Then $\Gamma(\supp(\bottom))=\O$, so
$\bottom(e)\in\{0,1\}$ for all $e\in E$.  Thus, $\bottom$
is determined by $\F$, and consequently is the unique element of
$\OIG$ with support $\hat0$.  
\end{proof}

We will consistently denote the unique element of $\OIG$ 
with support $\hat 0$ by $\bottom$.  

The next result describes the oriented interval greedoids of rank
$1$.

\begin{Proposition}
\label{p:Rank1OIGs}
Suppose $(E,\F,\OIG)$ is an oriented interval greedoid of rank $1$. 
Then $\OIG$ contains exactly three elements, and its Hasse diagram 
is
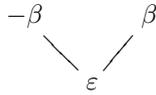
\begin{gather*}
\xymatrix@C=1em@R=3ex{
-\beta\ar@{-}[dr] & & \beta\ar@{-}[dl] 
\\
&\bottom}
\end{gather*}
\end{Proposition}
\begin{proof}
Since $\OIG$ has rank $1$, $\Phi$ contains exactly two elements, its
minimal and maximal elements $\hat0$ and $\hat1=[\O]$, respectively. 

By the previous lemma, $\bottom$ is the unique element element of $\OIG$
with support $\hat 0$. 
We now show that there exist exactly two elements in $\OIG$ of
support $\hat1$. 
By (OG1), there exists $\beta\in\OIG$ such that $\supp(\beta)=\hat1$.
Then $-\beta\in\OIG$ by (OG2). Since $\Gamma(\O)\ne \O$, 
$\beta\neq-\beta$.  
So $\OIG$ contains at least two elements of support $\hat1$.

Let $\alpha\in\OIG$, $\alpha\neq\beta$ and $\supp(\alpha)=\hat1$. 
Let $y\in \Gamma(\O)$, $y\not\in \SeparationSet(\alpha,\beta)$.  
Then let $\delta$ be the vector guaranteed by \ref{l:Lemma4.1.12}.
Then $\delta(y)=\beta(y)$.  But $\delta=\bottom$, so this is
impossible.  It follows that $\SeparationSet(\alpha,\beta)=
\Gamma(\O)$; in other words, $\alpha=-\beta$.  
\end{proof}

\subsection{Oriented interval greedoids of rank $2$}

\begin{Proposition}
\label{l:Rank2OIGs}
Suppose $(E,\F,\OIG)$ is an oriented interval greedoid of rank $2$. 
Then $\OIG$ is isomorphic to the semigroup of covectors of 
an oriented matroid of rank $2$, 
or the Hasse diagrams of $\OIG$ and $\Phi$ are, respectively,
the following two posets.
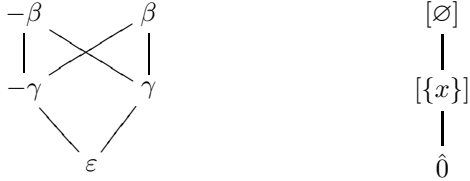
\begin{gather*}
\xymatrix@C=1em@R=3ex{
-\beta\ar@{-}[d]\ar@{-}[drr] 
& & 
\beta\ar@{-}[dll]\ar@{-}[d] 
&&& &&& [\O]\ar@{-}[d] 
\\
-\gamma\ar@{-}[dr] & & \gamma\ar@{-}[dl] 
&&& &&& [\{x\}]\ar@{-}[d]
\\
&\bottom& 
&&& &&& \hat0}
\end{gather*}
\end{Proposition}
\begin{proof}
There are two cases to consider.

\begin{Case}
Suppose $\Phi$ contains at least two coatoms.
By \ref{p:RestrictionToUnderlyingOrientedMatroid},
the restriction $\OIG\restd$ is an oriented matroid.
The map $C\mapsto [Y]$ for any $Y\in C$
embeds $\Phi\restd$ into the interval $[ [X],\hat1 ]$ of $\Phi$, 
where $X$ is the maximal among the feasible sets contained in
$\Gamma(\O)$ (see \ref{sss:RestrictionofIGs}).
Since every coatom of $\Phi$ is of the
form $[\{x\}]$ for some $x\in\Gamma(\O)$, there is a bijection between
the coatoms of $\Phi$ and those of $\Phi\restd$. 
Therefore, $[X]=\hat0$, so $\Phi\restd\cong\Phi$.
So $\OIG\restd$ is a rank 2 oriented matroid.
We argue that the map $\gamma\mapsto\gamma\restd$ is an
isomorphism $\OIG\cong\OIG\restd$.
By \ref{p:RestrictionToUnderlyingOrientedMatroid}
we need only show that this is an injection.

Let $\gamma,\gamma'\in\OIG$ and suppose $\gamma\restd=\gamma'\restd$. 
Then $\supp(\gamma)\restd=\supp(\gamma')\restd$.
Since $\Phi\restd\cong\Phi$, it follows that
$\supp(\gamma)=\supp(\gamma')$.
This implies that $\gamma(x)$ is $0$ or $1$ 
if and only if $\gamma'(x)$ is $0$ or $1$, respectively. 
Let $C=\supp(\gamma)=\supp(\gamma')$.

If $C=\hat0$, then $\gamma=\gamma'$ since
there is a unique element of $\OIG$ with support $\hat0$.
If $C=\hat1=[\O]$, then $\gamma\restd=\gamma'\restd$ implies
that $\gamma=\gamma'$ since they agree on $\Gamma(\O)$. 

Let $C\gtrdot\hat0$. Suppose $\gamma\neq\gamma'$. 
Arguing as in the end of \ref{p:Rank1OIGs}, we
conclude $\Gamma(C)=\SeparationSet(\gamma,\gamma')$.
Since $\gamma$ and $\gamma'$ agree on $\Gamma(\O)$ and disagree on
$\Gamma(C)$, it follows $\gamma\restd=\gamma'\restd$ takes values in
$\{0,1\}$. Thus, $C\restd=\hat0$, which implies $C=\hat0$,
contradicting that $C\gtrdot\hat0$. Thus, $\gamma=\gamma'$.
\end{Case}

\begin{Case}
Suppose $\Phi$ contains exactly one coatom.
Then $[X]$ is this coatom, so $X=\{x\}$ for some $x\in E$.
By \ref{l:bottom}, $\bottom$ is the unique element of $\OIG$ with
support $\hat0$.  

By (OG1), there exists $\gamma\in\OIG$ such that $\supp(\gamma)=[X]$.
By arguing as in \ref{p:Rank1OIGs}, we conclude that 
$\gamma\neq-\gamma$ and that if $\nu\in\OIG$ with $\supp(\nu)=[X]$,
then $\nu=\gamma$ or $\nu=-\gamma$.
Hence, there are exactly two elements in $\OIG$ of support $[X]$.

By (OG1), there exists $\beta\in\OIG$ such that $\supp(\beta)=[\O]$.
By (OG2) $-\beta\in\OIG$. As above, we have $\beta\neq-\beta$. 
$\Gamma(\O)\cap\Gamma(X)=\O$, since if $y\in \Gamma(\O)\cap\Gamma(X)$,
then $y\not\in \xi(X)$, so $[y]\ne[x]$, contradicting our assumption
that $\Phi$ has only one coatom.  Thus $\gamma,-\gamma<\beta,-\beta$.

Let $\nu\in\OIG$ such that $\supp(\nu)=\hat1$.
By arguing as before (using \ref{l:Lemma4.1.12}), 
it follows that $\nu=\beta$ or $\nu=-\beta$.\qedhere
\end{Case}
\end{proof}

\subsection{Intervals of length two}

Let $\hat\OIG$ denote the poset obtained from $\OIG$ by adjoining a
maximal element $\hat1$. We prove that all intervals of length two in
$\hat\OIG$ contain exactly four elements. 

\begin{Proposition}
\label{p:TwoIntervals}
Suppose $\OIG$ is an oriented interval greedoid. 
Then all intervals in $\hat\OIG$ of length two contain exactly four
elements.
\end{Proposition}

\begin{proof}
Let $\alpha,\beta,\gamma\in\hat\OIG$ such that
$\alpha\lessdot\gamma\lessdot\beta$. The case where $\beta=\hat1$ was
proved in \ref{p:Rank1OIGs}, so suppose $\beta\in\OIG$. 
Let $A=\supp(\alpha)$, $B=\supp(\beta)$ and $C=\supp(\gamma)$. By
\ref{p:SupportLatticeOIG}, $A\lessdot C\lessdot B$ in $\Phi$.

Let $\supp(\alpha)=[X]$ for some $X\in\F$.
By \ref{p:RestrictionSemigroup} and \ref{p:ProductAndPartialOrder},
$\OIG|_{\xi(X)}\cong\OIG_{\geq\alpha}$
(as posets), so 
$\{\delta\in\OIG:\alpha\lessdot\delta\lessdot\beta\}
\cong
\{\delta\in\OIG|_{\xi(X)}:
  \alpha|_{\xi(X)}\lessdot\delta\lessdot\beta|_{\xi(X)}\}.$
Thus, by passing to $\OIG|_{\xi(X)}$ we can suppose
that $A=\hat0$.

Let $\supp(\beta)=[Y]$ for some $Y\in\F$.
By \ref{p:ContractionSemigroup},
$\OIG/Y\cong\OIG_{\leq[Y]}=\{\nu\in\OIG:\supp(\nu)\leq[Y]\}$.
Since $\Phi/Y\cong[\hat0,[Y]]\sseq\Phi$ 
(\ref{p:PropertiesOfContraction}),
by passing to $\OIG/Y$, we can suppose that $B=\hat1$, 
and therefore that $\Phi$ is a lattice of rank 2.

\ref{l:Rank2OIGs} classified the oriented interval greedoids of rank
$2$ as being either an oriented matroid of rank $2$ or having
the Hasse diagram shown in the statement of \ref{l:Rank2OIGs}. 
For the latter situation a quick inspection of the given poset
establishes the result. And for the former situation, it is well-known
that this result holds for oriented matroids
(\cite[Theorem~4.1.14]{OrientedMatroids1993}). 
\end{proof}

\subsection{The Underlying Oriented Matroid}

Let $(E,\F,\OIG)$ be an oriented interval greedoid. The top element in
the poset of flats $\Phi$ is $[\O]$, and by
\ref{p:PropertiesOfContinuations} it follows that $\Gamma(\O)
\subseteq \Gamma(A) \cup \xi(A)$ for any flat $A \in \Phi$. This
implies that $\alpha(x)\in\{0,+,-\}$ for any $\alpha\in\OIG$ and any
$x\in\Gamma(\O)$. Moreover, $\Gamma(\O)$ is the largest subset of $E$
with this property: if $\alpha$ is maximal in $\OIG$, then
$\supp(\alpha) = [\O]$ and $\alpha(x) = 1$ if and only if $x \notin
\Gamma(\O)$. This observation implies that the restriction
to $\Gamma(\O)$ produces an oriented interval greedoid whose covectors
take values in $\{0,+,-\}$. Thus, $\OIG|_{\Gamma(\O)}$ is an oriented
matroid. Alternatively, one can note that the restriction 
$(\Gamma(\O),\F|_{\Gamma(\O)})$ is a matroid and appeal to
\ref{p:OMAndOIGs}.

\begin{Definition}
Let $\OIG$ denote an oriented interval greedoid over $(E,\F)$. The
\defn{underlying oriented matroid} of $\OIG$ is $\UOM =
\OIG|_{\Gamma(\O)}$.
\end{Definition}
The lattice of flats $\overline \Phi$ of $\UOM$ is a geometric lattice
because $\UOM$ is an oriented matroid. Moreover, it is isomorphic to
the sublattice of $\Phi$ generated by all the coatoms. 

\subsection{The Tope Graph}

A \defn{tope} of an oriented interval greedoid $(E,\F,\OIG)$ is a
covector that is maximal in $\OIG$ with respect to the partial order
on covectors. Alternatively, topes are covectors whose support is
$\hat 1=[\O]$. A \defn{subtope} of $\OIG$ is a covector in $\OIG$ that
is covered by some tope. From \ref{p:TwoIntervals} 
it follows that every subtope is
covered by exactly two different topes. Two topes are said to be
\defn{adjacent} if there exists a subtope that is covered by both
topes.

The \defn{tope graph} $\TG(\OIG)$, or just $\TG$, of $\OIG$ is the graph
with one vertex for each tope of $\OIG$ and an edge between two
vertices if the corresponding topes are adjacent.


\begin{Lemma}
 Suppose $\OIG$ is an oriented interval greedoid. Then the tope graph
 of $\OIG$ is isomorphic to the tope graph of the underlying oriented matroid 
 $\UOM$ of $\OIG$.
\end{Lemma}
\begin{proof}
First we will show that topes of $\OIG$ are in one-to-one
correspondence with the topes of $\UOM$. 
Suppose $\alpha$ is a tope in $\OIG$. Then
$\supp(\alpha)=[\O]=\{\O\}$, and so
$\supp|_{\Gamma(\O)}(\res_{\Gamma(\O)}(\alpha))=\{\O\}=\hat1\in\Phi\restd$.
Thus, $\res_{\Gamma(\O)}(\alpha)$ is a tope of $\UOM$.

Conversely, suppose $\res_{\Gamma(\O)}(\alpha)$ is a tope of $\UOM$. 
Then $\supp|_{\Gamma(\O)}(\res_{\Gamma(\O)}(\alpha))
 =[\O]|_{\Gamma(\O)} =\{\O\}$.
Let $A=\supp(\alpha)$. Then $A|_{\Gamma(\O)}=\{\O\}$, so $\O$ is
maximal among the feasible sets contained in $\Gamma(\O)\cap\xi(A)$.
This implies that $\Gamma(\O)\cap\xi(A)=\O$ (if
$x\in\Gamma(\O)\cap\xi(A)$, then $\{x\}\in\F$ because
$x\in\Gamma(\O)$, contradicting that maximality of $\O$). 
If $A\neq[\O]$, then $A\leq[\{y\}]$ for some $y\in\Gamma(\O)$. 
Hence, $y\in\Gamma(\O)\cap\xi(A)$, contradicting that 
$\Gamma(\O)\cap\xi(A)\neq\O$. Thus,
$A=[\O]$. So $\alpha$ is a tope of $\OIG$. 

Let $\alpha$ and $\beta$ be topes in $\OIG$ and suppose
$\res_{\Gamma(\O)}(\alpha)=\res_{\Gamma(\O)}(\beta)$.
We show that $\alpha=\beta$ by showing that they agree on
$\Gamma(\supp(\alpha))=\Gamma(\supp(\beta))=\Gamma(\O)$.
Since $\supp(\alpha)=\supp(\beta)=[\O]$, we have
$\Gamma(\O)=\Gamma|_{\Gamma(\O)}(\O)
=\Gamma|_{\Gamma(\O)}([\O]|_{\Gamma(\O)})$. Hence,
$\res_{\Gamma(\O)}(\alpha)(w)=\alpha(w)$ 
and 
$\res_{\Gamma(\O)}(\beta)(w)=\beta(w)$ 
for all $w\in\Gamma(\O)$. It follows that $\alpha(w)=\beta(w)$
for all $w\in\Gamma(\O)$.
This establishes the one-to-one correspondence.

Suppose $\alpha,\beta\in\OIG$ are two adjacent topes and let
$\gamma\in\OIG$ with $\gamma\lessdot\alpha,\beta$. Then
$\supp(\gamma)\lessdot\supp(\alpha)=\supp(\beta)=[\O]$. Since
$\res_{\Gamma(\O)}$ is a semigroup morphism, it follows that
$\res_{\Gamma(\O)}(\gamma)\leq\res_{\Gamma(\O)}(\alpha)$. We cannot
have equality since this would imply that both are topes of
$\OIG\restd$, contradicting that $\gamma$ is not a tope.
We have $\supp(\gamma)=[\{y\}]$ for some $y\in\Gamma(\O)$ since
all coatoms of $\Phi$ are of this form. Hence,
$\res_{\Gamma(\O)}(\gamma)=[\{y\}]\restd\lessdot[\O]\restd$.
Since $\supp\restd$ is cover-preserving, it follows that
$\res_{\Gamma(\O)}(\gamma)\lessdot\res_{\Gamma(\O)}(\alpha)$.
Similarly, 
$\res_{\Gamma(\O)}(\gamma)\lessdot\res_{\Gamma(\O)}(\beta)$.
So 
$\res_{\Gamma(\O)}(\alpha)$ and $\res_{\Gamma(\O)}(\beta)$
are adjacent topes.

Let $\res_{\Gamma(\O)}(\alpha),\res_{\Gamma(\O)}(\beta)\in\OIG\restd$
be adjacent topes and let $\res_{\Gamma(\O)}(\gamma)\in\OIG\restd$
with $\res_{\Gamma(\O)}(\gamma)\in\OIG$ with
$\res_{\Gamma(\O)}(\gamma)
\lessdot\res_{\Gamma(\O)}(\alpha),\res_{\Gamma(\O)}(\beta)$. Since
$\res_{\Gamma(\O)}(\gamma\circ\alpha)
=\res_{\Gamma(\O)}(\gamma)\circ\res_{\Gamma(\O)}(\alpha)
=\res_{\Gamma(\O)}(\alpha)$ and since $\gamma\circ\alpha$ and $\alpha$
are both topes, we have $\alpha=\gamma\circ\alpha$. So
$\gamma<\alpha$. To show that $\gamma\lessdot\alpha$, it suffices to
show that $\supp(\gamma)\lessdot[\O]$. Let $C=\supp(\gamma)$. If $C$
is not covered by $[\O]$, then $C\leq[\{x,y\}]$ for some
$x,y\in\Gamma(\O)$, $x\neq y$. Thus,
$\{x,y\}\sseq\Gamma(\O)\cap\xi(C)$. Let $Y\supseteq\{x,y\}$ be maximal
among the feasible sets contained in $\xi(C)\cap\Gamma(\O)$. By
definition, $\supp\restd(\res_{\Gamma(\O)}(\gamma))=C\restd$ is the
flat containing $Y$. Since $|Y|>2$, it follows that
$\supp\restd(\res_{\Gamma(\O)}(\gamma))$ is not a coatom of
$\Phi\restd$, contradicting that it is. Hence, $\gamma\lessdot\alpha$.
Similarly, $\gamma\lessdot\beta$. Therefore, $\alpha$ and $\beta$ are
adjacent topes.
\end{proof}

\section{CW-spheres from oriented interval greedoids}
\subsection{CW-spheres}
The Sphericity Theorem is an important result for oriented matroids 
which asserts that there is a certain
regular CW-sphere associated to any oriented matroid, whose cells correspond
to the covectors of the oriented matroid.  It is originally due to
Folkman and Lawrence \cite{FolkmanLawrence1978}; see also \cite[Theorem 4.3.3]{OrientedMatroids1993}.  
In this section and the next, we will prove
the corresponding result for oriented interval greedoids.  

We recall some topological definitions, following \cite[Section~4.7]{OrientedMatroids1993}.

A \defn{ball} in a topological space homeomorphic to the usual 
$d$-dimensional ball, for some nonnegative integer $d$.  

A \defn{regular cell complex} $\Delta$ is a finite set of balls in a
Hausdorff topological space $|\Delta|=\bigcup_{\sigma\in\Delta} \sigma$
with the properties that:
\begin{itemize}
\item The interiors of the balls $\sigma\in\Delta$ partition 
$|\Delta|$.  
\item For each $\sigma\in\Delta$, the boundary of $\sigma$ is the union
of some elements $\tau\in\Delta$.
\end{itemize}
This definition of a regular cell complex is (non-trivially) equivalent
to the usual definition of a regular CW-complex.  (See \cite
[Section 4.7]{OrientedMatroids1993}.)

A cell complex $\Delta$ is called a \defn{regular CW-sphere}
if its geometric realization
$|\Delta|$ is homeomorphic to a sphere.  

The \defn{face poset} of a cell complex is the poset structure on the 
cells of $\Delta$, ordered by containment.  The \defn{augmented face poset}
of a cell complex is the face poset with a maximal element $\hone$ adjoined. 

We can now state our main theorem for this section more precisely.

\begin{Theorem} \label{sphere}
For $(E,\F,\OIG)$ an oriented interval greedoid, 
$\hOIG$ is isomorphic to the augmented face poset of a regular CW-sphere.
\end{Theorem}

The \defn{order complex} of a bounded poset $P$ 
is the simplicial complex consisting
of chains in $P\setminus \{\hzero,\hone\}$.
Taking a barycentric subdivision of the CW-sphere in the previous 
theorem, we obtain the following.

\begin{Corollary}\label{spherecor}
The order complex of $\hOIG$ is a simplicial sphere.  
\end{Corollary}

\begin{proof}[Proof of \ref{sphere}]
As in the proof of the Sphericity Theorem in \cite{OrientedMatroids1993}, 
the main technical tool required in the proof is the notion of 
\emph{recursive coatom ordering}.  

A graded, bounded poset $P$ is said to have a \defn{recursive coatom
ordering} if it is either of rank 1, or if there is a linear ordering on
its coatoms, $q_1,\dots,q_r$ which satisfies:

\begin{enumerate}
\item[(i)]  $[\hzero,q_i]$ admits a recursive coatom ordering in which the coatoms
of $[\hzero,q_i]$ which lie below some $q_j$ with $j<i$, come first; 
\item[(ii)]  any element lying below $q_i$ and also below some $q_j$ for $j<i$, 
necessarily lies below a coatom of $[\hzero,q_i]$ which lies below some 
$q_k$ with $k<i$.  
\end{enumerate}

This concept is dual to the condition of having a recursive atom ordering,
which goes back to \cite{BjornerWachs1983}. The concept has been extended to 
non-graded posets \cite{BjornerWachs1996}, but we shall not need that here.  

The fundamental technical result is
the following lemma, whose proof we defer to the next
section.

\begin{Lemma}\label{rcord} $\hOIG$ admits a recursive coatom ordering. 
\end{Lemma}

A poset is called \defn{thin} if all intervals of length 2 have
cardinality four.  By \ref{p:TwoIntervals}, we know that $\hOIG$ is thin.  
The following theorem completes our proof.

\begin{Theorem}[{\cite{Bjorner1984},\cite[Theorem 4.7.24]{OrientedMatroids1993}}]
$P$ is isomorphic to the face poset of a shellable regular cell
decomposition of the sphere iff $P$ is thin and admits a recursive
coatom ordering. \end{Theorem}

(We shall not discuss the significance of the ``shellable'' in the
above theorem; the interested reader is directed to \cite{OrientedMatroids1993}.)
\end{proof}

We now turn to the proof of the corollary.  

\begin{proof}[Proof of \ref{spherecor}]
The order complex of the augmented face poset of a regular cell complex
$\Delta$ is homeomorphic
to $|\Delta|$ \cite[Proposition 4.7.8]{OrientedMatroids1993}.  (In fact, 
the order complex should be thought of as the barycentric subdivision of
the regular cell complex.)  The corollary follows.  
\end{proof}

\subsection{A recursive coatom ordering for $\hat{{\OIG}}$}

This section is devoted to the proof of \ref{rcord}, which asserts
that $\hOIG$ has a recursive coatom ordering.  


If one chooses a particular tope 
$\alpha$ of $\OIG$ then there is a natural poset structure on the topes with
respect to which $\alpha$ is the minimum element and $-\alpha$ is 
the maximum element, and the Hasse diagram is (a suitable orientation of)
the tope graph.  This poset is called $\pt(\OIG,\alpha)$. (Since the topes 
of $\OIG$ are identified with the topes of $\UOM$, this follows from the
analogous statements for oriented matroids; see \cite[Section 4.2]{OrientedMatroids1993}.)

Let $\alpha$ be a tope of $\OIG$. Consider a maximal chain $\overline\beta$ 
in $\pt(\OIG,\alpha)$, say $\alpha=\beta_0<\dots<\beta_r=-\alpha$.
Choose $\gamma_i$ to be a common facet of $\beta_{i-1}$ and $\beta_i$.  
Let $G_i=\supp(\gamma_i)$.  The $G_i$ are distinct and include all the
coatoms of $\Phi$.
Thus, $\overline\beta$   induces a linear order on the 
coatoms of $\Phi$.  However (unlike the situation for oriented
matroids) this does not immediately yield a linear order on the coatoms of 
$[\bottom,\alpha]$, because there may be more than one coatom with the same
support.  

For $1\leq i \leq r$, let $\OIG_i$ be the oriented matroid obtained by 
contracting $\OIG$ to $G_i$.  Consider the tope poset $\pt(\OIG_i,\gamma_i)$.

Let $\Delta$ be the set of facets of $\alpha$.  
Let $\Delta_i$ be the set of facets of $\alpha$ whose support is 
$G_i$.  (This set could be empty.)

A linear extension of $\pt(\OIG_i,\gamma_i)$ will be called 
\defn{adapted to} $\alpha$ if it contains in order:
\begin{enumerate}
\item first, the topes of $\OIG_i$ that lie on the same side as $\gamma_i$
of some $G_j$ for $j<i$,
\item then, the topes that are facets of $\alpha$,
\item finally, the remaining topes of $\OIG_i$. 
\end{enumerate}

We will need the following lemma:

\begin{Lemma} $\pt(\OIG_i,\gamma_i)$ admits a linear extension 
adapted to $\alpha$.  
\end{Lemma}

\begin{proof} It is certainly possible to define a linear extension
of $\pt(\OIG_i,\gamma_i)$ 
which begins with the elements (1) above, since they
form a lower order ideal in $\pt(\OIG_i,\gamma_i)$.  In order to be able
to construct a linear extension such that the next elements are those
from (2) above, we need to show that any tope below a tope from (2) not 
in (2), is contained in (1).  If $\delta$ is a tope of $\Delta_i$ which
is a facet of $\alpha$, and $\epsilon$ is a tope lying below
$\delta$ which is not a facet of $\alpha$, it must be separated from
$\alpha$ by some $G_j$ with $j<i$, which shows that $\epsilon$ is in (1).  
%
%
%
Thus the linear extension, whose beginning was
already described, can be continued with the set of facets of $\alpha$,
followed by the remaining topes of $\OIG_i$.  
\end{proof}

A linear order on $\Delta$ will be said to be \defn{compatible} with 
$\overline\beta$ if 
\begin{enumerate}
\item the elements of $\Delta$ are arranged first of all
in increasing order by support (so $\Delta_1$ comes first, then
$\Delta_2$, etc.), 
\item the elements of $\Delta_i$ are arranged according to a linear
order on $\pt(\OIG_i,\gamma_i)$ which is adapted to $\alpha$.  
\end{enumerate}

Now we will prove the following:

\begin{Proposition} \begin{enumerate}

\item For a tope $\alpha$ in $\OIG$, and a maximal chain
$\overline\beta$ in $\pt(\OIG,\alpha)$,  
any order on the coatoms of $[\bottom,\alpha]$ 
compatible with $\overline\beta$ is a recursive coatom order.

\item For a tope $\alpha$ in $\OIG$, any linear extension of 
$\pt(\OIG,\alpha)$ is a recursive coatom ordering for $\hOIG$. 
\end{enumerate}
\end{Proposition}

\begin{proof}
The proof will be by induction on the rank of $\OIG$.  
The base case, when the rank of $\OIG$ is 1, is trivial.  
We will assume
that (1) and (2) hold for oriented interval greedoids of rank less than $n$;
we will 
prove (1) for oriented interval greedoids of rank $n$, and then make use of 
(1) to prove (2) for 
oriented interval greedoids of rank $n$.  

\emph{Proof of (1).}  Pick a coatom order for $[\bottom,\alpha]$ which is 
compatible with $\overline\beta$.  As part of this, we are given
$\gamma_i$ a common facet of $\beta_{i-1}$ and $\beta_i$.  Let $G_i$ be
the support of $\gamma_i$.  
Let $\Delta_i$ be the coatoms of $\alpha$ with support $G_i$. 
As part of our coatom order for $[\bottom,\alpha]$, we are given a linear order
on $\Delta_i$ which is the restriction of a linear extension of  
$\pt(\OIG_i,\gamma_i)$ adapted to $\alpha$. Fix such a linear extension.

Let $\delta\in\Delta_i$ be a coatom of $[\bottom,\alpha]$.  
We must define a coatom order for $[\bottom,\delta]$.  
Using our chosen linear extension of $\pt(\OIG_i,\gamma_i)$, we can apply
(2) to $\hOIG_i$, obtaining a recursive coatom order for $[\bottom,\delta]$.  
We must show that this order satisfies the necessary conditions.  

Now, $\delta$ is a coatom of two different posets,
$[\bottom,\alpha]$ and $\hOIG_i$.  Let $X$ be the set of coatoms of 
$[\bottom,\alpha]$
which precede $\delta$ with respect to the coatom order on $[\bottom,\alpha]$,
and let $Y$ be the set of coatoms of $\hOIG_i$ which precede $\delta$
with respect to the fixed linear extension of $\pt(\OIG_i,\gamma_i)$.  
Let $\check X$ be the coatoms of $[\bottom,\delta]$ lying below an element 
of $X$, and let $\check Y$ be the coatoms of $[\bottom,\delta]$ lying
below an element of $Y$.  We will now show that $\check X$ and $\check Y$
coincide.

Let $\epsilon$ be a coatom of $[\bottom,\delta]$.  
By \ref{p:RestrictionSemigroup}, $\OIG_{\geq \epsilon}$ is itself an 
oriented greedoid, so we may assume that $\epsilon = \bottom$,
or, in other words, that $\OIG$ is rank 2.    
By \ref{l:Rank2OIGs}, we know that $\OIG$ is either isomorphic to a 
rank 2 oriented matroid, or else it is of the special form described in
that Proposition.  In either case, it is straightforward to check that
$\epsilon \in \check X$ iff $\epsilon \in \check Y$.  

Since we know property (i) of recursive coatom orders holds for our
fixed linear extension of $\pt(\OIG_i,\gamma_i)$, property
(i) also follows for our coatom ordering on $[\bottom,\alpha]$.  

Next, we check property (ii).  Let $\epsilon \in [\bottom,\delta]$, which
 lies under some element $\zeta \in X$.  
We must show that it also lies below
some element of $\check X$.  

Again, by restricting, we may assume that $\epsilon=\bottom$.  The fact 
that that $\epsilon$ lies under an element of $X$ implies, in particular,
 that $X$ is non-empty, and thus that $\delta$
is not the first coatom in our coatom order on $[\bottom,\alpha]$.  
We will now show that $Y$ is non-empty.  If $\delta$ is not the first coatom
with support $G_i$ in our recursive coatom order on $[\bottom,\alpha]$ then
this is clear.  So suppose that $\delta$ is the first coatom with support
$G_i$ in our recursive coatom order.  Since $\delta$ is not the first
coatom overall, it must be that $i>1$.  Therefore $\gamma_i$ is not
a facet of $\alpha$, so $\gamma_i\in Y$.  

Now, since we have assumed that $\epsilon=\bottom$, the fact that $Y$
is non-empty means that there are elements of $Y$ lying over $\epsilon$.  
Therefore, by property (ii) for the fixed linear extension of 
$\pt(\OIG_i,\gamma_i)$, we know that there are elements of $\check Y$ 
lying over $\epsilon$.  Since $\check Y=\check X$, we are done.  

\emph{Proof of (2), assuming (1).}  
Pick a linear extension of $\pt(\OIG,\alpha)$.  
For each coatom $\delta$ of $\hOIG$, pick a maximal chain $\overline\beta$ in 
$\pt(\OIG,\delta)$ which includes $\alpha$.  Then we claim that any 
linear order on the coatoms of $[\bottom,\delta]$,
compatible with $\overline\beta$, satisfies the necessary conditions.   
First of all, it is a recursive coatom order by (1).  

Second, define $Q_\delta$ to be the set of coatoms of $[\bottom,\delta]$  
 which also lie under some $\xi$ preceding $\delta$ in the linear extension
of $\pt(\OIG,\alpha)$.  
The coatoms of $Q_\delta$ precede the other
coatoms of $\delta$ in any order compatible with $\overline\beta$.  
(In fact, for this, it suffices to know that an order compatible with
$\overline\beta$ agrees with the order induced by $\overline\beta$ on 
the coatoms of $[\hat 0,\supp(\delta)]$.)  This proves (i).

Thirdly, we check that 
$$\bigcup_{\zeta\in Q_\delta}[\bottom,\zeta]= [\bottom,\delta]\cap \bigcup_{\xi \textrm{ preceding } \delta} [\bottom,\xi].$$
The containment of the lefthand side in the righthandside is obvious.  
For the opposite inclusion, let $\epsilon\in [\bottom,\delta]\cap 
[\bottom,\xi]$ for some $\xi$ preceding $\delta$.  
The topes of $\OIG$ that contain $\epsilon$ are exactly the topes
of $\UOM$ that contain $\epsilon|_{\Gamma(\emptyset)}$.  By \cite[Lemma 4.2.12]{OrientedMatroids1993},
this is an interval $I$ in $\pt(\UOM,\alpha)$.  Since $\epsilon$ is contained
in some $\xi$ preceding $\delta$, we know that $\delta$ is not the minimum
element of the interval.  Let $\rho$ be covered by $\delta$ in $I$.  
Since
$\rho$ lies below $\delta$ in $\pt(\OIG,\alpha)$, it precedes $\delta$
in the linear extension of $\pt(\OIG,\alpha)$.  Since $\rho$ is in $I$, 
$\epsilon \in [\hzero,\rho]$.  Finally, since $\rho$ and $\delta$ are
adjacent topes, they have a common subtope  $\sigma$ in $\UOM$.  Since, 
in $\hOIG$,  
$\rho$ and $\delta$ lie over $\epsilon|_{\Gamma(\emptyset)}$, $\sigma$ lies over $\epsilon|_{\Gamma(\emptyset)}$.  Thus $\supp(\sigma)$ lies over $\supp(\epsilon)$.

Let $\phi$ be covector of $\OIG$, such that $\phi|_{\Gamma(\emptyset)}=\sigma$.  Now consider $\epsilon\circ\phi$.  This lies over $\epsilon$, and its support
is $\supp(\epsilon)\vee\supp(\phi)=\supp(\phi)$.  Since, in
$\hOIG$,  $\phi$ and $\epsilon$ 
lie below both $\delta$ and $\rho$, the same is true of $\epsilon\circ\phi$, 
and
we are done: we can take $\epsilon\circ\phi$ as the common coatom of 
$[\bottom,\delta]$ and $[\bottom,\rho]$ lying over $\epsilon$.  
\end{proof}

\subsection{Face Enumeration}\label{subsec:face}

Here, we prove formulas counting chains in an oriented interval 
greedoid $\OIG$.  These results generalize results for oriented matroids
\cite[Proposition 4.6.2]{OrientedMatroids1993} and for oriented antimatroids \cite{BilleraHsiaoProvan2008}.  

Let $P$ be a poset.  Recall that the {\it M\"obius function} of $P$, denoted 
$\mu_P$, is the unique function from pairs $(x,y)$ with $x\leq y$ in $P$ to 
$\mathbb Z$,
such that:
\begin{itemize}
\item $\mu_P(x,x)=1$.
\item For $x<y$, $\sum_{x\leq z\leq y} \mu_P(y,z)=0$.  
\end{itemize}

\begin{Theorem}\label{t:CountingFlags}
Let $(E,\F,\OIG)$ be an oriented interval greedoid.  Let
$A_1>\dots>A_{k+1}=\hzero$ be a chain of flats in $\Phi$.  
Then:
$$|\supp^{-1}(A_1,\dots,A_{k+1})|= \prod_{i=1}^k \sum_{B\in [A_{i+1},A_i]}
|\mu_{\Phi}(B,A_i)|,$$
where $\mu_\Phi$ is the M\"obius function fo $\Phi$.
\end{Theorem}

First, we state and prove the following special case, which generalizes
\cite[Theorem 4.6.1]{OrientedMatroids1993}.  

\begin{Proposition}
Let $(E,\F,\OIG)$ be an oriented interval greedoid.  Then the number of 
topes of $\OIG$ is:
$$\sum_{B\in \Phi} |\mu_{\Phi}(B,\hone)|.$$
\end{Proposition}

\begin{proof}
One could adapt the proof for oriented matroids to this setting, thus
reproving the result for oriented matroids, but we prefer to assume the
result if $\OIG$ is an oriented matroid; this is \cite{OrientedMatroids1993}[Theorem 4.6.1].

Recall that the topes of $\OIG$ are the same as those of $\UOM$.  Applying
the proposition to $\UOM$, and writing $\overline\Phi$ for $\Phi(\UOM)$,
we need now only show that:
\begin{equation}\label{SumIsSum}
\sum_{B\in \Phi} |\mu_{\Phi}(B,\hone)| = \sum_{B\in\overline\Phi}
|\mu_{\overline\Phi}(B,\hone)|.
\end{equation}

Consider the order-preserving map $i:\overline \Phi\rightarrow \Phi$ defined
by $i([X])=[X]$, as discussed in \ref{sss:RestrictionofIGs}.  We prove
a few more properties of it here.

\begin{Lemma}\label{l:iProperties}
\begin{enumerate}
\item $i$ is a poset isomorphism onto its image.
\item For $A,B\in \overline\Phi$, we have $i(A\wedge B)=i(A)\wedge i(B)$.  
\end{enumerate}
\end{Lemma}

\begin{proof}
(1) \ref{p:CanonicalSurjectionForFlatsInRestriction} provides a restriction
map from $\Phi$ to $\overline \Phi$ defined by $A|_{\Gamma(\emptyset)}=
\mu|_{\Gamma(\emptyset)}(\xi(A)\cap\Gamma(\emptyset))$, which
is order-preserving.  
Since $i(A)|_{\Gamma(\emptyset)}=A$, we know 
$i$ is a poset isomorphism onto its image.  

(2) Let $A,B\in \overline \Phi$.  
Let $C=i(A)\wedge i(B)$, and let 
$D=i(C|_{\Gamma(\emptyset)})$.  It is immediate that $D\geq C$.  
However, we know $D\leq i(A)$ and $D\leq i(B)$, so
$D=C$.  This implies that $C$ is in the image of $i$, so, by (1),
$C=i(A\wedge B)$.  
\end{proof}

Thanks to \ref{l:iProperties} (1), 
we can identify $\overline \Phi$ as a subposet of $\Phi$.  

Let $x\in E$ such that $\{x\}\in \F$.  Then, by definition, 
$x\in \Gamma(\emptyset)$.  It follows that every coatom of $\Phi$ is in
$\overline \Phi$.  Further, since $\overline\Phi$ is a geometric lattice,
every element of $\overline\Phi$ can be written as a meet 
(in $\overline\Phi$) of coatoms.  Thanks to \ref{l:iProperties} (2), it
follows that $\overline\Phi$ consists exactly of those elements of 
$\Phi$ that can be written as a meet of coatoms in $\Phi$.  

\begin{Lemma} \begin{enumerate} 
\item If $A\in \Phi \setminus \overline\Phi$, then
$\mu_{\Phi}(A,\hone)=0$.
\item If $A\in \overline\Phi$, then $\mu_{\Phi}(A,\hone)=\mu_{\overline\Phi}
(A,\hone)$.
\end{enumerate}
\end{Lemma}

\begin{proof} (1) Since $A\not\in\overline\Phi$, $A$ cannot be expressed
as a meet of coatoms of $\Phi$.  It follows that the meet of the
coatoms of $[A,\hone]$ is strictly greater than $A$.  The Crosscut Theorem
(see \cite{BjornerHandbook}) now implies $\mu_{\Phi}(A,\hone)=0$.  

(2) We induct on the corank of $A$.  The statement is obvious for 
$A=\hone$.  For $A$ of positive corank, 
we use the formula:
$$\mu_{\Phi}(A,\hone)=-\sum_{A<B\in\Phi} \mu_{\Phi}(B,\hone).$$
Now we observe that, by (1), only the terms with $B\in\overline\Phi$
contribute.  By induction, these terms agree with 
$\mu_{\overline\Phi}(B,\hone)$, which proves the result.  
\end{proof}

\ref{SumIsSum} is now obvious, and the proposition follows.  
\end{proof}

\begin{proof}[Proof of \ref{t:CountingFlags}]  
The proof goes exactly as in the oriented matroid case, now that the
preparations have been made. $|\supp^{-1}(A_k)|$ is the number of topes
of $\OIG/A_k$, which is $\sum_{A_k\geq B} |\mu(B,A_k)|$, and then
the rest of the chain lies in $\OIG|_{\xi(A_k)}$, which accounts for the 
remaining terms. 
\end{proof}

\bibliographystyle{amsalpha}
\bibliography{references} 

\end{document}